\documentclass[11pt]{article}

\usepackage[margin=1cm]{geometry}
\usepackage{fullpage}
\usepackage{amsmath}
\usepackage{amsthm}
\usepackage{amssymb}
\usepackage{enumitem}
\usepackage{textcomp}
\usepackage[hidelinks]{hyperref}
\usepackage{cleveref}

\newcommand{\mbb}[1]{\mathbb{#1}}
\newcommand{\mbf}[1]{\mathbf{#1}}
\newcommand{\mc}[1]{\mathcal{#1}}

\newcommand{\tr}{\textup{tr}\,}
\newcommand{\wt}[1]{\widetilde{#1}}

\newcommand{\Tr}[1]{\left\langle{#1}\right\rangle}
\newcommand{\diff}{\,\mathrm{d}}
\renewcommand{\det}{\textup{det}}
\renewcommand{\Re}{\textup{Re}}
\renewcommand{\Im}{\textup{Im}}
\numberwithin{equation}{section}

\newtheorem{theorem}{Theorem}[section]

\newtheorem{lemma}{Lemma}[section]

\newtheorem{proposition}{Proposition}[section]

\theoremstyle{remark}

\theoremstyle{definition}
\newtheorem{definition}{Definition}[section]

\title{Bulk Universality for Sparse Complex non-Hermitian Random Matrices
\date{}
\author{Mohammed Osman\footnote{mohammed.osman@qmul.ac.uk}\\~\\\small Queen Mary University of London}}

\begin{document}

\maketitle

\abstract{We prove that the local eigenvalue statistics in the bulk for complex random matrices with independent entries whose $r$-th absolute moment decays as $N^{-1-(r-2)\epsilon}$ for some $\epsilon>0$ are universal. This includes sparse matrices whose entries are the product of a Bernouilli random variable with mean $N^{-1+\epsilon}$ and an independent complex-valued random variable. By a standard truncation argument, we can also conclude universality for complex random matrices with $4+\epsilon$ moments. The main ingredient is a sparse multi-resolvent local law for products involving any finite number of resolvents of the Hermitisation and deterministic $2N\times 2N$ matrices whose $N\times N$ blocks are multiples of the identity.}

\section{Introduction}
In this paper we are concerned with the local statistics of eigenvalues of non-Hermitian random matrices in the bulk of the spectrum. The model we will study is defined as follows.
\begin{definition}\label{def1}
Let $\epsilon>0$ and $q\geq N^{\epsilon}$. The set $\mc{S}_{N}(\epsilon)$ consists of random matrices $X\in\mbb{C}^{N\times N}$ with independent entries satisfying the following conditions:
\begin{enumerate}[label=(\alph*),ref=\thedefinition (\alph*)]
\item \[\mbb{E}X_{ij}=\mbb{E}X^{2}_{ij}=0;\]
\item \[\mbb{E}|X_{ij}|^{2}=\frac{1}{N};\]
\item\label{def1:3}\[\mbb{E}|X_{ij}|^{r}\leq\frac{C_{r}}{Nq^{r-2}},\qquad r>2.\]
\end{enumerate}
\end{definition}

As a first example, let $\{\xi_{ij}:i,j\in[N]\}$ be i.i.d. Bernouilli random variables with mean $p=N^{-1+\epsilon}$ and $\{x_{ij}:i,j\in[N]\}$ be independent, centred, complex random variables with finite moments such that $\mbb{E}x_{ij}^{2}=0$ and $\mbb{E}|x_{ij}|^{2}=1$. Then the matrix $X_{ij}=N^{-\epsilon/2}\xi_{ij}x_{ij}$ satisfies Definition \ref{def1}. Thus this model includes sparse complex matrices with on average $N^{1+\epsilon}$ nonzero entries. For this model it is known from the work of Tao--Vu \cite{tao_random_2008} and G\"{o}tze--Tikhomirov \cite{gotze_circular_2010} that the empirical distribution of eigenvalues converges almost surely to the circular law. In fact, Rudelson and Tikhomirov \cite{rudelson_sparse_2019} show that convergence in probability holds under the minimal assumption $Np\to\infty$. If instead $Np\to d<\infty$, Sah--Sahasrabudhe--Sawhney \cite{sah_limiting_2023} show convergence to a (non-explicit) deterministic measure $\mu_{d}$.

As a second example, let $X_{ij}=N^{-1/2}x_{ij}$, where $\{x_{ij}:i,j\in[N]\}$ are independent, centred, complex random variables with independent real and imaginary parts such that  $\mbb{E}|x_{ij}|^{2}=1$ and $\mbb{E}|x_{ij}|^{4+\delta}<C_{\delta}$ for some $\delta>0$. By \cite[Lemma 7.6]{erdos_spectral_2012}, for $\lambda=N^{1/2-\epsilon}$ we construct independent, centred, complex random variables $y_{ij}$ such that $|y_{ij}|\leq \lambda$ almost surely and
\begin{align*}
    \mbb{P}(x_{ij}\neq y_{ij})\leq C_{\delta}\lambda^{-4-\delta}.
\end{align*}
Then the matrix $Y_{ij}=N^{-1/2}y_{ij}$ satisfies Definition \ref{def1} and by \cite[Eq. (7.12)]{erdos_spectral_2012} the local statistics of $X$ and $Y$ are the same.

If we replace the third condition \ref{def1:3} with
\begin{align*}
    \mbb{E}|X_{ij}|^{r}&\leq C_{r}N^{-r/2},
\end{align*}
then bulk and edge universality have been proven in \cite{maltsev_bulk_2024} and \cite{cipolloni_edge_2020} respectively. With the slower decay in \ref{def1:3}, Yukun He \cite{he_edge_2023} proves edge universality for real matrices (He also allows for a rank-one mean matrix which is needed for the application to Erd\H{o}s--R\'{e}nyi digraphs). At the edge, universality can be deduced from a single resolvent local law following the method of Cipolloni--Erd\H{o}s--Schr\"{o}der \cite{cipolloni_edge_2020}, for which the difference between real and complex matrices is immaterial. By contrast, in the bulk a single resolvent local law is not sufficient and the method in \cite{maltsev_bulk_2024} required two-resolvent local laws. In the sparse case we consider here, this is still insufficient and we in fact require local laws involving an arbitrary finite number of resolvents of the Hermitisation. A large part of this paper is therefore dedicated to proving such local laws for the model in Definition \ref{def1}. Our approach to multi-resolvent local laws is based on the recent ``Zigzag" method of Cipolloni--Erd\H{o}s--Schr\"{o}der \cite{cipolloni_mesoscopic_2024}, which has proven successful in a variety of applications (see e.g. \cite{bao_law_2025,cipolloni_eigenvector_2025,cipolloni_optimal_2024,cipolloni_universality_2024,erdos_zigzag_2025}). Here we adapt this method to handle matrices with slowly decaying moments by combining it with an iterated cumulant expansion. Iterated cumulant expansions have been a central tool in the study of sparse random matrices (see for example the works \cite{lee_local_2018,huang_transition_2020,he_fluctuations_2021,schnelli_convergence_2022,lee_higher_2024} which are concerned with edge statistics of sparse Hermitian random matrices). In our context, the iterated cumulant expansion solves a closure problem: when attempting to gain from off-diagonal resolvent entries by the Cauchy--Schwarz inequality and the resolvent identity, we end up with entries of higher order products of resolvents. This problem also appears in the study of multi-resolvent local laws for Hermitian matrices \cite{cipolloni_optimal_2022}, where it is solved by a so-called ``reduction inequality". In the sparse case we cannot afford the extra powers of $N$ entailed by the reduction inequality and so resort to an iterated cumulant expansion.

We now come to the statement of our main result. Let
\begin{align}
    \rho^{(k)}_{GinUE}(\mbf{z})&=\det\left[\frac{1}{\pi}e^{-\frac{1}{2}(|z_{i}|^{2}+|z_{j}|^{2})+\bar{z}_{i}z_{j}}\right]_{i,j=1}^{k}
\end{align}
denote the limiting bulk correlation function of the complex Ginibre ensemble (GinUE).
\begin{theorem}\label{thm1}
Let $\epsilon>0$ and $X\in\mc{S}_{N}(\epsilon)$ with eigenvalues $\lambda_{i}(X),\,i=1,...,N$. Then, for any fixed $z\in\mbb{D}$, $k>0$ and $f\in C^{2}_{c}(\mbb{C}^{k})$, there is an $\omega>0$ such that
\begin{align}
    \mbb{E}\sum_{i_{1}\neq\cdots\neq i_{k}}f(N^{1/2}(\lambda_{i_{1}}(X)-z),...,N^{1/2}(\lambda_{i_{k}}(X)-z))&=\int_{\mbb{C}^{k}}f(\mbf{z})\rho^{(k)}_{GinUE}(\mbf{z})\diff\mbf{z}+O(N^{-\omega}),
\end{align}
for sufficiently large $N$.
\end{theorem}

This will follow from an intermediate result concerning the time $t$ for which the Gaussian-divisible matrix $X+\sqrt{t}Y$ (where $Y$ is a GinUE matrix) can be shown to have universal bulk correlation functions. For this we consider a deterministic $X$ satisfying certain conditions. To formulate these conditions, we introduce the Hermitisation
\begin{align}
    W_{z}&:=\begin{pmatrix}0&X_{z}\\X^{*}_{z}&0\end{pmatrix},
\end{align}
where $X_{z}:=X-z$, and we set $W\equiv W_{0}$. Here and in the following we identify scalars with multiples of the identity matrix. Let
\begin{align}
    G_{z}(w)&:=(W_{z}-w)^{-1}=\begin{pmatrix}w\wt{H}_{z}(w)&X_{z}H_{z}(w)\\H_{z}(w)X^{*}_{z}&w H_{z}(w)\end{pmatrix},
\end{align}
where
\begin{align}
    H_{z}(w)&=(|X_{z}|^{2}-w^{2}),\\
    \wt{H}_{z}(w)&=(|X^{*}_{z}|^{2}-w^{2})^{-1}.
\end{align}
Define the block matrix
\begin{align}
    F&:=\begin{pmatrix}0&1_{N}\\0&0\end{pmatrix}\in\mbb{C}^{2N\times2N}.
\end{align}
The conditions on $X$ are as follows, where $\mbb{D}\subset\mbb{C}$ denotes the open unit disk. 
\begin{definition}\label{def2}
Let $z\in\mbb{D}$, $\delta,\tau>0$, $n\in\mbb{N}$ and
\begin{align}
    \mc{D}(\delta,\tau)&:=\{w=E+i\eta\in\mbb{C}:|E|\leq\delta|\eta|,\,N^{-1+\tau}\leq|\eta|\leq10\}.\label{eq:D}
\end{align}
A matrix $X\in\mbb{C}^{N\times N}$ belongs to the set $\mc{M}_{N}(z,n,\tau)$ if there are constants $\delta,c,C>0$ such that:
\begin{enumerate}[label=(\alph*),ref=\thedefinition (\alph*)]
\item \[\|X\|\leq e^{C\log^{2}N};\]
\item \[c\leq\Im\Tr{G_{1}}\leq C;\]
\item \[c\leq\eta_{1}\eta_{2}\Tr{H_{z}(i\eta_{1})\wt{H}_{z}(i\eta_{2})}\leq C;\]
\item \[c\leq\eta_{1}^{3}H^{2}_{z}(i\eta_{1})\leq C;\]
\item\label{def:multiLL} \[|\Tr{G_{1}B_{1}\cdots G_{k}B_{k}}|\leq \frac{C}{\eta^{k/2-1}},\]
\end{enumerate}
uniformly in $w_{j}\in\mc{D}(\delta,\tau),\,B_{j}\in\{F,F^{*}\}$ and $k=2,...,2n$, where $\eta=\min|\Im w_{j}|$ and we used the shorthand $G_{j}:=G_{z}(w_{j})$.
\end{definition}
Note that there is some redudancy in this definition: the upper bound in (c) follows from the $k=2$ instance of (e).

In \cite{maltsev_bulk_2024} it was shown that the correlation functions are universal for $t=N^{-1/3+\epsilon}$ assuming conditions (a)-(d). With the addition of condition (e), we are able to reduce $t$ from $N^{-1/3+\epsilon}$ to $N^{-1+\epsilon}$.
\begin{theorem}\label{thm2}
Fix $\epsilon>0$, $z\in\mbb{D}$ and $k>0$. Set
\begin{align}
    n_{\epsilon,k}&:=\left\lceil\frac{3\cdot 2^{k}}{\epsilon}+2(2^{k}-1)\right\rceil+1.\label{eq:n_eps}
\end{align}
Let $X\in\mc{M}_{N}(z,n_{\epsilon,k},\epsilon)$ be deterministic and $Y$ be a GinUE(N) matrix. For $t>0$, let $\lambda_{i},\,i\in[N]$ denote the eigenvalues of $X+\sqrt{t}Y$. Then for any $t\geq N^{-1+2\epsilon}$ there is an $\eta_{z}\simeq t$ such that
\begin{align}
    t\Tr{H_{z}(i\eta_{z})}&=1.
\end{align}
Let
\begin{align}
    \sigma_{z}&:=\eta_{z}^{2}\Tr{H_{z}(i\eta_{z})\wt{H}_{z}(i\eta_{z})}+\frac{|\Tr{H_{z}^{2}(i\eta_{z})X_{z}}|^{2}}{\Tr{H_{z}^{2}(i\eta_{z})}}.\label{eq:sigma}
\end{align}
Then for any bounded and measurable $f:\mbb{C}^{k}\to\mbb{C}$ we have
\begin{align}
    &\mbb{E}\sum_{i_{1}\neq\cdots\neq i_{k}}f(N^{1/2}\sigma^{1/2}_{z}(\lambda_{i_{1}}-z),...,N^{1/2}\sigma^{1/2}_{z}(\lambda_{i_{k}}-z))\nonumber\\&=\int_{\mbb{C}^{k}}f(\mbf{z})\rho^{(k)}_{GinUE}(\mbf{z})\diff\mbf{z}+O\left(\frac{\log N}{\sqrt{Nt}}\right).
\end{align}
\end{theorem}

It is perhaps instructive to compare this result with the corresponding result for Hermitian matrices (i.e. $W+\sqrt{t}V$ where $W$ is a deterministic Hermitian matrix and $V$ is a Gaussian Hermitian matrix). Landon--Sosoe--Yau show in \cite[Theorem 2.2]{landon_fixed_2019} that local universality holds for $t>N^{-1+\epsilon}$ essentially under the sole condition
\begin{align}
    c&\leq\frac{1}{N}\Im\tr(W-E-i\eta)^{-1}\leq C\label{eq:hermitianCondition}
\end{align}
for $|E|<\delta$ and $\eta\geq N^{-1+\epsilon}$. This is analogous to condition (b) in Definition \ref{def2}. Conditions (c) and (e) are new to the non-Hermitian case. If $X$ is Hermitian, then all traces of the form $\tr G_{1}B_{1}\cdots G_{m}B_{m}$ for $B_{i}\in\{F,F^{*}\}$ can be reduced to traces of the form $\tr G_{1}\cdots G_{m}$, which can in turn be reduced to the trace of a single resolvent by contour integration and the resolvent identity. Thus conditions (b)-(e) collapse to a single condition on the trace of one resolvent, which suggests that these conditions are natural for general matrices and that one should perhaps not expect Theorem \ref{thm2} to hold under significantly weaker conditions.

The condition \eqref{eq:hermitianCondition} in the Hermitian case has a simple interpretation: the lower and upper bounds ensure that there are enough but not too many eigenvalues in a neighbourhood of $E$. By contrast, it is not obvious how to interpret \ref{def:multiLL}. One possible interpretation is obtained by observing that, since the eigenvalue spacing is $O(N^{-1/2})$ in the complex plane, we should have some uniform control of $\log\det(W_{w}-i\eta)$ for $|w-z|<CN^{-1/2}$. Writing $x=N^{1/2}(w-z)$ and extracting a factor of $\det(W_{z}-i\eta)$, we find
\begin{align}
    \log\frac{\det(W_{w}-i\eta)}{\det(W_{z}-i\eta)}&=\tr\log\left(1-\frac{1}{\sqrt{N}}G_{z}(i\eta)(xF+\bar{x}F^{*})\right).\label{eq:logdet}
\end{align}
From \ref{def:multiLL}, if $\eta>N^{-1+\epsilon}$ we can deduce that
\begin{align*}
    N^{-n/2}\|(G_{z}(i\eta)F)^{n}\|&\leq\frac{CN^{1/2}}{(N\eta)^{n/2}}\leq CN^{\frac{1-n\epsilon}{2}}.
\end{align*}
If $n\epsilon>3$ we can truncate the Taylor series of the logarithm in the right hand side of \eqref{eq:logdet} at the $n$-th term and estimate each previous term by \ref{def:multiLL}. Thus this condition allows us to approximate the function
\begin{align*}
    x&\mapsto\log\det(W_{z+N^{-1/2}x}-i\eta)
\end{align*}
when $\eta>N^{-1+\epsilon}$. The relevance of this function can be seen from Girko's formula,
\begin{align*}
    \sum_{n}f(N^{1/2}(\lambda_{n}(X)-z))&=\frac{1}{4N\pi}\int_{\mbb{C}}\Delta f(x)\log\det W_{z+N^{-1/2}x}\diff^{2}x.
\end{align*}
In practice, we do not use Girko's formula to prove Theorem \ref{thm2}, relying instead on the method of partial Schur decomposition.

Girko's formula is, however, used to prove Theorem \ref{thm1}: based on the standard moment-matching method, Theorem \ref{thm1} will follow from Theorem \ref{thm2} if we can show that matrices in $\mc{S}_{N}(\epsilon)$ are also in $\mc{M}_{N}(z,n,\xi)$ with very high probability. This amounts to proving a sparse multi-resolvent local law. We define the block matrices
\begin{align}
    E_{\pm}&=\begin{pmatrix}1_{N}&0\\0&\pm1_{N}\end{pmatrix}\in\mbb{C}^{2N\times2N},\label{eq:Epm}
\end{align}
and let
\begin{align}
    \mbb{H}&:=\text{span}\{E_{+},E_{-},F,F^{*}\},
\end{align}
denote the subspace of $2N\times 2N$ matrices whose $N\times N$ blocks are multiples of $1_{N}$ (such matrices can be thought of as quaternions, hence the notation $\mbb{H}$). The following multi-resolvent law concerns test matrices in $\mbb{H}$. In its statement the matrices $M_{z}(w_{1},B_{1},...,w_{m})$ represent deterministic approximations to resolvent chains which have an explicit definition (see \cite[Definition 4.1]{cipolloni_optimal_2024})
\begin{theorem}[Theorem \ref{thm:multiLL} below]\label{thm3}
Let $\epsilon>0$ and $X\in\mc{S}_{N}(\epsilon)$. Let $m\in\mbb{N}$, $B_{i}\in\mbb{H}$ for $i=1,...,m$ and
\begin{align}
    a&:=\bigl|\{i:B_{i}\in\{F,F^{*}\}\}\bigr|.
\end{align}
Fix $\delta,\tau>0$ and define
\begin{align}
    \mc{D}(\delta,\tau)&:=\bigl\{w=E+i\eta:|E|\leq\delta|\eta|,\,N^{-1+\tau}\leq|\eta|\leq10\bigr\},
\end{align}
Then for any fixed $z\in\mbb{D}$ and sufficiently small $\delta>0$
\begin{align}
    \bigl|\Tr{G_{1}B_{1}\cdots G_{m}B_{m}}-\Tr{M_{z}(w_{1},B_{1},...,w_{m})B_{m}}\bigr|&\prec\left(\frac{1}{N\eta}+\frac{1}{q}\right)\frac{1}{\eta^{m-a/2-1}\wedge1},\\
    \max_{\mathfrak{x},\mathfrak{y}\in[2N]}\bigl|\bigl(G_{1}B_{1}\cdots G_{m}-M_{z}(w_{1},B_{1},...,w_{m})\bigr)_{\mathfrak{x},\mathfrak{y}}\bigr|&\prec\left(\frac{1}{\sqrt{N\eta}}+\frac{1}{q}\right)\frac{1}{\eta^{m-a/2}},
\end{align}
uniformly in $w_{j}\in\mc{D}(\delta,\tau)$, where $G_{j}:=G_{z}(w_{j})$ and $\eta=\min_{j}|\Im w_{j}|$.
\end{theorem}
Each factor of $F$ or $F^{*}$ reduces the size of the deterministic approximation (see Lemma \ref{lem:deterministicBounds} below) and the error by a factor of $\sqrt{\eta}$. In the language of \cite[Definitions 3.1 and 4.2]{cipolloni_optimal_2024}, $F$ and $F^{*}$ are \textit{regular} matrices (up to terms of $O(\eta)$). They play a role analogous to that of traceless matrices in the multi-resolvent law for Wigner matrices \cite{cipolloni_optimal_2022}.

We remark that with this local law and the arguments in the proofs of Theorems \ref{thm1} and \ref{thm2}, it is straightforward to extend the results on bulk eigenvectors in \cite{dubova_gaussian_2024,osman_least_2024,osman_universality_2024} to matrices satisfying Definition \ref{def1}.

The rest of this paper is organised as follows. In Sections \ref{sec2} and \ref{sec3} we prove Theorems \ref{thm2} and \ref{thm3} respectively. In Section \ref{sec4} we give a brief sketch of the (standard) proof of Theorem \ref{thm1}, which combines Theorems \ref{thm2} and \ref{thm3} with a moment-matching argument. The Appendix contains some deferred proofs.

\paragraph{Notation} We use $[n]$ to denote the set $\{1,2,...,n\}$. The open complex unit disk is denoted by $\mbb{D}$. For a matrix $A\in\mbb{C}^{d\times d}$, $\Tr{A}:=d^{-1}\tr A$ denotes the normalised trace. For a probability distribution $\mc{D}$, $X\sim\mc{D}$ means that $X$ has distribution $\mc{D}$. Throughout the paper, $N$ will be the fundamental large parameter taken to infinity. Thus, ``constant" or ``fixed" will mean independent of $N$. The notation $x\lesssim y$ means that there is a constant $C>0$ such that $x\leq Cy$, and $x\simeq y$ means that $x\lesssim y$ and $y\lesssim x$. We also use the notation of stochastic domination (see e.g. \cite[Definition 2.1]{erdos_local_2013}): $X\prec Y$ or $X=O_{\prec}(Y)$ if for any $D>0$ and $\xi>0$ we have
\begin{align*}
    |X|\leq N^{\xi}|Y|
\end{align*}
with probability at least $1-N^{-D}$ for sufficiently large $N$. Statements such as this that hold with probability at least $1-N^{-D}$ for any $D>0$ will be referred to as holding ``very high probability".

\section{Proof of Theorem \ref{thm2}}\label{sec2}
In this section we prove that, if $X$ satisfies certain multi-resolvent local laws, then the bulk correlation functions of the Gaussian-divisible matrix $X+\sqrt{t}Y$ are universal for $t=N^{-1+\epsilon}$. This should be compared with the earlier result in \cite{maltsev_bulk_2024} which required $t=N^{-1/3+\epsilon}$ assuming only two-resolvent laws. It turns out that to reach $t=N^{-1+\epsilon}$ we need local laws involving a large but finite number (depending on $\epsilon$) of resolvents.

We begin by recalling the partial Schur decomposition (see \cite[Section 4]{maltsev_bulk_2024}). Let $t>0$, $X\equiv X^{(0)}\in\mbb{C}^{N\times N}$ and $\mbf{z}\in\mbb{C}^{k}$. For $i=1,...,k$, we define a set of probability measures $\mu_{i}$ on $S^{N-i}$ through the following recursive procedure which takes a matrix $X^{(i-1)}\in\mbb{C}^{(N-i+1)\times(N-i+1)}$ as input. Let $\mu_{i}$ have density
\begin{align}
    \frac{\diff\mu_{i}}{\diff_{H}\mbf{v}}&=\frac{1}{K_{i}(z_{i})}\left(\frac{N}{\pi t}\right)^{N-i}e^{-\frac{N}{t}\|X^{(i-1)}_{z_{i}}\mbf{v}\|^{2}},\label{eq:mu(i)}
\end{align}
with respect to the Haar measure $\diff_{H}\mbf{v}$, where $K_{i}(z_{i})$ is the normalisation. We denote by $\mbb{E}^{(i)}$ the expectation with respect to $\mu_{i}$. Let $\mbf{v}_{i}\sim\mu_{i}$ and $R_{i}(\mbf{v}_{i})=(\mbf{v}_{i},Q_{i})$ be the Householder matrix whose first column is $\mbf{v}_{i}$. Define
\begin{align}
    X^{(i)}&:=Q_{i}^{*}X^{(i-1)}Q_{i},\label{eq:X(i)}
\end{align}
to be the projection of $X^{(i-1)}$ onto the space complementary to $\mbf{v}_{i}$. We can now define $\mu^{(i+1)}$ in the same way with $X^{(i)}$ as input. Note that the normalisation $K_{i}(z_{i})$ is a random variable depending on $\{\mbf{v}_{j}:j<i\}$.

Define
\begin{align}
    F_{N-k}(\mbf{z},X^{(k)})&:=\mbb{E}_{N-k}\prod_{j=1}^{k}\Bigl|\det\Bigl(X^{(k)}_{z_{i}}+\sqrt{\frac{Nt}{N-k}}Y_{N-k}\Bigr)\Bigr|^{2},\label{eq:F}
\end{align}
where the expectation is with respect to $Y_{N-k}\sim GinUE(N-k)$. Then we have the following formula for the $k$-point functions of $X+\sqrt{t}Y$ \cite[Corollary 4.2]{maltsev_bulk_2024}:
\begin{align}
    \rho^{(k)}_{N}(\mbf{z})&=\left(\frac{N}{\pi t}\right)^{k}|\Delta(\mbf{z})|^{2}K_{1}(z_{1})\mbb{E}^{(1)}\bigl[\cdots K_{k}(z_{k})\mbb{E}^{(k)}\bigl[F_{N-k}(\mbf{z},X^{(k)})\bigr]\cdots\bigr],\label{eq:rho(k)}
\end{align}
where
\begin{align}
    \Delta(\mbf{z})&=\prod_{i<j}(z_{i}-z_{j}).
\end{align}

For $z\in\mbb{D}$, $X^{(j)}\in\mc{M}_{N-j}(z,n,\epsilon)$ and $t\geq N^{-1+2\epsilon}$, define $\eta^{(j)}_{z}\equiv\eta^{(j)}_{z}(t)$ to be the unique positive solution of
\begin{align}
    t\Tr{H^{(j)}_{z}(i\eta^{(j)}_{z})}&=1.\label{eq:eta(j)}
\end{align}
That such a solution exists and satisfies
\begin{align}
    \eta^{(j-1)}_{z}&\simeq t,\\
    |\eta^{(j)}_{z}-\eta^{(j-1)}_{z}|&\lesssim\frac{1}{N},
\end{align}
has been shown in \cite[Eq. (5.10) and (5.11)]{maltsev_bulk_2024}. These properties do not rely on condition \ref{def:multiLL} and so the value of $n$ is irrelevant. In the rest of this section we use the shorthand
\begin{align}
    G^{(j)}_{z}&\equiv G^{(j)}_{z}(i\eta^{(j)}_{z}),\quad H^{(j)}_{z}\equiv H^{(j)}_{z}(i\eta^{(j)}_{z}),\quad \wt{H}^{(j)}_{z}\equiv\wt{H}^{(j)}_{z}(i\eta^{(j)}_{z}),
\end{align}
i.e. we suppress the argument of a resolvent when it is equal to $i\eta^{(j)}_{z}$.

The proof of Theorem \ref{thm2} proceeds by showing that bounds on traces of resolvents carry over from $X^{(i-1)}$ to $X^{(i)}$ with very high probability (with respect to $\mu_{i}$). In other words, the event that $X^{(i)}\in\mc{M}_{N-i}(z,n',\epsilon)$ assuming $X^{(i-1)}\in\mc{M}_{N-i+1}(z,n,\epsilon)$ holds with very high probability, where $n'$ is explicitly determined by $n$. We obtain uniform bounds on $K_{i}$ and $F_{N-k}$ which allow us to restrict to these events. We can then Taylor expand certain determinants that appear in the asymptotic analysis of the integrals in the formula \eqref{eq:rho(k)} for $\rho^{(k)}_{N}$. 

In the remainder of this section we assume $t\geq N^{-1+2\epsilon}$. The first step is to prove concentration of certain quadratic forms in $\mbf{v}\sim\mu^{(1)}$ with improved error (compare with \cite[Lemma 7.2]{maltsev_bulk_2024}).
\begin{lemma}\label{lem:conc}
Let $X\in\mc{M}_{N}(z,2,\epsilon)$, $\mbf{v}\sim\mu^{(1)}$ and $\eta\simeq t$. There are constants $C,c>0$ such that
\begin{align}
    t\left|\eta\mbf{v}^{*}H_{z}(i\eta)\mbf{v}-\eta t\Tr{H_{z}H_{z}(i\eta)}\right|&\leq\frac{C\log N}{\sqrt{Nt}},\label{eq:conc1}\\
    \left|\eta\mbf{v}^{*}\wt{H}_{z}(i\eta)\mbf{v}-\eta t\Tr{H_{z}\wt{H}_{z}(i\eta)}\right|&\leq\frac{C\log N}{\sqrt{Nt}},\label{eq:conc2}\\
    \sqrt{t}\left|\mbf{v}^{*}X_{z}H_{z}(i\eta)\mbf{v}-t\Tr{H_{z}X_{z}H_{z}(i\eta)}\right|&\leq\frac{C\log N}{\sqrt{Nt}},\label{eq:conc3}
\end{align}
with probability $1-e^{-c\log^{2}N}$.
\end{lemma}
Note that in \eqref{eq:conc3} the error is larger than the expectation when $t\ll N^{-1/2}$, so strictly speaking this quadratic form does not concentrate. As we will see later, what matters is an upper bound on $\mbf{v}^{*}X_{z}H_{z}\mbf{v}$.
\begin{proof}
Let $B=B^{*}\in\mbb{C}^{N\times N}$ and $\lambda\in\mbb{R}$ such that
\begin{align}
    \frac{t|\lambda|}{N}\|H_{z}^{1/2}BH_{z}^{1/2}\|<1.
\end{align}
Then we have the following bound on the moment generating function of $\mbf{v}^{*}B\mbf{v}$ (see \cite[Eq. (7.12)]{maltsev_bulk_2024}):
\begin{align*}
    m_{B}(\lambda)&:=\mbb{E}\left[e^{\lambda \mbf{v}^{*}B\mbf{v}}\right]\lesssim\exp\left\{\lambda t\Tr{H_{z}B}+\frac{\lambda^{2}t^{2}\Tr{(H_{z}B)^{2}}}{N\left(1-\frac{|\lambda|t}{N}\left\|\sqrt{H_{z}}B\sqrt{H}_{z}\right\|\right)}\right\}.
\end{align*}
Using the bound $\|A\|\leq(\tr|A|^{2})^{1/2}$ we can obtain a simpler bound
\begin{align}
    m_{B}(\lambda)&\lesssim\exp\left\{\lambda t\Tr{H_{z}B}+\frac{\lambda^{2}t^{2}\Tr{(H_{z}B)^{2}}}{N\left(1-\frac{|\lambda|t}{N^{1/2}}\Tr{(H_{z}B)^{2}}^{1/2}\right)}\right\}.\label{eq:MGFBound2}
\end{align}
The bound in \eqref{eq:conc1} follows directly from
\begin{align*}
    \Tr{(H_{z}H_{z}(i\eta))^{2}}&\lesssim\frac{1}{t^{7}}.
\end{align*}

For \eqref{eq:conc2}, we choose $B=\eta\wt{H}_{z}(i\eta)$. By \ref{def:multiLL} with $k=4$ we have
\begin{align*}
    \Tr{(H_{z}\wt{H}_{z}(i\eta))^{2}}&\lesssim\frac{1}{t^{5}},
\end{align*}
and so
\begin{align*}
    m_{F}(\lambda)&\leq C\exp\left\{\frac{C\lambda^{2}}{Nt\left(1-\frac{C|\lambda|}{\sqrt{Nt}}\right)}\right\},\quad|\lambda|<\sqrt{Nt}/C.
\end{align*}
The bound \eqref{eq:conc2} now follows by Markov's inequality with the choice $\lambda=c(Nt)^{1/2}$ for some small $c$.

For \eqref{eq:conc3}, we choose $B=\sqrt{t}\Re(X_{z}H_{z})$ and note that
\begin{align*}
    \Tr{(H_{z}\Re(X_{z}H_{z}(i\eta)))^{2}}&\leq\Tr{H_{z}X_{z}H_{z}(i\eta)H_{z}H_{z}(i\eta)X_{z}^{*}}\\
    &=\Tr{H_{z}\wt{H}_{z}(i\eta)^{2}(1-\eta_{z}^{2}\wt{H}_{z})}\\
    &\leq\frac{1}{\eta^{2}}\Tr{H_{z}\wt{H}_{z}(i\eta)}\\
    &\lesssim\frac{1}{t^{4}}.
\end{align*}
Thus we again find
\begin{align*}
    m_{F}(\lambda)&\leq C\exp\left\{\frac{C\lambda^{2}}{Nt\left(1-\frac{C|\lambda|}{\sqrt{Nt}}\right)}\right\},\quad|\lambda|<\sqrt{Nt}/C,
\end{align*}
and complete the proof by Markov's inequality.
\end{proof}
Note that the first and third bounds only require two-resolvent laws, so the four-resolvent law is needed for the second bound. Compared with \cite[Lemma 7.2]{maltsev_bulk_2024}, we have gained a factor of $t$.

From this lemma we deduce that, for $\mbf{v}\sim\mu^{(1)}$,
\begin{align}
    \left|\det(1_{2}\otimes\mbf{v}^{*})G_{z}(1_{2}\otimes\mbf{v})\right|&=\eta_{z}^{2}(\mbf{v}^{*}H_{z}\mbf{v})(\mbf{v}^{*}\wt{H}_{z}\mbf{v})+|\mbf{v}^{*}X_{z}H_{z}\mbf{v}|^{2}\nonumber\\
    &=\left[1+O\left(\frac{\log N}{\sqrt{Nt}}\right)\right]t^{2}\eta_{z}^{2}\Tr{H_{z}^{2}}\Tr{H_{z}\wt{H}_{z}}+O\left(\frac{\log^{2}N}{Nt^{2}}\right)\nonumber\\
    &=\left[1+O\left(\frac{\log N}{\sqrt{Nt}}\right)\right]t^{2}\eta_{z}^{2}\Tr{H^{2}_{z}}\Tr{H_{z}\wt{H}_{z}},\label{eq:detApprox1}
\end{align}
with probability $1-e^{-c\log^{2}N}$, where the last line follows from $\eta_{z}^{2}\Tr{H_{z}\wt{H}_{z}}\geq c>0$ and $\eta_{z}^{3}\Tr{H^{2}_{z}}\geq c>0$. This should be compared with the displayed equation immediately preceeding \cite[Lemma 7.3]{maltsev_bulk_2024}, where the error is $(Nt^{2})^{-1}\log N$.

Let $V=1_{2}\otimes\mbf{v}\in\mbb{C}^{2N\times2N}$. We also obtain an entry-wise bound for $(V^{*}G_{z}V)^{-1}$:
\begin{align}
    (V^{*}G_{z}V)^{-1}&=\frac{1}{\eta^{2}_{z}(\mbf{v}^{*}H_{z}\mbf{v})(\mbf{v}^{*}\wt{H}_{z}\mbf{v})+|\mbf{v}^{*}X_{z}H_{z}\mbf{v}|^{2}}\begin{pmatrix}i\eta_{z}\mbf{v}^{*}H_{z}\mbf{v}&\mbf{v}^{*}X_{z}H_{z}\mbf{v}\\\mbf{v}^{*}H_{z}X_{z}^{*}\mbf{v}&i\eta_{z}\mbf{v}^{*}\wt{H}_{z}\mbf{v}\end{pmatrix}\nonumber\\
    &\lesssim\begin{pmatrix}1&N^{-1/2}\vee t\\N^{-1/2}\vee t&t\end{pmatrix},\label{eq:entrywiseBound1}
\end{align}
with probability $1-e^{-c\log^{2}N}$, where the inequality is to be understood entrywise.

Our next step is to show that the multi-resolvent bounds in \ref{def:multiLL} can be transferred from $X^{(i-1)}$ to $X^{(i)}$.
\begin{lemma}\label{lem:traceApproximation}
Let $i\geq1$ and $X^{(i-1)}\in\mc{M}_{N-i+1}(z,n,\epsilon)$. Then there is a constant $c>0$ such that
\begin{align}
    \mu_{i}\bigl(X^{(i)}\in\mc{M}_{N-i}(z,n/2-1,\epsilon)\bigr)&\geq1-e^{-c\log^{2}N}.\label{eq:traceApproximation}
\end{align}
\end{lemma}

Without loss we fix $i=1$. To prepare for the proof of Lemma \ref{lem:traceApproximation}, we note the following identity:
\begin{align}
    R(\mbf{v})\begin{pmatrix}0&0\\0&G^{(1)}_{z}\end{pmatrix}R(\mbf{v})&=G_{z}-G_{z}V(V^{*}G_{z}V)^{-1}V^{*}G_{z},\label{eq:minor}
\end{align}
where $R(\mbf{v})$ is the Householder matrix with first column $\mbf{v}$. If we replace each $G^{(1)}_{z}$ with the right hand side of \eqref{eq:minor}, then we have
\begin{align}
    \Tr{(G^{(1)}_{z}F^{(1)})^{m}}&=\Tr{(G_{z}F)^{m}}+\text{error},
\end{align}
where the error consists of terms of the form
\begin{align}
    \Tr{(V^{*}G_{z}V)^{-1}V^{*}(G_{z}F)^{p_{1}}GV\cdots(V^{*}G_{z}V)^{-1}V^{*}(G_{z}F)^{p_{l}}GV},
\end{align}
for $l=1,...,m$, $p_{j}\in\mbb{N}_{0}$ and $p_{1}+\cdots+p_{l}=m$. We therefore need bounds on the quadratic forms $\mbf{v}^{*}B\mbf{v}$ where $B$ is an $N\times N$ block of $(G_{z}F)^{p}G$.
\begin{lemma}\label{lem:entrywiseBound2}
Let $X\in\mc{M}_{N}(z,n,\epsilon)$, $p\in[n-2]$, $F_{j}\in\{F,F^{*}\}$ and $\eta_{j}\simeq t$ for $j\in[p]$. Then there is a constant $c>0$ such that
\begin{align}
    V^{*}G_{1}F_{1}\cdots G_{p}F_{p}G_{p+1}V&\lesssim\begin{pmatrix}t^{-p/2}&t^{-(p+1)/2}\\t^{-(p+1)/2}&t^{-p/2-1}\end{pmatrix},\label{eq:entrywiseBound2}
\end{align}
with probability $1-e^{-c\log^{2}N}$, where $G_{j}\equiv G_{z}(i\eta_{j})$ and the inequality holds entrywise.
\end{lemma}
\begin{proof}
For ease of notation we assume that $F_{j}=F,\,j=1,...,p$; the general case follows in exactly the same way. Observe that for a block matrix
\begin{align*}
    B&=\begin{pmatrix}B_{11}&B_{12}\\B_{21}&B_{22}\end{pmatrix},
\end{align*}
with $B_{ij}\in\mbb{C}^{N\times N}$, we have
\begin{align*}
    t\Tr{B_{11}H_{z}}&=\frac{2t}{\eta_{z}}\Tr{BF\Im G_{z}F^{*}},\\
    t\Tr{B_{12}H_{z}}&=\frac{2t}{\eta_{z}}\Tr{B\Im G_{z}F^{*}},\\
    t\Tr{B_{21}H_{z}}&=\frac{2t}{\eta_{z}}\Tr{BF\Im G_{z}},\\
    t\Tr{B_{22}H_{z}}&=\frac{t}{2\eta_{z}}\Tr{B(E_{+}-E_{-})\Im G_{z}(E_{+}-E_{-})}.
\end{align*}
Using the identity
\begin{align*}
    E_{-}G_{z}(w)&=-E_{-}G_{z}(-w)
\end{align*}
and a combination of contour integration and the resolvent identity, when $B=(G_{z}F)^{p}G$ we can reduce all traces to the form in \ref{def:multiLL}. In doing so, each contour integral brings a factor of $\eta^{-1}$. Thus we see that the expectations of the quadratic forms are approximately
\begin{align*}
    \frac{t}{\eta_{z}}\Tr{(G_{z}F)^{p}G_{z}F\Im G_{z}F^{*}}&=O\left(\frac{1}{t^{p/2}}\right),\\
    \frac{t}{\eta_{z}}\Tr{(G_{z}F)^{p}G_{z}\Im G_{z}F^{*}}&=O\left(\frac{1}{t^{(p+1)/2}}\right),\\
    \frac{t}{\eta_{z}}\Tr{(G_{z}F)^{p}G_{z}F\Im G_{z}}&=O\left(\frac{1}{t^{(p+1)/2}}\right),\\
    \frac{t}{\eta_{z}}\Tr{(G_{z}F)^{p}G_{z}(E_{+}-E_{-})\Im G_{z}(E_{+}-E_{-})}&=O\left(\frac{1}{t^{p/2+1}}\right).
\end{align*}
It remains to show that the quadratic forms concentrate, but this follows by bounding the moment generating function as done in Lemma \ref{lem:conc} using the identities
\begin{align*}
    t^{2}\Tr{B_{11}H_{z}B_{11}^{*}H_{z}}&=\frac{2t^{2}}{\eta_{z}^{2}}\Tr{B F\Im G_{z}F^{*}B^{*}F\Im G_{z}F^{*}},\\
    t^{2}\Tr{B_{12}H_{z}B_{12}^{*}H_{z}}&=\frac{t^{2}}{\eta_{z}^{2}}\Tr{B\Im G_{z}(E_{+}-E_{-})B^{*}F\Im G_{z}F^{*}},\\
    t^{2}\Tr{B_{21}H_{z}B_{21}^{*}H_{z}}&=\frac{t^{2}}{\eta_{z}^{2}}\Tr{B^{*}\Im G_{z}(E_{+}-E_{-})BF\Im G_{z}F^{*}},\\
    t^{2}\Tr{B_{22}H_{z}B_{22}^{*}H_{z}}&=\frac{t^{2}}{2\eta_{z}^{2}}\Tr{B(E_{+}-E_{-})\Im G_{z}B^{*}(E_{+}-E_{-})\Im G_{z}}.
\end{align*}
With $B=(GF)^{p}G$, using \ref{def:multiLL} we find
\begin{align*}
    t^{2}\Tr{B_{11}H_{z}B_{11}^{*}H_{z}}&\lesssim\frac{1}{t^{p+1}},\\
    t^{2}\Tr{B_{12}H_{z}B_{12}^{*}H_{z}}&\lesssim\frac{1}{t^{p+2}},\\
    t^{2}\Tr{B_{21}H_{z}B_{21}^{*}H_{z}}&\lesssim\frac{1}{t^{p+2}},\\
    t^{2}\Tr{B_{22}H_{z}B_{22}^{*}H_{z}}&\lesssim\frac{1}{t^{p+3}}.
\end{align*}
We now proceed using the bound on the moment generating function in \eqref{eq:MGFBound2} and Markov's inequality. In this step we must approximate traces involving at most $2p+4\leq 2n$ factors of $F$.
\end{proof}

We can now prove Lemma \ref{lem:traceApproximation}.
\begin{proof}[Proof of Lemma \ref{lem:traceApproximation}]
Let 
\begin{align*}
    F^{(1)}&=\begin{pmatrix}0&1_{N-1}\\0&0\end{pmatrix},
\end{align*}
and $F^{(1)}_{j}\in\{F,F^{*}\}$. Using the entrywise bounds in \eqref{eq:entrywiseBound1} and \eqref{eq:entrywiseBound2}, for $p\in[n-2]$ we have
\begin{align*}
    &N\bigl|\Tr{G^{(1)}_{1}F^{(1)}_{1}\cdots G^{(1)}_{p}F^{(1)}_{p}}-\Tr{G_{1}F_{1}\cdots G_{p}F_{p}}\bigr|\\
    &\lesssim\sum_{l=1}^{p}\sum_{p_{1}+\cdots+p_{l}=p}\tr\prod_{j=1}^{l}\begin{pmatrix}1&N^{-1/2}\vee t\\N^{-1/2}\vee t&t\end{pmatrix}\begin{pmatrix}t^{-p_{j}/2}&t^{-(p_{j}+1)/2}\\t^{-(p_{j}+1)/2}&t^{-p_{j}/2-1}\end{pmatrix}\\
    &\lesssim\sum_{l=1}^{p}\sum_{p_{1}+\cdots+p_{l}=p}\tr\begin{pmatrix}t^{-(p_{1}+\cdots+p_{l})/2}&t^{-(p_{1}+\cdots+p_{l}+1)/2}\\t^{-(p_{1}+\cdots+p_{l}-1)/2}&t^{-(p_{1}+\cdots+p_{l})/2}\end{pmatrix}\\
    &\lesssim \frac{1}{t^{p/2}},
\end{align*}
with probability $1-e^{-c\log^{2}N}$. The second inequality follows from the equality
\begin{align*}
    t^{-p/2}+(N^{-1/2}\vee t)t^{-(p+1)/2}&=t^{-p/2}\left(1+\frac{1}{\sqrt{Nt}}\vee \sqrt{t}\right).
\end{align*}
Note that the factor of $N$ comes from the normalised trace. We therefore have
\begin{align*}
    \Tr{G^{(1)}_{1}F^{(1)}_{1}\cdots G^{(1)}_{p}F^{(1)}_{p}}&=\Tr{G_{1}F_{1}\cdots G_{p}F_{p}}+O\left(\frac{1}{Nt^{p/2}}\right),
\end{align*}
for $p\in[n-2]$ and $\Re w_{j}=0,\,|\Im w_{j}|\simeq t$. The arguments in the proofs of Lemmas \ref{lem:conc} and \ref{lem:entrywiseBound2} can be easily extended to $|\Im w_{j}|\geq N^{-1+\epsilon}$ and $|\Re w_{j}|\leq\delta|\Im w_{j}|$ for some small $\delta>0$, from which it follows that $X^{(1)}\in\mc{M}_{N}(z,n/2-1,\epsilon)$.
\end{proof}

In the second step, we use \ref{def:multiLL} to approximate the following determinant:
\begin{align}
    D&:=\det\left[1-\frac{1}{\sqrt{N}}\begin{pmatrix}G_{z}(i\eta_{1})F_{11}&\cdots&G_{z}(i\eta_{1})F_{1k}\\\vdots&\ddots&\vdots\\G_{z}(i\eta_{k})F_{k1}&\cdots &G_{z}(i\eta_{k})F_{kk}\end{pmatrix}\right],
\end{align}
where $F_{jl}\in\{F,F^{*}\}$.
\begin{lemma}\label{lem:detApprox2}
Fix $z\in\mbb{D}$, $n\in\mbb{N}$ and $\epsilon>0$ such that $n\epsilon>3$ and let $X\in\mc{M}_{N}(z,n,\epsilon)$. Then for any choice of $F_{jl}\in\{F,F^{*}\},\,j,l\in[k]$, we have
\begin{align}
    D&=\left[1+O(N^{(3-n\epsilon)/2})\right]\exp\left\{-\sqrt{N}\sum_{j=1}^{k}\Tr{G_{z}(i\eta_{j})F_{jj}}-\frac{1}{2}\sum_{j,l=1}^{k}\Tr{G_{z}(i\eta_{j})F_{jl}G_{z}(i\eta_{l})F_{lj}}\right\}.
\end{align}
\end{lemma}
\begin{proof}
First, observe that by \ref{def:multiLL} we have
\begin{align*}
    \|(G_{z}F)^{n}\|^{2}&=\|(G_{z}F)^{n}(F^{*}G^{*}_{z})^{n}\|\\
    &\leq\tr G_{z}F\cdots G_{z}F\cdot F^{*}G^{*}_{z}\cdots F^{*}G_{z}^{*}\\
    &\leq\tr\frac{\Im G_{z}}{\eta}F(G_{z}F)^{n-2}\frac{\Im G_{z}}{\eta}F^{*}(G^{*}_{z}F^{*})^{n-2}\\
    &\lesssim\frac{N}{\eta^{n}},
\end{align*}
and so
\begin{align*}
    N^{-n/2}\|(G_{z}F)^{n}\|&\lesssim\frac{N^{1/2}}{(N\eta)^{n/2}}\lesssim N^{\frac{1-\epsilon n}{2}}.
\end{align*}
If $n\epsilon>1$ then $1-N^{-1/2}G_{z}F$ is invertible and equal to
\begin{align*}
    (1-N^{-1/2}G_{z}F)^{-1}&=\sum_{m=0}^{n-1}N^{-m/2}(G_{z}F)^{m}+N^{-n/2}(G_{z}F)^{n}(1-N^{-1/2}G_{z}F)^{-1}\\
    &=\left(1-N^{-n/2}(G_{z}F)^{n}\right)^{-1}\sum_{m=0}^{n-1}N^{-m/2}(G_{z}F)^{m}.
\end{align*}
Therefore we can make a Taylor expansion of the determinant $D$ and truncate at order $n$ using $|\tr A|\leq N\|A\|$:
\begin{align*}
    D&=\exp\left\{-\sum_{m=0}^{n-1}\frac{1}{N^{m/2}}\tr\begin{pmatrix}G_{z}(i\eta_{1})F_{11}&\cdots&G_{z}(i\eta_{1})F_{1k}\\\vdots&\ddots&\vdots\\G_{z}(i\eta_{k})F_{k1}&\cdots &G_{z}(i\eta_{k})F_{kk}\end{pmatrix}^{n}+O(N^{(3-n\epsilon)/2})\right\}.
\end{align*}
Estimating each term by \ref{def:multiLL}, we find that only the first two terms are non-negligible.
\end{proof}

Finally, in order to restrict to the events in Lemma \ref{lem:traceApproximation}, we need upper bounds on $K_{i}$ and $F_{N-k}$ that are uniform in $\mbf{v}$.
\begin{lemma}\label{lem:FBound}
Let $z_{i}\in\mbb{D},\,i=1,...,k$. We have
\begin{align}
    K_{i}(z_{i})&\lesssim\sqrt{\frac{t^{3}}{N}}e^{\frac{N}{t}(\eta^{(i-1)}_{z_{i}})^{2}}\det^{-1}\bigl((\eta^{(i-1)}_{z_{i}})^{2}+|X^{(i-1)}_{z_{i}}|^{2}\bigr),
\end{align}
and
\begin{align}
    F_{N-k}(\mbf{z},X^{(k)})&\lesssim(Nt)^{k(k-1/2)}\prod_{i=1}^{k}e^{-\frac{N}{t}(\eta^{(i-1)}_{z_{i}})^{2}}\det\bigl((\eta^{(i-1)}_{z_{i}})^{2}+|X^{(i-1)}_{z_{i}}|^{2}\bigr),
\end{align}
uniformly in $(\mbf{v}_{1},...,\mbf{v}_{k})$.
\end{lemma}
Deferring the proof to the appendix, we can now conclude the proof of Theorem \ref{thm2}.
\begin{proof}[Proof of Theorem \ref{thm2}]
We start from the formula for $\rho^{(k)}_{N}$ in \eqref{eq:rho(k)}. Let
\begin{align*}
    n_{0}&=n_{\epsilon,k},\\
    n_{i}&=\frac{n_{i-1}}{2}-1,\qquad i=1,...,k,
\end{align*}
where $n_{\epsilon,k}$ is defined by \eqref{eq:n_eps}. Define the events
\begin{align*}
    \mc{E}_{i}&:=\{X^{(i)}\in\mc{M}_{N-i}(z,n_{i},\epsilon)\},\qquad i=1,...,k.
\end{align*}
By Lemma \ref{lem:traceApproximation}, we have 
\begin{align*}
    \mu_{i}(\mc{E}_{i})&\geq1-e^{-c\log^{2}N}.
\end{align*}
Using the bounds on $K_{i}$ and $F_{N-k}$ from Lemma \ref{lem:FBound}, we have
\begin{align*}
    K_{1}(z_{1})\mbb{E}_{1}\Bigl[\cdots K_{k}(z_{k})\mbb{E}_{k}\Bigl[\bigl(1-\prod_{i=1}^{k}1_{\mc{E}_{i}}\bigr)F_{N-k}(\mbf{z},X^{(k)})\Bigr]\cdots\Bigr]&\lesssim e^{-c\log^{2}N}.
\end{align*}
Working on $\mc{E}_{1}\cap\cdots\cap\mc{E}_{k}$, we can repeat the asymptotic analysis in \cite{maltsev_bulk_2024} with the following changes: in place of \cite[Lemma 3.5]{maltsev_bulk_2024} we use \ref{def:multiLL}, in place of \cite[Eq. (6.21)]{maltsev_bulk_2024} we use Lemma \ref{lem:detApprox2} and in place of \cite[Lemma 7.2]{maltsev_bulk_2024} we use Lemma \ref{lem:conc}. By our choice of $n_{\epsilon,k}$ we have $n_{i}\epsilon>3$ for each $i=0,1,...,k$ which ensures that we can apply Lemma \ref{lem:detApprox2}.
\end{proof}

\section{Multi-resolvent Local Law}\label{sec3}
Let $\epsilon>0$, $q=N^{\epsilon}$ and $X\in\mc{S}_{N}(\epsilon)$, where we recall Definition \ref{def1} of $\mc{S}_{N}(\epsilon)$. Let 
\begin{align*}
    \kappa_{r,s}&=Nq^{r+s-2}r!s!\mbb{E}X_{ij}^{r}\bar{X}_{ij}^{s}
\end{align*}
denote the rescaled cumulants of $X_{ij}$. By definition we have $\kappa_{0,1}=\kappa_{0,2}=0$ and $\kappa_{1,1}=1$. We will use the notion of second order renormalisation (see \cite[Eq. (42)]{cipolloni_eigenstate_2021}),
\begin{align}
    \underline{zf(z,\bar{z})}&:=zf(z,\bar{z})-\mbb{E}|z|^{2}\mbb{E}\bar{\partial}_{z}f(z,\bar{z}),
\end{align}
where $\bar{\partial}_{z}=\partial/\partial\bar{z}$, Central to our arguments is the cumulant expansion (see e.g. \cite[Lemma 7.1]{he_mesoscopic_2017}). We will apply it in the situation where we have two polynomials $f$ and $g$ in entries of resolvents $G_{z}(w)$, one of which is renormalised:
\begin{align}
    \mbb{E}\underline{X_{ij}f}g&=\frac{1}{N}\mbb{E}f\bar{\partial}_{ij}g+\sum_{r+s=2}^{L}\frac{\kappa_{r+1,s}}{Nq^{r+s-1}}\mbb{E}\partial^{r}_{ij}\bar{\partial}^{s}_{ij}fg+O(N^{-D}),
\end{align}
where, for any $D>0$, the error $N^{-D}$ can be achieved by truncating the sum at some finite $L$ depending on $\epsilon$ (this is a straightforward consequence of the single resolvent local law in Theorem \ref{thm:singleLL} below and the fact that $q=N^{\epsilon}$).

We recall the following useful identity:
\begin{align}
    G_{z}(w)E_{-}&=-E_{-}G_{z}(-w).\label{eq:GE-}
\end{align}
We will use the convention that fraktur indices take values in $[2N]$ and non-fraktur indices in $[N]$. A hat on an index means addition of $N$ modulo $2N$, i.e. $\hat{\mathfrak{i}}=\mathfrak{i}+N\text{ mod }2N$.

Define
\begin{align}
    \hat{u}_{z}(w)&:=\frac{1}{|z|^{2}-(\Tr{G}+w)^{2}},\quad \hat{m}_{z}(w):=(\Tr{G}+w)\hat{u}_{z}(w),
\end{align}
and
\begin{align}
    \hat{M}_{z}(w)&:=\begin{pmatrix}\hat{m}_{z}(w)&-z\hat{u}_{z}(w)\\-\bar{z}\hat{u}_{z}(w)&\hat{m}_{z}(w)\end{pmatrix}.\label{eq:Mhat}
\end{align}
Then using the tautology $G_{z}(w)(W_{z}-w)=1$ we can derive
\begin{align}
    G_{z}(w)&=\hat{M}_{z}(w)-\hat{M}_{z}(w)\underline{WG_{z}(w)}\label{eq:self1}\\
    &=\hat{M}_{z}(w)-\underline{G_{z}(w)W}\hat{M}_{z}(w),\label{eq:self2}
\end{align}
We also have the following identities which extend the renormalisation to a product of resolvents:
\begin{align}
    \underline{WG_{1}}B_{1}\cdots G_{m+1}&=-\sum_{l=2}^{m+1}\sum_{\nu=\pm}\nu\Tr{G_{1}B_{1}\cdots G_{l}E_{\nu}}E_{\nu}G_{l}B_{l}\cdots G_{m+1}+\underline{WG_{1}B_{1}\cdots G_{m+1}},\label{eq:self3}\\
    G_{1}B_{1}\cdots\underline{G_{m+1}W}&=-\sum_{l=1}^{m}\sum_{\nu=\pm}\nu\Tr{G_{l}B_{l}\cdots G_{m+1}E_{\nu}}G_{1}B_{1}\cdots G_{l}E_{\nu}+\underline{G_{1}B_{1}\cdots G_{m+1}W}.\label{eq:self4}
\end{align}

Let
\begin{align}
    M_{z}(w)&=\begin{pmatrix}m_{z}(w)&-zu_{z}(w)\\-\bar{z}u_{z}(w)&m_{z}(w)\end{pmatrix},\label{eq:M_{z}single}
\end{align}
be the deterministic approximation to $G_{z}(w)$, i.e. $m_{z}$ is the unique solution of
\begin{align}
    -\frac{1}{m_{z}(w)}&=m_{z}(w)+w-\frac{|z|^{2}}{m_{z}(w)+w},\quad \Im w\cdot\Im m_{z}(w)>0,\label{eq:cubic}
\end{align}
and
\begin{align}
    u_{z}(w)&=\frac{m_{z}(w)}{m_{z}(w)+w}.\label{eq:u}
\end{align}
For $|z|<1$ we have the asymptotic
\begin{align}
    m_{z}(w)&=\text{sign}(\Im w)\cdot i\sqrt{1-|z|^{2}}+O(w)\label{eq:m},
\end{align}
and the bound
\begin{align}
    |m'_{z}(w)|&\lesssim\frac{1}{|1-|z|^{2}|+|\Im w|^{2/3}}.\label{eq:m'}
\end{align}
Let 
\begin{align}
    M_{z}(w_{1},B_{1},...,w_{k})\label{eq:M_{z}multi}
\end{align}
denote the deterministic approximation to resolvent chain
\begin{align*}
    G_{z}(w_{1}) B_{1}\cdots G_{z}(w_{k-1})B_{k-1}G_{z}(w_{k}).
\end{align*}
A recursive definition of $M_{z}$ can be found in \cite[Definition 4.1]{cipolloni_optimal_2024}, although the explicit definition will not be important for us.

Throughout this section we will state errors in terms of
\begin{align}
    \mc{E}^{av}_{\eta,q}&:=\frac{1}{N\eta}+\frac{1}{q},\label{eq:avError}\\
    \mc{E}^{iso}_{\eta,q}&:=\frac{1}{\sqrt{N\eta}}+\frac{1}{q}.\label{eq:isoError}
\end{align}
A basic input for the arguments below is the sparse single resolvent law from \cite{he_edge_2023}.
\begin{theorem}[Theorem 3.1 in \cite{he_edge_2023}]\label{thm:singleLL}
For any $\delta>0$ we have
\begin{align}
    |\Tr{G_{z}(w)-M_{z}(w)}|&\prec\mc{E}^{av}_{\eta,q},\label{eq:avSingleLL}\\
    \max_{\mathfrak{x},\mathfrak{y}\in[2N]}|(G_{z}(w))_{\mathfrak{x},\mathfrak{y}}-(M_{z}(w))_{\mathfrak{x},\mathfrak{y}}|&\prec\mc{E}^{iso}_{\eta,q},\label{eq:isoSingleLL}
\end{align}
uniformly in
\begin{align}
    \mbf{D}_{\delta}&:=\bigl\{(z,w)\in\mbb{C}^{2}:|z|\leq1-\delta,\,w=E+i\eta,\,|E|\leq\delta,\,N^{-1+\delta}\leq\eta\leq\delta^{-1}\bigr\}.
\end{align}
\end{theorem}
A consequence of this is that
\begin{align}
    \|\hat{M}_{z}(w)\|&\lesssim1,
\end{align}
with very high probability. Note that \cite[Theorem 3.1]{he_edge_2023} has a weaker error bound $(N\eta)^{-1/6}+q^{-1/3}$, which is converted into a stronger bound in \cite[Theorem 3.4]{he_edge_2023} at the edge. Similar arguments can be made in the bulk to obtain \eqref{eq:avSingleLL} and \eqref{eq:isoSingleLL} (alternatively, we can use a simpler version of the arguments in this section with \cite[Theorem 3.1]{he_edge_2023} as input; see Appendix \ref{app:singleLL}).

Our goal in this section is to prove Theorem \ref{thm3}. Define
\begin{align}
    S^{av}_{z}(w_{1},B_{1},...,w_{m},B_{m})&:=\Tr{G_{1}B_{1}\cdots G_{m}B_{m}}-\Tr{M_{z}(w_{1},B_{1},...,w_{m})B_{m}},\label{eq:S^av}\\
    S^{iso}_{\mathfrak{i},\mathfrak{j},z}(w_{1},B_{1},...,w_{m+1})&:=\bigl(G_{1}B_{1}\cdots G_{m+1}-M_{z}(w_{1},B_{1},...,w_{m+1})\bigr)_{\mathfrak{i},\mathfrak{j}}\label{eq:S^iso},
\end{align}
where $G_{j}:=G_{z}(w_{j})$. We can restate Theorem \ref{thm3} as follows.
\begin{theorem}\label{thm:multiLL}
Let $\epsilon>0$ and $X\in\mc{S}_{N}(\epsilon)$. Let $m\in\mbb{N}$, $B_{i}\in\mbb{H}$ for $i=1,...,m$ and
\begin{align}
    a&:=\bigl|\{i:B_{i}\in\{F,F^{*}\}\}\bigr|.
\end{align}
Fix $\delta\geq0,\tau>0$ and define
\begin{align}
    \mc{D}(\delta,\tau)&:=\bigl\{w=E+i\eta:|E|\leq\delta|\eta|,\,N^{-1+\tau}\leq|\eta|\leq10\bigr\},\label{eq:domain}
\end{align}
Then for any fixed $z\in\mbb{D}$
\begin{align}
    \bigl|S^{av}_{z}(w_{1},B_{1},...,w_{m},B_{m})\bigr|&\prec\frac{\mc{E}^{av}_{\eta,q}}{\eta^{m-a/2-1}\wedge1},\label{eq:averagedLL}\\
    \max_{\mathfrak{x},\mathfrak{y}\in[2N]}\bigl|S^{iso}_{\mathfrak{x},\mathfrak{y},z}(w_{1},B_{1},...,w_{m+1})\bigr|&\prec\frac{\mc{E}^{iso}_{\eta,q}}{\eta^{m-a/2}},\label{eq:entry-wiseLL}
\end{align}
uniformly in $w_{j}\in\mc{D}(\delta,\tau)$, where $\eta=\min_{j}|\Im w_{j}|$.
\end{theorem}

For even $a$, the error is indeed of lower order than the deterministic approximation, as shown in the following lemma whose proof is given in Appendix \ref{app:deterministicBounds}.
\begin{lemma}\label{lem:deterministicBounds}
Let $m\in\mbb{N}$, $B_{i}\in\mbb{H}$ for $i=1,...,m$ and
\begin{align}
    a&=|\{i:B_{i}\in\{F,F^{*}\}\}|.
\end{align}
Then for any $z\in\mbb{D}$ and $w_{j}\in\mbb{C}\setminus\mbb{R}$ such that $|\Re w_{j}|\lesssim|\Im w_{j}|$ we have
\begin{align}
    \Tr{M_{z}(w_{1},B_{1},...,w_{m})B_{m}}&\lesssim\frac{1}{\eta^{m-\lceil a/2\rceil-1}},\label{eq:deterministicAverageBound}\\
    \|M_{z}(w_{1},B_{1},...,w_{m+1})\|&\lesssim\frac{1}{\eta^{m-\lceil a/2\rceil}},\label{eq:deterministicNormBound}
\end{align}
where $\eta=\min_{i}|\Im w_{i}|$.
\end{lemma}
The condition $|\Re w_{j}|\lesssim|\Im w_{j}|$ is needed for the case $(m,a)=(2,1)$; in general we have
\begin{align}
    |\Tr{M_{z}(w_{1},B_{1},w_{2})}|&\lesssim1+\frac{|\Re w_{1}+\Re w_{2}|}{\eta},
\end{align}
if $\Im w_{1}\Im w_{2}<0$.

For $M\in\mbb{N}$ even, define the error parameters
\begin{align}
    \Psi^{av}(\eta,m,a)&:=\frac{1_{\{m\leq M/2\}}\cdot\mc{E}^{av}_{\eta,q}+1_{\{m>M/2\}}\cdot\mc{E}^{iso}_{\eta,q}}{\eta^{m-a/2-1}},\label{eq:Psi^av}\\
    \Psi^{iso}(\eta,m,a)&:=\frac{1_{\{m\leq M/2\}}\cdot\mc{E}^{iso}_{\eta,q}+1_{\{m>M/2\}}}{\eta^{m-a/2}},\label{eq:Psi^iso}
\end{align}
and the properties
\begin{align}
    \mc{L}^{av}_{z}(m,\delta,\tau)&:\quad\max_{B_{j}\in\mbb{H}}\sup_{w_{j}\in\mc{D}(\delta,\tau)}\frac{|S^{av}(w_{1},B_{1},...,w_{m},B_{m})|}{\Psi^{av}(\eta,m,a)}\prec1,\label{eq:Lav}\\
    \mc{L}^{iso}_{z}(m,\delta,\tau)&:\quad\max_{B_{j}\in\mbb{H}}\sup_{w_{j}\in\mc{D}(\delta,\tau)}\frac{|S^{iso}(w_{1},B_{1},...,w_{m+1})|}{\Psi^{iso}(\eta,m,a)}\prec1,\label{eq:Liso}
\end{align}
We will also need to isolate the following special case of $\mc{L}^{av}_{z}(2,\delta,\tau)$:
\begin{align}
    \mc{L}^{av}_{z}(2,1,\delta,\tau)&:\quad\max_{B_{1}\in\{F,F^{*}\}}\max_{\nu=\pm}\sup_{w_{j}\in\mc{D}(\delta,\tau)}\frac{|S^{av}_{z}(w_{1},B_{1},w_{2},E_{\nu})|}{\Psi^{av}(\eta,2,1)}\prec1,\label{eq:Lav21}
\end{align}
If $m\leq M/2$, then Theorem \ref{thm:multiLL} is equivalent to $\mc{L}^{av}_{z}(m,\delta,\tau)$ and $\mc{L}^{iso}_{z}(m,\delta,\tau)$. In the course of proving Theorem \ref{thm:multiLL} we will need the weaker errors for $m>M/2$.

We will use the ``zigzag" method of Cipolloni--Erd\H{o}s--Schr\"{o}der \cite{cipolloni_mesoscopic_2024}, which has three parts. First, we show that the local laws hold in the global regime $|\Im w|\gtrsim1$; in this regime we can afford the trivial bound $\|G_{z}(w)\|\leq|\Im w|^{-1}$ which greatly simplifies the proof. Second, we show that if we simultaneously evolve the parameters $(z,w)$ in time $t$ while adding a Gaussian component of variance $t$, then the desired bounds propagate along the flow. Third, we show that, for a certain range of $|\Im w|$, the Gaussian component can be removed by bounding the time derivative of resolvents. The second (``zig") and third (``zag") steps are iterated until we reach the threshold $|\Im w|\geq N^{-1+\tau}$. Combining these three steps, we can propagate the local laws from the global to local regimes. The first two steps are not much affected by the sparsity of the matrix (as long as we restrict to $|\Re w|\lesssim|\Im w|$). The major work comes in the third step, for which we will rely on iterated cumulant expansions (see \cite{lee_local_2018,huang_transition_2020,he_fluctuations_2021,schnelli_convergence_2022,lee_higher_2024} for the use of iterated cumulant expansions in the study of edge statistics of sparse Hermitian matrices). These steps will be discussed in three successive subsections below.

\subsection{Global Law}
For the global law, the arguments $w_{j}$ of the resolvents are far from the real axis, i.e. $|\Im w_{j}|\gtrsim1$.
\begin{lemma}\label{lem:global}
For any $z\in\mbb{D}$, $\delta>0$ and $m\in\mbb{N}$, $\mc{L}^{av}_{z}(m,\delta,1)$ and $\mc{L}^{iso}_{z}(m,\delta,1)$ are true.
\end{lemma}
\begin{proof}
The proof is based on the Lee--Schnelli \cite{lee_local_2018} approach via recursive moment estimates. The main simplification in the global regime is that we can afford the norm bound $\|G_{z}(w)\|\leq|\Im w|^{-1}$. Since the proof is very similar to the proof of the global multi-resolvent local law for Wigner matrices \cite[Appendix B]{cipolloni_optimal_2022}, we merely give a sketch for the entry-wise law. 

Let 
\begin{align*}
    S_{\mathfrak{i},\mathfrak{j}}&\equiv S_{\mathfrak{i},\mathfrak{j}}^{iso}(w_{1},B_{1},...,w_{m+1}),\\
    \Omega_{p}(m,a)&=\max_{\mathfrak{i},\mathfrak{j}\in[2N]}\sup_{w_{j}\in\mc{D}(\delta,1)}\mbb{E}|S_{\mathfrak{i},\mathfrak{j}}|^{2p};
\end{align*}
we want to show that, for any fixed $p\in\mbb{N}$,
\begin{align}
    \Omega_{p}(m,a)&\prec\left(\Psi^{iso}(\eta,m,a)\right)^{2p}.\label{eq:Omega2pm}
\end{align}
We do this by induction on $m$. For simplicity we drop the subscripts from $G_{j}$ and $B_{j}$ and we drop the argument of $\Omega_{p}$. The base case $m=0$ is the single resolvent local law. Assume that the global averaged and entry-wise laws are true for $l\leq m$ and $l\leq m-1$ respectively. Using \eqref{eq:self2} and \eqref{eq:self4}, we have
\begin{align*}
    (GB)^{m}G&=(GB)^{m-1}GB\hat{M}+\sum_{l=0}^{m-1}\sum_{\nu=\pm}\nu\Tr{(GB)^{m-l}GE_{\nu}}(GB)^{l}GE_{\nu}\hat{M}-\underline{(GB)^{m}GW}\hat{M}.
\end{align*}
Replacing one factor of $S$ in $|S|^{2p}$ with the right-hand side above and using the induction hypothesis, we find that
\begin{align*}
    \mbb{E}|S_{i,j}|^{2p}&\prec\Psi^{iso}(\eta,m,a)\Omega^{(2p-1)/2p}_{p}+\left|\mbb{E}(\underline{(GB)^{m}GW}\hat{M})_{i,j}S_{i,j}^{p-1}\bar{S}_{i,j}^{p}\right|,
\end{align*}
where we have used H\"{o}lder's inequality in the form $\mbb{E}|S_{\mathfrak{i},\mathfrak{j}}|^{2p-1}\leq(\mbb{E}|S_{\mathfrak{i},\mathfrak{j}}|^{2p})^{(2p-1)/2p}$. Henceforth we will ignore the difference between $S$ and $\bar{S}$, which is immaterial to the proof (we write for instance $S^{p-1}\bar{S}^{p}=S^{2p-1}$). Since $\hat{M}$ is a scalar in each block we have
\begin{align*}
    (\underline{(GB)^{m}GW}\hat{M})_{i,j}&=\sum_{k}\bigl(\hat{m}\underline{((GB)^{m}G)_{i,\hat{k}}\bar{X}_{j,k}}-\bar{z}\hat{u}\underline{((GB)^{m}G)_{i,k}X_{k,j}}\bigr).
\end{align*}
Consider the contribution from the first term above; by a cumulant expansion we have
\begin{align*}
    \mbb{E}\sum_{k}\underline{((GB)^{m}G)_{i,\hat{k}}\bar{X}_{j,k}}S^{2p-1}_{i,j}&=\mbb{E}\frac{1}{N}\sum_{k}((GB)^{m}G)_{i,\hat{k}}\partial_{j,k}\hat{m}S_{i,j}^{2p-1}\\
    &+\sum_{r+s=2}^{L}\frac{C_{r,s}}{Nq^{r+s-1}}\mbb{E}\sum_{k}\partial^{r}_{j,k}\bar{\partial}^{s}_{j,k}\hat{m}((GB)^{m}G)_{i,\hat{k}}S_{i,j}^{2p-1}\\
    &+O(N^{-D}).
\end{align*}
Consider the contribution from the second cumulant (the first term in the right hand side above). The derivative acting on $\hat{m}$ gives
\begin{align*}
    \frac{\partial\hat{m}}{\partial\Tr{G}}\cdot\frac{1}{N}(G^{2})_{\hat{k},j}&=\frac{\hat{u}+2(\hat{m})^{2}}{N}(G^{2})_{\hat{k},j}.
\end{align*}
Since $|\hat{m}|,|\hat{u}|\lesssim1$, by $|(G^{2})_{i,j}|\leq\eta^{-1}\Im G_{i,j}$ we gain a factor of $(N\eta)^{-1}$. The derivative acting on $S$ gives
\begin{align*}
    \sum_{l=0}^{m}((GB)^{m-l}G)_{i,j}((GB)^{l}G)_{\hat{k},j}.
\end{align*}
Evaluating the sum over $k$ we obtain terms of the form
\begin{align*}
    \frac{1}{N}\mbb{E}((GB)^{m}G(E_{+}-E_{-})(GB)^{l}G)_{i,j}((GB)^{m-l}G)_{i,j}S_{i,j}^{2p-2}.
\end{align*}
When $n\leq m-1$ we can use the induction hypothesis to estimate $(GB)^{n}G$. When $n\geq m+1$, we use Cauchy--Schwarz and the norm bound:
\begin{align*}
    |((GB)^{m+l}G)_{i,j}|&\leq\frac{1}{\eta}\bigl(((GB)^{m/2+l-1}\Im G(B^{*}G^{*})^{m/2+l-1})_{i,i}((G^{*}B^{*})^{m/2}\Im G(BG)^{m/2})_{j,j}\bigr)^{1/2}\\
    &\leq\frac{1}{2\eta^{l}}\bigl(((GB)^{m/2}\Im G(B^{*}G^{*})^{m/2})_{i,i}+((G^{*}B^{*})^{m/2}\Im G(BG)^{m/2})_{j,j}\bigr),
\end{align*}
where we have assumed for concreteness that $m$ is even (similar manipulations can be done for $m$ odd). Combining these bounds with Young's inequality and the definition of $\Omega_{p}$ we obtain
\begin{align*}
    \left|\frac{1}{N}\mbb{E}((GB)^{m}G(E_{+}-E_{-})(GB)^{l}G)_{i,j}((GB)^{m-l}G)_{i,j}S_{i,j}^{2p-2}\right|&\lesssim\frac{\Omega^{(2p-2)/2p}_{p}}{N\eta^{m+1}}+\frac{\Omega^{(2p-1)/2p}_{p}}{N\eta^{m/2+1}}.
\end{align*}

The contribution from higher order cumulants can be bounded in a similar fashion. The first simplyfying remark is that derivatives of $\hat{m}$ and $\hat{u}$ with respect to $\Tr{G}$ are themselves polynomials in $\hat{m}$ and $\hat{u}$. The single factor of $G^{2}$ that results from derivatives of $\Tr{G}$ with respect to matrix entries can be reduced to $G$ by the identity $|G^{2}|=\eta^{-1}\Im G$. Therefore, derivatives of $\hat{m}$ and $\hat{u}$ result in the product of a polynomial in $\hat{m}$ and $\hat{u}$ and a polynomial in entries of $G$. We organise the sum in terms of the number $d$ of copies of $S$ on which derivatives act, so that a general term (up to a polynomial in $\hat{m}$ and $\hat{u}$, which can be ignored since $\hat{m},\hat{u}\lesssim1$ with very high probability) takes the form
\begin{align*}
    R&:=\frac{1}{Nq^{r+s-1}}\mbb{E}\sum_{k}P_{x}(\mathfrak{i},\mathfrak{j})Q_{y}(\mathfrak{k},\mathfrak{l})S_{i,j}^{2p-d},
\end{align*}
where $P_{x}$ and $Q_{y}$ are monomials of degree $x$ and $y$ in entries of $(GB)^{n}G,\,n\leq m-1$ and $(GB)^{m}G$ respectively:
\begin{align*}
    P_{x}(\mbf{n},\mathfrak{i},\mathfrak{j})&=((GB)^{n_{1}}G)_{\mathfrak{i}_{1},\mathfrak{j}_{1}}\cdots((GB)^{n_{x}}G)_{\mathfrak{i}_{x},\mathfrak{j}_{x}},\\
    Q_{y}(\mathfrak{k},\mathfrak{l})&=((GB)^{m}G)_{\mathfrak{k}_{1},\mathfrak{l}_{1}}\cdots((GB)^{m}G)_{\mathfrak{k}_{y},\mathfrak{l}_{y}}.
\end{align*}
The indices $\mathfrak{i},\mathfrak{j},\mathfrak{k},\mathfrak{l}$ belong to $\{i,j,\hat{k}\}$ and we have the relations $x+y=r+s+1$, $|\mbf{n}|=(d-y)m$ and $1\leq d\leq r+s+1$. Let
\begin{align*}
    \mc{O}_{x}&\equiv\mc{O}_{x}(\mathfrak{i},\mathfrak{j})=\{j:|\{\mathfrak{i}_{j},\mathfrak{j}_{j}\}\cap\{\hat{k}\}|=1\},\\
    \mc{O}_{y}&\equiv\mc{O}_{y}(\mathfrak{k},\mathfrak{l})=\{j:|\{\mathfrak{k}_{j},\mathfrak{l}_{j}\}\cap\{\hat{k}\}|=1\},
\end{align*}
count off-diagonal occurrences of index $\hat{k}$. If $|\mc{O}_{x}|=|\mc{O}_{y}|=0$, then each copy of $S$ on which a derivative acts must be acted on by at least two derivatives, and there must be at least one derivative acting on the factor $((GB)^{m}G)_{i,\hat{k}}$. This implies that $r+s\geq 2d-1$ and $r+s\geq3$ if $d=1$, which in turn guarantees that we have sufficient powers of $q^{-1}$. Using the induction hypothesis for $P_{x}$, the definition of $\Omega_{p}$ for $Q_{y}$ and Young's inequality, we obtain the bound
\begin{align*}
    R&\prec\frac{\Omega^{(2p-d)/2p}_{p}}{q^{d}\eta^{d(m-a/2)}}.
\end{align*}

There are two cases in which $|\mc{O}_{x}|+|\mc{O}_{y}|=1$. In the first case, an odd number of derivatives act on $((GB)^{m}G)_{i,\hat{k}}$ and exactly one copy of $S$. In the second case, an even number of (or zero) derivatives act on $((GB)^{m}G)_{i,\hat{k}}$ and all copies of $S$ on which a derivative acts are acted on by at least two derivatives. In either case, we must have $r+s\geq2d-2$ and $r+s\geq2$ if $d=1$. We are therefore one short of the requisite number of factors of $q^{-1}$, but we can gain a factor of $(N\eta)^{-1/2}$ from the off-diagonal entry with index $k$ by Cauchy--Schwarz. As before, if Cauchy--Schwarz leads to a factor of $(GB)^{m+l}G$ we use the norm bound to reduce to $(GB)^{m}G$. Thus we have the bound
\begin{align*}
    R&\prec\frac{\Omega^{(2p-d)/2p}_{p}}{\sqrt{N\eta}q^{d-1}\eta^{d(m-a/2)}}.
\end{align*}

If $|\mc{O}_{x}|+|\mc{O}_{y}|\geq2$, then we can gain a factor of $(N\eta)^{-1}$ from the sum over $k$. Since we always have $r+s\geq d-1$, this leads to the bound
\begin{align*}
    R&\prec\frac{\Omega^{(2p-d)/2p}_{p}}{Nq^{d-2}\eta^{d(m-a/2)+1}}.
\end{align*}

Combining these bounds we have
\begin{align*}
    R&\prec\left(\Psi^{iso}(\eta,m,a)\right)^{d}\Omega^{(2p-d)/2p}_{p}.
\end{align*}
Summing over $d=1,...,2p$, we have found that
\begin{align*}
    \left|\mbb{E}\sum_{k}\underline{((GB)^{m}G)_{i,\hat{k}}\bar{X}_{j,k}}S^{2p-1}_{i,j}\right|&\prec\sum_{d=1}^{2p}\left(\Psi^{iso}(\eta,m,a)\right)^{d}\Omega^{(2p-d)/2p}_{p},
\end{align*}
and ultimately
\begin{align*}
    \mbb{E}|S|^{2p}_{i,j}&\prec\sum_{d=1}^{2p}\left(\Psi^{iso}(\eta,m,a)\right)^{d}\Omega^{(2p-d)/2p}_{p}.
\end{align*}
Taking the maximum over $i,j$ (and similarly for $\hat{\imath},\hat{\jmath}$), this implies \eqref{eq:Omega2pm}.
\end{proof}

\subsection{Characteristic Flow}\label{sec:zig}
The second step is to simultaneously evolve the matrix $X$ and the spectral parameters $(z,w)$ in such a way that $|\Im w|$ decreases as the variance of the Gaussian component in $X$ increases. The evolution of $(z,w)$ is chosen so that certain terms cancel in the ensuing stochastic differential equations (SDEs) for resolvents.

Let $B_{ij}(t),\,i,j=1,...,N$ be iid standard complex Brownian motions and consider the SDE
\begin{align}
    \mathrm{d}X(t)&=-\frac{1}{2}X(t)\mathrm{d}t+\frac{1}{\sqrt{N}}\mathrm{d}B(t),\label{eq:matrixOU}
\end{align}
with initial condition $X(0)=X$. Let
\begin{align}
    \Lambda_{t}(w,z)&:=\begin{pmatrix}w_{t}&z_{t}\\\bar{z}_{t}&w_{t}\end{pmatrix}
\end{align}
solve
\begin{align}
    \frac{\mathrm{d}\Lambda_{t}}{\mathrm{d}t}&=-\Tr{M_{t}}-\frac{1}{2}\Lambda_{t},\label{eq:dLambda}
\end{align}
with initial condition $(w,z)\in\mbb{C}^{2}$, where
\begin{align}
    M_{t}&\equiv M_{z_{t}}(w_{t}):=\begin{pmatrix}m_{t}&-z_{t}u_{t}\\-\bar{z}_{t}u_{t}&m_{t}\end{pmatrix}.
\end{align}
This is equivalent to the system
\begin{align}
    \frac{\mathrm{d}w_{t}}{\mathrm{d}t}&=-m_{t}-\frac{1}{2}w_{t},\quad\frac{\mathrm{d}z_{t}}{\mathrm{d}t}=-\frac{1}{2}z_{t}.\label{eq:dw}
\end{align}
By \cite[Lemma 6.5]{cipolloni_mesoscopic_2024}, we have
\begin{align}
    \mathrm{d}M_{t}&=\frac{1}{2}M_{t}\mathrm{d}t.
\end{align}
In particular, $\mathrm{d}m_{t}=\frac{1}{2}m_{t}\mathrm{d}t$. The explicit solution is given by
\begin{align}
    (w_{t},z_{t})&=\bigl(w_{0}e^{-t/2}-m_{0}(e^{t/2}-e^{-t/2}),z_{0}e^{-t/2}\bigr).
\end{align}
From this we can obtain the bound
\begin{align}
    \int_{0}^{t}\frac{1}{|\Im w_{s}|^{\alpha+1}}\diff s&\lesssim\frac{1}{|\Im m_{z_{t}}(w_{t})||\Im w_{t}|^{\alpha}},\qquad\alpha>0.\label{eq:integralBound}
\end{align}
The solution to the reverse dynamics is given by
\begin{align}
    (w_{-t},z_{-t})&=\bigl(w_{0}e^{t/2}+m_{0}(e^{t/2}-e^{-t/2}),z_{0}e^{t/2}\bigr).\label{eq:reverse1}
\end{align}

For $w_{j}\in\mbb{C}\setminus\mbb{R}$, let $(w_{j,t},z_{t})$ be the solution to \eqref{eq:dw} with initial condition $(w_{j},z)$, and
\begin{align}
    G_{j,t}&:=\begin{pmatrix}-w_{j,t}&X(t)-z_{t}\\X(t)^{*}-\bar{z}_{t}&-w_{j,t}\end{pmatrix}^{-1}.
\end{align}
Define
\begin{align}
    S^{av}_{t}(w_{1},B_{1},...,w_{m},B_{m})&:=\Tr{\bigl(G_{1,t}B_{1}\cdots G_{m,t}-M_{t}(w_{1},B_{1},...,w_{m})\bigr)B_{m}},\\
    S^{iso}_{\mathfrak{i},\mathfrak{j},t}(w_{1},B_{1},...,w_{m+1})&=\bigl(G_{1,t}B_{1}\cdots G_{m+1,t}-M_{t}(w_{1},B_{1},...,w_{m+1})\bigr)_{\mathfrak{i},\mathfrak{j}},
\end{align}
where $M_{t}(w_{1},B_{1},...,w_{m+1})$ is the deterministic approximation to $G_{1,t}B_{1}\cdots G_{m+1,t}$. For $\mbf{w}\in\mbb{C}^{m}$, we define
\begin{align}
    \sigma(\mbf{w})&:=\{\mbf{w}'\in\mbb{C}^{m}:|\Re w'_{j}|=|\Re w_{k}|,\,|\Im w'_{j}|\geq|\Im w_{k}|\text{ for some }k\},
\end{align}
and we use the same symbol $\sigma$ for each $m\in\mbb{N}$. Note that if $\mbf{w}'\in\sigma(\mbf{w})$ then $\sigma(\mbf{w}')\subset\sigma(\mbf{w})$. Let $M\in\mbb{D}$ be even and define the stopping times
\begin{align}
    \tau^{av}_{z}(\mbf{w})&:=\min_{1\leq m\leq M}\inf\left\{t\geq0:\max_{B_{j}\in\mbb{H}}\sup_{\mbf{w}'\in\sigma(\mbf{w})}\frac{|S^{av}_{t}(w'_{1},B_{1},...,w'_{m},B_{m})|}{\Psi^{av}(\eta_{t},m,a)}>N^{(m+a)\xi}\right\},\\
    \tau^{iso}_{z}(\mbf{w})&:=\min_{1\leq m\leq M}\inf\left\{t\geq0:\max_{B_{j}\in\mbb{H}}\sup_{\mbf{w}'\in\sigma(\mbf{w})}\frac{|S^{iso}_{t}(w'_{1},B_{1},...,w'_{m+1})|}{\Psi^{iso}(\eta_{t},m,a)}>N^{(m+2+a)\xi}\right\},
\end{align}
where $\eta_{t}=\min_{j}|\Im w_{j,t}|$ and $a$ is the number of $B_{j}\in\{F,F^{*}\}$. Define the overall stopping time
\begin{align}
    \tau_{z}(\mbf{w})&=\min\{\tau^{av}_{z}(\mbf{w}),\tau^{iso}_{z}(\mbf{w})\}.
\end{align}
Note that the power of $N^{\xi}$ in the definition of $\tau^{av/iso}$ grows in $m\in[M]$ and $a\in[m]$. The reason for this choice is that when we evaluate the derivative of $S^{av/iso}_{t}$ we will bound the terms with lower values of $m$ and $a$ using the definition of $\tau^{av/iso}$. Moreover, the power of $N^{\xi}$ is larger for $\tau^{iso}$ than $\tau^{av}$, for a similar reason. Note also that the errors for $S^{av/iso}_{t}$ for $m>M/2$ are larger than those for $m\leq M/2$, i.e. in order to obtain the strong error for $M/2$ resolvents we need to obtain a weaker error for $M$ resolvents. The maximal time for which we run the flow is
\begin{align}
    T(\mbf{w})&:=\inf\left\{t\geq0:\min_{j}|\Im w_{j,t}|<N^{-1}\right\}.
\end{align}

\begin{lemma}\label{lem:averagedStoppingTime}
Let $z\in\mbb{D}$, $m\in\mbb{N}$, $\mbf{w}_{0}\in(\mbb{C}\setminus\mbb{R})^{m}$ such that $|\Re w_{0,j}|\lesssim|\Im w_{0,j}|$, $\mbf{w}\in\sigma(\mbf{w}_{0})$ and assume that
\begin{align}
    |S^{av}_{0}(w_{1},B_{1},...,w_{m},B_{m})|&<N^{(m+a-1)\xi}\Psi^{av}(\eta,m,a),
\end{align}
with very high probability. Then, for any $0\leq t\leq \tau(\mbf{w}_{0})\wedge T(\mbf{w}_{0})$, we have
\begin{align}
    |S^{av}_{t}(w_{1},B_{1},...,w_{m},B_{m})|&<N^{(m+a-1)\xi}\Psi^{av}(\eta_{t},m,a),\label{eq:averagedStoppingTime}
\end{align}
with very high probability.
\end{lemma}
\begin{proof}
If there is a $p\in[1,...,m]$ such that either $B_{p}=E_{+}$ and $\Im(w_{p})\Im(w_{p+1})<0$ or $B_{p}=E_{-}$ and $\Im(w_{p})\Im(w_{p+1})>0$, then by \eqref{eq:GE-} and the resolvent identity we have
\begin{align*}
    |S^{av}_{t}(w_{1},B_{1},...,w_{m},B_{m})|&\leq\frac{1}{|\Im w_{p,t}|+|\Im w_{p+1,t}|}\left(|S^{av}_{t}(w_{1},B_{1},...,\pm w_{p},B_{p+1},...,w_{m})|\right.\\
    &\left.+|S^{av}_{t}(w_{1},B_{1},...,w_{p-1},B_{p-1},w_{p+1},...,w_{m})|\right),
\end{align*}
and the terms on the right hand side contain $m-1$ resolvents. When $m=2$, $B_{1}\in\{F,F^{*}\}$ and $B_{2}=E_{\pm}$, we require a separate argument since we cannot afford the factor of $\eta^{-1}$ that comes from the resolvent identity. Consider $\Tr{G_{1}FG_{2}}$ where $\Im w_{1}\Im w_{2}<0$. Since $\Tr{A^{*}}=\overline{\Tr{A}}$ we can assume without loss that $0<\eta_{1}=\Im w_{1}\leq\eta_{2}=-\Im w_{2}$. Using $G_{1}=G_{2}^{*}+(w_{1}-\bar{w}_{2})G_{2}^{*}G_{1}$ we have
\begin{align*}
    \Tr{G_{1}FG_{2}}&=\Tr{(1+(w_{1}-\bar{w}_{2})G_{1})FG_{2}G_{2}^{*}}=\frac{w_{1}-\bar{w}_{2}}{2i\eta_{2}}\Tr{G_{1}F(G_{2}^{*}-G_{2})},
\end{align*}
and so
\begin{align}
    \left(1+\frac{w_{1}-\bar{w}_{2}}{2i\eta_{2}}\right)\Tr{G_{1}FG_{2}}&=\frac{w_{1}-\bar{w}_{2}}{2i\eta_{2}}\Tr{G_{1}FG_{2}^{*}}.\label{eq:G1FG2}
\end{align}
By our assumption that $\eta_{1}\leq\eta_{2}$ and $|\Re w_{j}|\lesssim|\Im w_{j}|$, we have
\begin{align*}
    \left|1+\frac{w_{1}-\bar{w}_{2}}{2i\eta_{2}}\right|&\geq c>0,\\
    \left|\frac{w_{1}-\bar{w}_{2}}{2i\eta_{2}}\right|&\leq C.
\end{align*}
Since the identity in \eqref{eq:G1FG2} is also true for the deterministic approximations, we have
\begin{align*}
    |S^{av}_{t}(w_{1},F,w_{2},E_{+})|&\lesssim|S^{av}_{t}(w_{1},F,\bar{w}_{2},E_{+})|,
\end{align*}
and on the right hand side the imaginary parts of the arguments have the same sign: $\Im w_{1}\Im \bar{w}_{2}>0$. Thus, we can assume the following:
\begin{equation*}
    \text{if $B_{p}=E_{+}$ then $\Im(w_{p})\Im(w_{p+1})>0$; if $B_{p}=E_{-}$ then $\Im(w_{p})\Im(w_{p+1})<0$.}\tag{\textasteriskcentered}\label{eq:signAssumption}
\end{equation*}
Here we consider addition modulo $m$, so for $p=m$ we have $p+1=1$. Due to this assumption, the fact that 
\begin{align*}
    \Tr{M_{t}(w_{p},E_{\pm},w_{p+1})E_{\mp}}&=0,
\end{align*}
and \eqref{eq:m'}, we have
\begin{align}
    |\Tr{M_{t}(w_{p},B_{p},w_{p+1})E_{\nu}}|&\lesssim1,\label{eq:signBound}
\end{align}
whenever $B_{p}=E_{\pm}$.

A second observation is that (see e.g. \cite[Eq. (5.4)]{cipolloni_optimal_2022})
\begin{align}
    |G_{z}(w)|&=\frac{2}{\pi}\int_{0}^{N^{L}}\frac{\Im G_{z}(w_{x})}{\eta_{x}}\diff x+O_{\prec}(N^{-D}),\label{eq:absoluteValue}
\end{align}
for any $D>0$ and sufficiently large $L>0$ depending on $D$, where $\eta_{x}:=\sqrt{\eta^{2}+x^{2}}$ and $w_{x}:=E+i\eta_{x}$. Moreover, we have
\begin{align*}
    \int_{0}^{N^{L}}\frac{\diff x}{\eta_{x}^{b/2}}&\prec\frac{1}{\eta^{b-1}},\qquad b\geq1.
\end{align*}
Thus resolvent chains with or without absolute values $|G_{j,t}|$ obey the same bounds. For any $x>0$, we can find a $w'\in\sigma(w)$ such that $\Im w'_{t}=\sqrt{(\Im w_{t})^{2}+x^{2}}$, which allows us to bound resolvent chains with absolute values using the definition of $\tau(\mbf{w}_{0})$.

Let $t\in[0,\tau(\mbf{w}_{0})\wedge T(\mbf{w}_{0})]$ and $\eta_{t}=\min_{j}|\Im w_{j,t}|$. By It\^{o}'s lemma and the chain rule we have
\begin{align}
    \diff S^{av}_{t}(w_{1},B_{1},...,w_{m},B_{m})&=\frac{m}{2}S^{av}_{t}(w_{1},B_{1},...,w_{m},B_{m})\diff t+\sum_{p\leq r}^{m}A_{p,r}(t)\diff t\nonumber\\
    &+\frac{1}{\sqrt{N}}\sum_{i,j=1}^{N}\partial_{ij}\bigl(S^{av}_{t}(w_{1},B_{1},...,w_{m},B_{m})\bigr)\diff B_{ij}\label{eq:dS}\\
    &+\frac{1}{\sqrt{N}}\sum_{i,j=1}^{N}\bar{\partial}_{ij}\bigl(S^{av}_{t}(w_{1},B_{1},...,w_{m},B_{m})\bigr)\diff \bar{B}_{ij},\nonumber
\end{align}
where
\begin{align}
    A_{p,p}(t)&=\Tr{G_{p,t}-M_{p,t}}\Tr{G_{1,t}B_{1}\cdots G^{2}_{p,t}B_{p}\cdots G_{m,t}B_{m}},\label{eq:App}
\end{align}
and, for $p<r$,
\begin{align}
    A_{p,r}(t)&=\sum_{\nu=\pm}\nu\Big[\Tr{G_{p,t}B_{p}\cdots G_{r,t}E_{\nu}}\Tr{G_{r,t}B_{r}\cdots G_{m,t}B_{m}G_{1,t}B_{1}\cdots G_{p,t}E_{\nu}}\nonumber\\
    &-\Tr{M_{t}(w_{p},B_{p},...,w_{r})E_{\nu}}\Tr{M_{t}(w_{r},B_{r},...,w_{m},B_{m},w_{1},B_{1},...,w_{p})E_{\nu}}\Big].\label{er:Apr}
\end{align}

Consider first $A_{1,1}$. By the single resolvent law we have
\begin{align*}
    \left|\Tr{G_{1,t}-M_{1,t}}\right|&\leq N^{\xi/2}\mc{E}^{av}_{\eta_{t},q},
\end{align*}
with very high probability. If $m=M$, then since $M$ is even we have
\begin{align*}
    |\Tr{G^{2}_{1,t}B_{1}\cdots G_{m,t}B_{m}}|&=\left|\Tr{(G_{1,t}B_{1}\cdots B_{m/2}G_{m/2+1,t}^{1/2})\cdot(G_{m/2+1,t}^{1/2}B_{m/2+1}\cdots G_{m,t}B_{m}G_{1,t})}\right|\\
    &\leq\frac{1}{\eta_{1,t}}\Tr{\Im(G_{1,t})B_{1}G_{2,t}\cdots B_{m/2}|G_{m/2+1,t}|B_{m/2}^{*}\cdots G_{2,t}^{*}B_{1}^{*}}^{1/2}\\
    &\times\Tr{\Im(G_{1,t})B_{m}^{*}G_{m,t}^{*}\cdots B_{m/2+1}^{*}|G_{m/2+1,t}|B_{m/2+1}\cdots G_{m,t}B_{m}}^{1/2}\\
    &\leq\frac{1}{\eta_{t}^{m-a/2}},
\end{align*}
where in the second line we used Cauchy--Schwarz and in the last line we used the integral representation in \eqref{eq:absoluteValue} and the definition of $\tau(\mbf{w}_{0})$ to bound the traces involving $m$ resolvents. If $m<M$, then from the definition of $\tau(\mbf{w}_{0})$ we have
\begin{align*}
    |\Tr{G^{2}_{1,t}B_{1}\cdots G_{m,t}B_{m}}|&\leq|\Tr{M_{t}(w_{1},E_{+},w_{1},B_{1},...,w_{m},B_{m})}|+N^{(m+a)\xi}\Psi^{av}(\eta_{t},m,a)\\
    &\lesssim\frac{1}{\eta_{t}^{m-a/2}}.
\end{align*}
Bounding $A_{p,p},\,p=1,...,m$ in the same way we obtain
\begin{align*}
    \int_{0}^{t}|A_{p,p}(s)|\diff s&\lesssim N^{\xi/2}\Psi^{av}(\eta_{t},m,a),
\end{align*}
where we have used the integral inequality in \eqref{eq:integralBound}.

Now consider $A_{p,r}$ for $p<r$, which we rewrite as follows:
\begin{equation}
\begin{aligned}
    A_{p,r}&=\sum_{\nu=\pm}\nu\Big[\Tr{M_{t}(w_{1},B_{1},...,w_{p},E_{\nu},w_{r},B_{r},...,w_{m})B_{m}}S^{av}_{t}(w_{p},B_{p},...,w_{r},E_{\nu})\\
    &+\Tr{M_{t}(w_{p},B_{p},...,w_{r})E_{\nu}}S^{av}_{t}(w_{1},B_{1},...,w_{p},E_{\nu},w_{r},B_{r},...,w_{m},B_{m})\\
    &+S^{av}_{t}(w_{p},B_{p},...,w_{r},E_{\nu})S^{av}_{t}(w_{1},B_{1},...,w_{p},E_{\nu},w_{r},B_{r},...,w_{m},B_{m})\Big]
\end{aligned}\label{eq:Apq2}
\end{equation}
From the definition of $\tau(\mbf{w}_{0})$ we have
\begin{align*}
    \frac{|S^{av}_{t}(w_{p},B_{p},...,w_{r},E_{\nu})|}{\Psi^{av}(\eta_{t},r-p+1,a_{p,r})}&\leq N^{(r-p+1+a_{p,r})\xi},\\
    \frac{|S^{av}_{t}(w_{1},B_{1},...,w_{p},E_{\nu},w_{r},B_{r},...,w_{m},B_{m})|}{\Psi^{av}(\eta_{t},m-r+p+1,a-a_{p,r})}&\leq N^{(m-r+p+1+a-a_{p,r})\xi},
\end{align*}
where $a_{p,r}$ is the number of $B_{j}\in\{F,F^{*}\}$ for $j=p,...,r-1$. Since $0\leq a_{p,r}\leq a$, if either: i) $r-p\in[2,3,...,m-2]$; ii) $r=p+1$ and $a_{p,r}=1$; iii) $p=1,\,r=m$ and $a_{p,r}=a-1$, we have
\begin{align*}
    r-p+1+a_{p,r}&\leq m+a-1,\\
    m-r+p+1+a-a_{p,r}&\leq m+a-1,
\end{align*}
and so
\begin{align*}
    \int_{0}^{t}|A_{p,r}(s)|\diff s&\lesssim N^{(m+a-1)\xi}\Psi^{av}(\eta_{t},m,a).
\end{align*}
This leaves the cases $p=1,\,r=m$ and $a_{p,r}=a$ or $r=p+1,\,p=2,...,m-1$ and $a_{p,r}=0$, where the relevant terms are those in the first and second lines respectively on the right hand side of \eqref{eq:Apq2} (the term in the last line is bounded by $\bigl(N^{(m+a)\xi}\Psi^{av}\bigr)^{2}\lesssim N^{(m+a-1)\xi}\Psi^{av}$). Note that $a_{p,p+1}=0$ and $a_{1,m}=a$ mean that $B_{p}=E_{\pm}$ and $B_{m}=E_{\pm}$ respectively. Since $\Tr{M_{t}(w_{1},E_{\pm},w_{2})E_{\mp}}=0$, these terms are
\begin{align*}
    \left(\sum_{p:B_{p}=E_{\pm}}\pm\Tr{M_{t}(w_{p},B_{p},w_{p+1})B_{p}}\right)S^{av}_{t}(w_{1},B_{1},...,w_{m},B_{m}),
\end{align*}
where the $\pm$ sign is chosen according to $B_{p}=E_{\pm}$. Observe that $S^{av}_{t}(w_{1},B_{1},...,w_{m},B_{m})$ on the right hand side is the same quantity whose derivative we have computed.

For the stochastic term we bound the quadratic variation by
\begin{align*}
    &\frac{1}{N}\sum_{p=1}^{m}\sum_{i,j=1}^{N}\left|\Tr{G_{1,t}B_{1}\cdots G_{p,t}\mbf{e}_{i}\mbf{e}_{\hat{\jmath}}^{*}G_{p,t}B_{p}\cdots G_{m,t}B_{m}}\right|^{2}\\
    &\leq\frac{1}{N^{2}\eta_{p,t}^{2}}\sum_{p=1}^{m}\Tr{\Im G_{p,t}B_{p}\cdots G_{m,t}B_{m}G_{1,t}B_{1}\cdots\Im G_{p,t}B_{p-1}^{*}\cdots G_{1,t}^{*}B_{m}^{*}G_{m,t}^{*}B_{m-1}^{*}\cdots G_{p+1,t}^{*}B_{p}^{*}}.
\end{align*}
The trace in the last line involves an alternating product of $2m$ resolvents. If $m\leq M/2$ then we can directly estimate this using the definition of $\tau(\mbf{w}_{0})$ to obtain the bound
\begin{align*}
    \frac{1}{N^{2}\eta_{p,t}^{2}}\cdot\frac{1}{\eta^{2m-a-1}_{t}}&\prec\frac{1}{N^{2}\eta^{2m-a+1}_{t}}.
\end{align*}
Integrating over $t$ and using the Burkholder--Davis--Gundy (BDG) inequality we obtain
\begin{align*}
    \left|\int_{0}^{t}\frac{1}{\sqrt{N}}\sum_{i,j=1}^{N}\partial_{ij}\bigl(S^{av}_{s}(w_{1},B_{1},...,w_{m},B_{m})\bigr)\diff B_{ij}(s)\right|&\leq N^{\xi}\left(\int_{0}^{t}\frac{1}{N^{2}\eta_{s}^{2m-a+1}}\diff s\right)^{1/2}\\
    &\leq\frac{N^{\xi}}{N\eta_{t}}\cdot\frac{1}{\eta^{m-a/2-1}_{t}}\\
    &\leq N^{\xi}\Psi^{av}(\eta_{t},m,a),
\end{align*}
with very high probability.

If $m>M/2$, we reduce the resolvent chain to chains of $m$ resolvents using the following argument based on the spectral decomposition of $G_{z}(w)$ (see \cite[Eq. (5.8)-(5.10)]{cipolloni_optimal_2022}). For simplicity we drop the subscripts $j$ from $G_{j}$ and $B_{j}$. If $m$ is even then
\begin{align*}
    \left|\Tr{(GB)^{2m}}\right|&=\left|\Tr{(GB)^{m/2-1}GB(GB)^{m/2-1}GB(GB)^{m/2-1}GB(GB)^{m/2-1}GB}\right|\\
    &\leq\frac{1}{2N}\sum\frac{|\mbf{w}_{i}^{*}B(GB)^{m/2-1}\mbf{w}_{j}\mbf{w}_{j}^{*}B(GB)^{m/2-1}\mbf{w}_{k}\mbf{w}_{k}^{*}B(GB)^{m/2-1}\mbf{w}_{l}\mbf{w}_{l}^{*}B(GB)^{m/2-1}\mbf{w}_{i}|}{|\lambda_{i}-w||\lambda_{j}-w||\lambda_{k}-w||\lambda_{l}-w|}\\
    &\lesssim\frac{1}{2N}\sum\frac{|\mbf{w}_{i}^{*}B(GB)^{m/2-1}\mbf{w}_{j}|^{2}|\mbf{w}_{k}^{*}B(GB)^{m/2-1}\mbf{w}_{l}|^{2}}{|\lambda_{i}-w||\lambda_{j}-w||\lambda_{k}-w||\lambda_{l}-w|}\\
    &\lesssim N\Tr{|G|B(GB)^{m/2-1}|G|(B^{*}G^{*})^{m/2-1}B^{*}}^{2}.
\end{align*}
If $m$ is odd we obtain instead
\begin{align*}
    \left|\Tr{(GB)^{2m}}\right|&\lesssim N\Tr{|G|B(GB)^{(m-1)/2}|G|(B^{*}G^{*})^{(m-1)/2}B^{*}}\Tr{|G|B(GB)^{(m-3)/2}|G|(B^{*}G^{*})^{(m-3)/2}B^{*}}.
\end{align*}
The first factor contains $m+1\leq M$ resolvents since $m\leq M-1$ when $m$ is odd. In either case we obtain the bound $N^{-1}\eta_{t}^{-2m+a}$ for the quadratic variation. By the BDG inequality we obtain
\begin{align*}
    \left|\int_{0}^{t}\frac{1}{\sqrt{N}}\sum_{i,j=1}^{N}\partial_{ij}\bigl(S^{av}_{s}(w_{1},B_{1},...,w_{m},B_{m})\bigr)\diff B_{ij}(s)\right|&\leq N^{\xi}\left(\int_{0}^{t}\frac{1}{N\eta_{s}^{2m-a}}\diff s\right)^{1/2}\\
    &\leq\frac{N^{\xi}}{\sqrt{N\eta_{t}}}\cdot\frac{1}{\eta^{m-a/2-1}_{t}}\\
    &\leq N^{\xi}\Psi^{av}(\eta_{t},m,a),
\end{align*}
with very high probability. 

Altogether we have shown that
\begin{align*}
    \diff S^{av}_{t}(w_{1},B_{1},...,w_{m},B_{m})&=\phi(t)S^{av}_{t}(w_{1},B_{1},...,w_{m},B_{m})\diff t+R(t)\diff t,
\end{align*}
where
\begin{align*}
    \phi(t)&=\frac{m}{2}+\sum_{p:B_{p}=E_{\pm}}\pm\Tr{M_{t}(w_{p},B_{p},w_{p+1})B_{p}}
\end{align*}
and
\begin{align*}
    \int_{0}^{t}|R(s)|\diff s&\lesssim\frac{N^{(m+a-1)\xi}\mc{E}^{av}_{\eta_{t},q}}{\eta_{t}^{m-a/2-1}}.
\end{align*}
By \eqref{eq:signBound} we have
\begin{align*}
    \int_{0}^{t}|\phi(s)|\diff s&\lesssim1,
\end{align*}
and so by Gr\"{o}nwall's inequality and the assumption on $|S^{av}_{0}(w_{1},B_{1},...,w_{m},B_{m})|$ we obtain \eqref{eq:averagedStoppingTime}.
\end{proof}

For the entry-wise stopping time $\tau^{iso}$, we have
\begin{lemma}\label{lem:entry-wiseStoppingTime}
Let $z\in\mbb{D}$, $m\in\mbb{N}$, $\mbf{w}_{0}\in(\mbb{C}\setminus\mbb{R})^{m+1}$, $\mbf{w}\in\sigma(\mbf{w}_{0})$ and assume that
\begin{align}
    |S^{iso}_{\mathfrak{i},\mathfrak{j},0}(w_{1},B_{1},...,w_{m+1})|&<N^{(m+1+a)\xi}\Psi^{iso}(\eta,m,a).
\end{align}
Then, for any $0\leq t\leq \tau(\mbf{w}_{0})\wedge T(\mbf{w}_{0})$, we have
\begin{align}
    |S^{iso}_{\mathfrak{i},\mathfrak{j},t}(w_{1},B_{1},...,w_{m+1})|&\lesssim N^{(m+1+a)\xi}\Psi^{iso}(\eta_{t},m,a).
\end{align}
\end{lemma}
We postpone the proof to Appendix \ref{app:entry-wiseStoppingTime} since it is very similar to the preceeding proof (indeed, it more or less amounts to the replacement of $B_{m}$ with $N\mbf{e}_{\mathfrak{j}}\mbf{e}^{*}_{i}$). Combining Lemmas \ref{lem:averagedStoppingTime} and \ref{lem:entry-wiseStoppingTime}, we conclude the following, which is the essence of the ``zig" step.
\begin{proposition}\label{prop:zig}
For $t>0$ let $X(t)$ be the solution to \eqref{eq:matrixOU}. Let $z\in\mbb{D}$, $\mbf{w}_{0}\in(\mbb{C}\setminus\mbb{R})^{m+1}$ and $\mbf{w}\in\sigma(\mbf{w}_{0})$. If 
\begin{align}
    \max_{B_{j}\in\mbb{H}}\sup_{\mbf{w}'\in\sigma(\mbf{w})}\frac{|S^{av}_{0}(w_{1},B_{1},...,w_{m},B_{m})|}{\Psi^{av}(\eta,m,a)}<N^{(m+a-1)\xi},\\
    \max_{B_{j}\in\mbb{H}}\sup_{\mbf{w}'\in\sigma(\mbf{w})}\frac{|S^{iso}_{0}(w_{1},B_{1},...,w_{m+1})|}{\Psi^{iso}(\eta,m,a)}<N^{(m+a)\xi},
\end{align}
for each $m\in[M]$, then $\tau_{z}(\mbf{w}_{0})\geq T(\mbf{w}_{0})$.
\end{proposition}

\subsection{Green's Function Comparison}\label{sec:zag}
In this step we remove the Gaussian component by a Green's function comparison argument. The matrix $X(t)$ follows the same dynamics \eqref{eq:matrixOU} as before but now the spectral parameters are time-independent; we redefine $S^{av/iso}_{t}$ accordingly. Let $\delta,\tau>0$ and recall the definition \eqref{eq:domain} of $\mc{D}(\delta,\tau)$. For $p>0$ define 
\begin{align}
    \Omega^{iso}_{t}(m,p)&:=\max_{\mathfrak{i},\mathfrak{j}\in[2N]}\max_{B_{j}\in\mbb{H}}\sup_{w_{j}\in\mc{D}(\delta,\tau)}\mbb{E}|S^{iso}_{\mathfrak{i},\mathfrak{j},t}(w_{1},B_{1},...,w_{m+1})|^{2p},\\
    \Omega^{av}_{t}(m,p)&:=\max_{B_{j}\in\mbb{H}}\sup_{w_{j}\in\mc{D}(\delta,\tau)}\mbb{E}|S^{av}_{t}(w_{1},B_{1},...,w_{m},B_{m})|^{2p}.
\end{align}

The main ingredients are the bounds on the time derivative of high moments of $S^{iso/av}$ in the following two lemmas.
\begin{lemma}\label{lem:isoGFT}
Let $z\in\mbb{D}$, $m\in\mbb{N}$, $t\geq0$ and $S:=S^{iso}_{\mathfrak{x},\mathfrak{y},t}(w_{1},B_{1},...,w_{m+1})$. Assume that $\mc{L}^{iso}_{z}(l,\delta,\tau)$ is true for $l=0,...,m-1$. Then for any fixed $p\in\mbb{N}$
\begin{align}
    \left|\frac{\diff}{\diff t}\mbb{E}|S|^{2p}\right|&\prec\left(1+\frac{1}{q\eta}\right)\left(\left(\Psi^{iso}(\eta,m,a)\right)^{2p}+\Omega^{iso}_{t}(m,p)\right).\label{eq:isoGFT}
\end{align}
uniformly in $\mc{D}(\delta,\tau)$, where $\eta=\min_{j}|\Im w_{j}|$.
\end{lemma}
\begin{lemma}\label{lem:avGFT}
Let $z\in\mbb{D}$, $m\in\mbb{N}$, $t\geq0$ and $S:=S^{av}_{t}(w_{1},B_{1},...,w_{m},B_{m})$. Assume that $\mc{L}^{iso}_{z}(l,\delta,\tau)$ is true for $l=0,...,m$. Then for any fixed $p\in\mbb{N}$
\begin{align}   
    \left|\frac{\diff}{\diff t}\mbb{E}|S|^{2p}\right|&\prec\left(1+\frac{1}{q\eta}\right)\left(\left(\Psi^{av}(\eta,m,a)\right)^{2p}+\Omega^{av}_{t}(m,p)\right),\label{eq:avGFT}
\end{align}
uniformly in $\mc{D}(\delta,\tau)$, where $\eta=\min_{j}|\Im w_{j}|$.
\end{lemma}
We will prove the first lemma in this section; the proof of the second is much shorter and left to Appendix \ref{app:avGFT}.

Before proving the lemma, we introduce a class of polynomials in resolvent entries in order to organise terms that result from cumulant expansions. In the following, for a tuple $\mathfrak{i}=(\mathfrak{i}_{1},...,\mathfrak{i}_{a})$, $\{\mathfrak{i}\}$ is shorthand for $\{\mathfrak{i}_{1},...,\mathfrak{i}_{a}\}$ and similarly for $\{\mathfrak{i},\mathfrak{j}\}$.
\begin{definition}\label{def:polynomials}
Let $w_{i,j}\in\mbb{C}\setminus\mbb{R}$, $B_{i,j}\in\mbb{H}$ and $G_{i,j}\equiv G_{z}(w_{i,j})$. For $x,y\in\mbb{N}$ and tuples of indices $\mathfrak{i},\mathfrak{j},\mathfrak{k},\mathfrak{l}$ define
\begin{align}
    P_{x}(\mbf{n},\mathfrak{i},\mathfrak{j})&:=(G_{1,1}B_{1,1}\cdots G_{1,n_{1}+1})_{\mathfrak{i}_{1},\mathfrak{j}_{1}}\cdots(G_{x,1}B_{x,1}\cdots G_{x,n_{x}+1})_{\mathfrak{i}_{x},\mathfrak{j}_{x}},\label{eq:P}\\
    Q_{y}(\mathfrak{k},\mathfrak{l})&=(G_{1,1}B_{1,1}\cdots G_{1,m+1})_{\mathfrak{k}_{1},\mathfrak{l}_{1}}\cdots(G_{y,1}B_{y,1}\cdots G_{y,m+1})_{\mathfrak{k}_{y},\mathfrak{l}_{y}},\label{eq:Q}
\end{align}
and
\begin{align}
    \mc{O}_{x}\equiv\mc{O}_{x}(\mathfrak{i},\mathfrak{j})&:=\{j:\mathfrak{i}_{j}\neq\mathfrak{j}_{j}\},\label{eq:O_a}\\
    \mc{G}_{x}\equiv\mc{G}_{x}(\mathfrak{i},\mathfrak{j})&:=\{j:n_{j}\leq m/2\}.\label{eq:G_a}
\end{align}
The set $\mc{P}$ consists of polynomials of the form
\begin{align}
    \sum_{i_{1},...,i_{b}}P_{x}(\mbf{n},\mathfrak{i},\mathfrak{j})Q_{b}(\mathfrak{k},\mathfrak{l}),\label{eq:polynomial}
\end{align}
where, with $I(\mbf{i}):=\{i_{1},\hat{\imath}_{1},...,i_{b},\hat{\imath}_{b}\}$, we have:
\begin{enumerate}
\item $|\{\mathfrak{i}_{i},\mathfrak{j}_{i}\}\cap I(\mbf{i})|\geq1,\qquad i=1,...,x$;
\item $|\{\mathfrak{k}_{i},\mathfrak{l}_{i}\}\cap I(\mbf{i})|\geq1,\qquad i=1,...,y$;
\item $|\{j\in[x]:|\{\mathfrak{i}_{j},\mathfrak{j}_{j}\}\cap I(\mbf{i})|=1\}\cup\{j\in[y]:|\{\mathfrak{k}_{j},\mathfrak{l}_{j}\}\cap I(\mbf{i})|=1\}|\geq2$;
\item there is a $j\in[x]$ such that $|\{\mathfrak{i}_{j},\mathfrak{j}_{j}\}|\cap I(\mbf{i})=1$ and $n_{j}\leq m/2$.
\end{enumerate}
The subset $\mc{P}^{(1)}$ consists of polynomials for which
\begin{align}
    \left|\left\{j\in\mc{O}_{x}\cap\mc{G}_{x}:|\{\mathfrak{i}_{j},\mathfrak{j}_{j}\}\cap\{\mathfrak{k},\mathfrak{l}\}\cap I(\mbf{i})|<|\{\mathfrak{i}_{j},\mathfrak{j}_{j}\}\cap I(\mbf{i})|\right\}\right|&\geq2.
\end{align}
The subset $\mc{P}^{(2)}$ consists of polynomials for which
\begin{align}
    |\mc{O}_{x}\cap\mc{G}_{x}|&\geq2,
\end{align}
and
\begin{align}
    \left|\left\{j\in\mc{O}_{x}\cap\mc{G}_{x}:|\{\mathfrak{i}_{j},\mathfrak{j}_{j}\}\cap\{\mathfrak{k},\mathfrak{l}\}\cap I(\mbf{i})|<|\{\mathfrak{i}_{j},\mathfrak{j}_{j}\}\cap I(\mbf{i})|\right\}\right|&<2.
\end{align}
\end{definition}
Note that we allow the indices $\mathfrak{i},\mathfrak{j},\mathfrak{k},\mathfrak{l}$ to take values other than those in $I(\mbf{i})$. The first two conditions in the definition of $\mc{P}$ mean that each resolvent entry carries at least one summation index. The third condition means that there are always at least two off-diagonal entries, and the fourth condition means that at least one of these is an entry of $(GB)^{n}G$ with $n\leq m/2$. The definitions of $\mc{P}^{(1)}$ and $\mc{P}^{(2)}$ require that there are at least two off-diagonal entries of $(GB)^{n}G$ with $n\leq m/2$. For $\mc{P}^{(1)}$, at least two of these carry a summation index that is not carried by any entry of $(GB)^{m}G$. Here and throughout this subsection, off-diagonal refers to entries $(\mathfrak{i},\mathfrak{j})$ such that $\mathfrak{j}\notin\{\mathfrak{i},\hat{\mathfrak{i}}\}$, i.e. entries not on the diagonal of the $N\times N$ blocks. We will also often use the shorthand $(GB)^{n}G$ to represent $G_{1}B_{1}\cdots G_{n+1}$ for generic $w_{j}$ and $B_{j}$.

Using the single resolvent local law in the form $|G_{\mathfrak{i},\mathfrak{j}}|\prec1$, the deterministic bound $\|G\|\leq\eta^{-1}$ and Cauchy--Schwarz, we immediately obtain that the polynomial in \eqref{eq:polynomial} is stochastically dominated by
\begin{align}
    N^{b-1}\eta^{-(|\mbf{n}|+x-n_{0}+(m+1)y)},\label{eq:naive}
\end{align}
where $n_{0}:=|\{j:n_{j}=0\}|$ and $|\mbf{n}|=\sum_{i}n_{i}$.

In the following three lemmas, we consider the product of a polynomial in resolvent entries and a power of a single entry $S:=((GB)^{m}G)_{\mathfrak{x},\mathfrak{y}}$. First we have a general bound which holds conditionally on local laws for products of fewer resolvents.
\begin{lemma}[General polynomial bound]\label{lem:polynomialBound1}
Let $z\in\mbb{D}$, $m\in\mbb{N}$ and assume that $\mc{L}^{iso}_{z}(l,\delta,\tau)$ is true for $l=0,1,...,m-1$. Let $d\in\mbb{N}$ such that $d\geq y$. Then for any polynomial in $\mc{P}$ we have
\begin{align}
    &\mbb{E}\frac{1}{N^{b}q^{u}}\sum_{i_{1},...,i_{b}}|P_{x}(\mbf{n},\mathfrak{i},\mathfrak{j})Q_{y}(\mathfrak{k},\mathfrak{l})S^{2p-d}|\nonumber\\
    &\prec1_{\{d>y\}}\left(\frac{1_{\{m>M/2\}}+(\mc{E}^{iso}_{\eta,q})^{|\mc{O}_{x}|+u}}{\eta^{|\mbf{n}|-a/2}}\right)^{2p/(d-y)}+\Omega^{iso}_{t}(m,p),\label{eq:polynomialBound1}
\end{align}
uniformly in $\mc{D}(\delta,\tau)$, where $a$ is the number of $B_{i,j}\in\{F,F^{*}\}$ in $P_{x}$.
\end{lemma}
\begin{proof}
We apply the assumption that the local law holds for $l=0,1,...,m-1$ to each factor in $P_{x}$. Since the deterministic approximation is diagonal in each block, this gives an additional factor of $1_{\{m>M/2\}}+\mc{E}^{iso}_{\eta,q}$ for each off-diagonal entry (of which there are $|\mc{O}_{x}|$). The desired bound now follows readily from Young's inequality.
\end{proof}

From the definition of $\mc{P}^{(1)}$, we expect that we can improve the general bound by a factor of $(N\eta)^{-1}$ by using Cauchy--Schwarz and the Ward identity for off-diagonal entries. The next lemma shows that this is indeed the case.
\begin{lemma}[Bound on $\mc{P}^{(1)}$]\label{lem:polynomialBound2}
Let $z\in\mbb{D}$, $m\in\mbb{N}$ and assume that $\mc{L}^{iso}_{z}(l,\delta,\tau)$ is true for $l=0,1,...,m-1$. Let $d\in\mbb{N}$ such that $d\geq y$. Then for any polynomial in $\mc{P}^{(1)}$ we have
\begin{align}
    &\mbb{E}\frac{1}{N^{b-1}q^{u}}\sum_{i_{1},...,i_{b}}|P_{x}(\mbf{n},\mathfrak{i},\mathfrak{j})Q_{y}(\mathfrak{k},\mathfrak{l})S^{2p-d}|\nonumber\\
    &\prec\frac{1}{q\eta}\left(1_{\{d>y\}}\left(\frac{1_{\{m>M/2\}}+(\mc{E}^{iso}_{\eta,q})^{|\mc{O}_{x}|+u-3}}{\eta^{|\mbf{n}|-a/2}}\right)^{2p/(d-y)}+\Omega^{iso}_{t}(m,p)\right),\label{eq:polynomialBound2}
\end{align}
uniformly in $\mc{D}(\delta,\tau)$, where $a$ is the number of $B_{i,j}\in\{F,F^{*}\}$ in $P_{x}$.
\end{lemma}
\begin{proof}
From the definition of $\mc{P}^{(1)}$, there are two off-diagonal entries of $(GB)^{n}G$ with $n\leq m/2$ such that each has a summation index that does not appear in $Q_{y}$. Let these be \[(G_{1,1}B_{1,1}\cdots G_{1,n_{1}+1})_{\mathfrak{i}_{1},\mathfrak{j}_{1}},\,(G_{2,1}B_{2,1}\cdots G_{2,n_{2}+1})_{\mathfrak{i_{2}},\mathfrak{j}_{2}},\] with $a_{1}$ and $a_{2}$ factors of $B_{i,j}\in\{F,F^{*}\}$ respectively. We bound the remaining entries in $P_{x}$ using the assumed local laws to obtain the bound
\begin{align*}
    &\frac{1_{\{m>M/2\}}+(\mc{E}^{iso}_{\eta,q})^{|\mc{O}_{x}|+u-3}}{\eta^{|\mbf{n}|-n_{1}-n_{2}-(a-a_{1}-a_{2})/2}}\\
    &\times\frac{1}{N^{b-1}q}\sum_{i_{1},...,i_{b}}|(G_{1,1}B_{1,1}\cdots G_{1,n_{1}+1})_{\mathfrak{i}_{1},\mathfrak{j}_{1}}|(G_{2,1}B_{2,1}\cdots G_{2,n_{2}+1})_{\mathfrak{i_{2}},\mathfrak{j}_{2}}||Q_{y}||S|^{2p-d},
\end{align*}
where we have absorbed the factor of $q^{u-1}$ in the power of $\mc{E}^{iso}_{\eta,q}$. We gain a factor of $(N\eta)^{-1}$ by using Cauchy--Schwarz and the Ward identity for the summation index (or pair of indices) that does not appear in $Q_{y}$. The bound in \eqref{eq:polynomialBound2} follows by Young's inequality.
\end{proof}

The second subset $\mc{P}^{(2)}$ obeys the same bound but the proof is more involved and so we formulate this fact as a separate lemma. In brief, the presence of an entry of $(GB)^{m}G$ with the same summation index as an off-diagonal entry of $(GB)^{n}G$ prevents a direct application of Cauchy--Schwarz, so we introduce a new summation index in order to decouple these entries, after which we can use Lemma \ref{lem:polynomialBound2}.
\begin{lemma}[Bound on $\mc{P}^{(2)}$]\label{lem:polynomialBound3}
Let $z\in\mbb{D}$, $m\in\mbb{N}$ and assume that $\mc{L}^{iso}_{z}(l,\delta,\tau)$ is true for $l=0,...,m-1$. Let $d\in\mbb{N}$ such that $d\geq y$. Then for any polynomial in $\mc{P}^{(2)}$
\begin{align}
    &\mbb{E}\frac{1}{N^{b-1}q^{u}}\sum_{i_{1},...,i_{b}}|P_{x}(\mbf{n},\mathfrak{i},\mathfrak{j})Q_{y}(\mathfrak{k},\mathfrak{l})S^{2p-d}|\nonumber\\
    &\prec\frac{1}{q\eta}\left(1_{d>y}\left(\frac{1_{\{m>M/2\}}+(\mc{E}^{iso}_{\eta,q})^{|\mc{O}_{x}|+u-3}}{\eta^{|\mbf{n}|-a/2}}\right)^{2p/(d-y)}+\Omega^{iso}_{t}(m,p)\right),\label{eq:polynomialBound3}
\end{align}
uniformly in $\mc{D}(\delta,\tau)$, where $a$ is the number of $B_{i,j}\in\{F,F^{*}\}$ in $P_{x}$.
\end{lemma}
\begin{proof}
For ease of notation we will assume that each $B_{i,j}$ in $P_{x}$ is equal to $F$ and so $a=|\mbf{n}|/2$; the general case is no different. We will also drop subscripts from resolvents $G_{j}$. By definition, there are at least two off-diagonal entries of $(GF)^{n}G$ with $n\leq m/2$ but we cannot find two such that both have a summation index that does not appear in $Q_{y}$. We will therefore replace at least one of them using \eqref{eq:self2} and \eqref{eq:self4}. For concreteness, assume that $P_{x}$ takes the form
\begin{align*}
    P_{x}(\mbf{n},\mathfrak{i},\mathfrak{j})&=((GF)^{n_{1}}G)_{\mathfrak{x},i_{1}}P_{x-1}(\mbf{n}',\mathfrak{i}',\mathfrak{j}'),
\end{align*}
with $1\leq n_{1}\leq m/2$ (the case $n_{1}=0$ is similar), and that the index $i_{1}$ appears in $Q_{y}$. By \eqref{eq:self2} and \eqref{eq:self4} we have
\begin{equation}
\begin{aligned}
    ((GF)^{n_{1}}G)_{\mathfrak{x},i_{1}}&=-\bar{z}\hat{u}((GF)^{n_{1}-1}G)_{\mathfrak{x},i_{1}}\\
    &+\frac{\hat{m}}{2}\sum_{l=0}^{n_{1}-1}\Tr{(GF)^{n_{1}-l}G(E_{+}-E_{-})}((GF)^{l}G)_{\mathfrak{x},i_{1}}\\
    &-\frac{\bar{z}\hat{u}}{2}\sum_{l=0}^{n_{1}-1}\Tr{(GF)^{n_{1}-l}G(E_{+}+E_{-})}((GF)^{l}G)_{\mathfrak{x},\hat{\imath}_{1}}\\
    &-\bigl(\underline{(GF)^{n_{1}}GW}\hat{M}\bigr)_{\mathfrak{x},i_{1}}.
\end{aligned}\label{eq:self5}
\end{equation}
The resolvent entries in the non-underlined terms on the right hand side have at most $n_{1}-1$ factors of $F$, so we can iterate this identity until all matrix entries are either entries of $\hat{M}$ (multiplied by products of traces of resolvent chains) or entries of underlined terms. A general non-underlined term is of the form
\begin{align*}
    p(\hat{m},\hat{u})\Tr{(GF)^{l_{1}}GE_{\mu_{1}}}\cdots\Tr{(GF)^{l_{\alpha}}GE_{\mu_{\alpha}}}\hat{M}_{\mathfrak{x},i_{1}},
\end{align*}
where $p$ is a polynomial of degree $n_{1}$, $l_{j}\geq1$ and $\sum_{j}l_{j}\leq n_{1}$. A general underlined term is of the form
\begin{align*}
    \wt{p}(\hat{m},\hat{u})\Tr{(GF)^{l_{1}}GE_{\nu_{1}}}\cdots\Tr{(GF)^{l_{\beta}}GE_{\nu_{\beta}}}(\underline{(GF)^{l_{\beta+1}}GW}\hat{M})_{\mathfrak{x},i_{1}},
\end{align*}
where $l_{j}$ is as before but now $\wt{p}$ is a polynomial of degree at most $n_{1}-1$. We can estimate the products of traces using the assumption $\mc{L}^{iso}_{z}(l,\delta,\tau)$ for $l=1,...,\lfloor m/2\rfloor$ and the deterministic bounds in Lemma \ref{lem:deterministicBounds}:
\begin{align*}
    |p(\hat{m},\hat{u})\Tr{(GF)^{l_{1}}GE_{\mu_{1}}}\cdots\Tr{(GF)^{l_{\alpha}}GE_{\mu_{\alpha}}}|&\prec\frac{1}{\eta^{\frac{1}{2}\sum_{u=1}^{\alpha}l_{u}}}.
\end{align*}
Note that here $\mc{L}^{iso}_{z}$ is sufficient since we only needed an upper bound:
\begin{align*}
    |\Tr{(GF)^{l}GE_{\pm}}|&=\frac{1}{2N}\Bigl|\sum_{i}\Bigl(\bigl((GF)^{l}G\bigr)_{i,i}\pm\bigl((GF)^{l}G\bigr)_{\hat{\imath},\hat{\imath}}\Bigr)\Bigr|\\
    &\lesssim\frac{1}{\eta^{l/2}}
\end{align*}
with very high probability if $\mc{L}^{iso}_{z}(l,\delta,\tau)$ is true. The non-underlined terms each have a factor $\hat{M}_{\mathfrak{x},i_{1}}$ which collapses the sum over $i_{1}$ since $\hat{M}$ is diagonal in each $N\times N$ block. Thus, as in Lemma \ref{lem:polynomialBound1}, the contribution from these terms is bounded by
\begin{align*}
    &\frac{1}{N^{b-1}q^{u}\eta^{n_{1}/2}}\sum_{i_{2},...,i_{b}}|P_{x-1}(\mbf{n}',\mathfrak{i}',\mathfrak{j}')Q_{y}(\mathfrak{k},\mathfrak{l})|\\
    &\prec 1_{d>y}\left(\frac{1_{\{m>M/2\}}+(\mc{E}^{iso}_{\eta,q})^{\mc{O}_{x}+u-1}}{\eta^{|\mbf{n}|/2}}\right)^{2p/(d-y)}+\Omega^{iso}_{t}(m,p).
\end{align*}

To bound the contribution from underlined terms, we use a second cumulant expansion. Observe that
\begin{align*}
    \bigl(\underline{(GF)^{n}GW}\hat{M}\bigr)_{\mathfrak{x},i_{1}}&=\hat{m}\bigl(\underline{(GF)^{n}GW}\bigr)_{\mathfrak{x},i_{1}}-\bar{z}\hat{u}\bigl(\underline{(GF)^{n}GW}\bigr)_{\mathfrak{x},\hat{\imath}_{1}}.
\end{align*}
We will treat the contribution from the first term; the second is treated in exactly the same way. Let
\begin{align*}
    R&:=\hat{m}\wt{p}(\hat{m},\hat{u})\Tr{(GF)^{l_{1}}GE_{\nu_{1}}}\cdots\Tr{(GF)^{l_{\beta}}GE_{\nu_{\beta}}}.
\end{align*}
Then the contribution to bounded takes the form
\begin{align*}
    &\frac{1}{N^{b-1}q^{u}}\sum_{i_{1},...,i_{b}}\mbb{E}\bigl(\underline{(GF)^{l_{\beta+1}}GW}\bigr)_{\mathfrak{x},i_{1}}P_{x-1}Q_{y}RS^{2p-d}\\
    &=\frac{1}{N^{b}q^{u}}\sum_{i_{1},...,i_{b+1}}\mbb{E}\bigl((GF)^{l_{\beta+1}}G\bigr)_{\mathfrak{x},\hat{\imath}_{b+1}}\partial_{i_{1},i_{b+1}}(P_{x-1}Q_{y}RS^{2p-d})\\
    &+\sum_{r_{1}+s_{1}=2}^{L}\frac{\kappa_{r_{1},s_{1}+1}}{N^{b}q^{u+r_{1}+s_{1}-1}}\sum_{i_{1},...,i_{b+1}}\mbb{E}\partial_{i_{1},i_{b+1}}^{r_{1}}\bar{\partial}_{i_{1},i_{b+1}}^{s_{1}}\bigl((GF)^{l_{\beta+1}}G\bigr)_{\mathfrak{x},\hat{\imath}_{b+1}}P_{x-1}Q_{y}RS^{2p-d}\\
    &+O(N^{-D})
\end{align*}
For the first term on the right hand side, we note that the single derivative produces exactly one resolvent entry with left index $\hat{\imath}_{b+1}$. The sum over $i_{b+1}$ can then be bounded by Cauchy--Schwarz and the Ward identity, gaining a factor of $(N\eta)^{-1/2}$. Note that if the derivative acts on an entry of $(GF)^{m}G$ we need to apply Young's inequality with a different factor of $(GF)^{m}G$ (of which there will always be at least one) after Cauchy--Schwarz, e.g.
\begin{align*}
    &\left(\sum_{i_{b+1}}|((GF)^{l_{\beta+1}}G)_{\mathfrak{x},\hat{\imath}_{b+1}}||((GF)^{m}G)_{\hat{\imath}_{b+1},\mathfrak{y}}|\right)|((GF)^{m}G)_{\mathfrak{z},\mathfrak{w}}|\\
    &\prec\sqrt{\frac{N}{\eta^{l_{\beta+1}+1}}}\left(\frac{1}{N}\sum_{i_{b+1}}|((GF)^{m}G)_{\hat{\imath}_{b+1},\mathfrak{y}}|^{2}\right)^{1/2}|((GF)^{m}G)_{\mathfrak{z},\mathfrak{w}}|\\
    &\prec\sqrt{\frac{N}{\eta^{l_{\beta+1}+1}}}\left(\frac{1}{N}\sum_{i_{b+1}}((GF)^{m}G)_{\hat{\imath}_{b+1},\mathfrak{y}}|^{2}+|((GF)^{m}G)_{\mathfrak{z},\mathfrak{w}}|^{2}\right).
\end{align*}
In the last line we applied Young's inequality to obtain a homogeneous polynomial in entries of $(GF)^{m}G$. The remaining sum over $i_{1},...,i_{b}$ has a prefactor $N^{-b}$ so we can use the bound in Lemma \ref{lem:polynomialBound1} to ultimately conclude that
\begin{align*}
    &\mbb{E}\frac{1}{N^{b}q^{u}}\sum_{i_{1},...,i_{b+1}}\bigl((GF)^{l_{\beta+1}}G\bigr)_{\mathfrak{x},\hat{\imath}_{b+1}}\partial_{i_{1},i_{b+1}}(P_{x-1}Q_{y}RS^{2p-d})\\
    &\prec\frac{1}{q\eta}\left(\left(\frac{1_{\{m>M/2\}}+(\mc{E}^{iso}_{\eta,q})^{|\mc{O}_{x}|+u-1}}{\eta^{|\mbf{n}|/2}}\right)^{2p/(d-y)}+\Omega^{iso}_{t}(m,p)\right).
\end{align*}

The contributions of higher order cumulants are dealt with recursively. The main observation is that, up to factors consisting of products of traces $\Tr{(GF)^{n}GE_{\mu}}$, $\hat{m}$ and $\hat{u}$, the derivatives generate polynomials in $\mc{P}^{(1)}\cup\mc{P}^{(2)}$. Indeed, first note that derivatives acting on a trace $\Tr{(GF)^{n}GE_{\mu}}$ create an entry of $(GF)^{n}GE_{\mu}G$. The extra factor of $\eta^{-1}$ we incur from $E_{\nu}$ is compensated by the factor of $N^{-1}$ from the normalisation of the trace. Thus, whenever a derivative acts on $\Tr{(GF)^{n}GE_{\mu}}$ we gain a factor of $(\mc{E}^{iso}_{\eta,q})^{2}$. Second, $\partial_{i_{1},i_{b+1}}$ (resp. $\bar{\partial}_{i_{1},i_{b+1}}$) creates two new entries, one with right (resp. left) index $i_{1}$ and one with left (resp. right) index $\hat{\imath}_{b+1}$. The only diagonal entries that can be created are thus indexed by $i_{1}$ or $\hat{\imath}_{b+1}$. In particular, derivatives acting on $\bigl((GF)^{n}G\bigr)_{\mathfrak{x},\mathfrak{y}}$ can only result in a product of diagonal entries if $\mathfrak{x},\mathfrak{y}\in\{i_{1},\hat{\imath}_{b+1}\}$. If none of the initial off-diagonal entries satisfy this constraint, then the total number of off-diagonal entries is non-decreasing in the number of derivatives. In summary, we can always reduce
\begin{align*}
    \frac{1}{N^{b}q^{u+r_{1}+s_{1}-1}}\sum_{i_{1},...,i_{b+1}}\partial^{r_{1}}_{i_{1},i_{b+1}}\bar{\partial}^{s_{1}}_{i_{1},i_{b+1}}\bigl((GF)^{l_{\beta+1}}G\bigr)_{\mathfrak{x},\hat{\imath}_{b+1}}P_{x-1}Q_{y}RS^{2p-d}
\end{align*}
to a sum of terms of the form
\begin{align*}
    \frac{1}{N^{b}q^{u+r_{1}+s_{1}-1}}\sum_{i_{1},...,i_{b+1}}P_{x'}Q_{y'}\wt{R}S^{2p-d'},
\end{align*}
for some polynomial in $\mc{P}^{(1)}\cup\mc{P}^{(2)}$ and some product of traces $\wt{R}$.

If the derivatives generate a polynomial in $\mc{P}^{(1)}$, then we use Lemma \ref{lem:polynomialBound2}. Otherwise, we repeat the above procedure, replacing one of the off-diagonal entries of $(GF)^{n}G$ with $n\leq m/2$ using \eqref{eq:self5}. The contribution from non-underlined terms is bounded as before, and we perform a second cumulant expansion to deal with the underlined terms. At the $l$-th cumulant expansion, we gain a factor $N^{-1}q^{1-r_{l}-s_{l}}$ and the resulting terms take the form
\begin{align}
    \frac{1}{N^{b+l-1}q^{u+|\mbf{r}|+|\mbf{s}|-l}}\sum_{i_{1},...,i_{c+l}}P_{x}Q_{y}RS^{2p-d},\label{eq:lthCumulantExpansion}
\end{align}
where $\mbf{r},\mbf{s}\in\mbb{N}^{l}$ are such that $r_{j}+s_{j}\geq2$. We terminate after we have accumulated enough factors of $q^{-1}$ to afford the naive bound in \eqref{eq:naive}. Since $|\mbf{n}|+my\leq2mp$, the power of $\eta^{-1}$ in \eqref{eq:naive} is bounded above by a constant depending only on $m$ and $p$, so we can terminate after a finite number of steps depending only on $m,p$ and $\epsilon$ (recall that $q=N^{\epsilon}$). This leaves us with a finite number of terms of the form \eqref{eq:lthCumulantExpansion} where the polynomial belongs to $\mc{P}^{(1)}$, and we can use Lemma \ref{lem:polynomialBound2} to obtain the desired bound, noting that $|\mbf{r}|+|\mbf{s}|-l\geq l$. We can summarise this argument in the following algorithm:
\begin{enumerate}
\item start with polynomial in $\mc{P}^{(2)}$;
\item replace an off-diagonal entry of $(GF)^{n}G$ with $n\leq m/2$ and sharing a summation index with an entry of $(GF)^{m}G$ by repeatedly applying \eqref{eq:self5} until all matrix entries are entries of $\hat{M}$ or entries of an underlined term $\underline{(GF)^{n'}GW}\hat{M}$ with $n'\leq n$;
\item bound the contribution from non-underlined terms using the fact that $\hat{M}$ is diagonal in each $N\times N$ block;
\item introduce a new summation index by performing a cumulant expansion of underlined terms;
\item bound the contribution from the second cumulant by Cauchy--Schwarz and the Ward identity;
\item higher order cumulants correspond to polynomials in $\mc{P}^{(1)}\cup\mc{P}^{(2)}$; if a term is in $\mc{P}^{(1)}$ then use Lemma \ref{lem:polynomialBound2}, otherwise return to Step 1.
\end{enumerate}
\end{proof}

Alongside the bounds on polynomials in resolvent entries, we will need the following fluctuation averaging bound.
\begin{lemma}\label{lem:fluctuationAveraging}
Let $z\in\mbb{D}$, $m\in\mbb{N}$ and assume that $\mc{L}^{iso}_{z}(l,\delta,\tau)$ is true for $l=0,1,...,m$. Then
\begin{equation}
\begin{aligned}
    \max_{\mathfrak{i}\in[2N]}\left|\sum_{j}(G_{1}B_{1}\cdots G_{m+1})_{\mathfrak{i},j}\right|&\prec\frac{1}{\eta^{m-a/2+1}},\\
    \max_{\mathfrak{i}\in[2N]}\left|\sum_{j}(G_{1}B_{1}\cdots G_{m+1})_{\mathfrak{i},\hat{\jmath}}\right|&\prec\frac{1}{\eta^{m-a/2+1}},
\end{aligned}\label{eq:fluctuationAveraging}
\end{equation}
uniformly in $w_{j}\in\mc{D}(\delta,\tau)$.
\end{lemma}
\begin{proof}
The argument is similar to \cite[Lemma 3.8]{he_edge_2023}. Again, for ease of notation we assume that $B_{j}=F$ and drop subscripts from resolvents. Let
\begin{align*}
    \mc{G}_{\mathfrak{i}}&:=\sum_{j}\bigl((GF)^{m}G\bigr)_{\mathfrak{i},j},\\
    \mc{G}_{*}&:=\max_{\mathfrak{i}}|\mc{G}_{\mathfrak{i}}|,
\end{align*}
and consider the $2p$-th absolute moment of $\mc{G}_{i}$:
\begin{align*}
    \mbb{E}|\mc{G}_{i}|^{2p}&=\mbb{E}\sum_{j}\bigl((GF)^{m}G\bigr)_{i,j}\mc{G}_{i}^{p-1}\bar{\mc{G}}_{i}^{p}.
\end{align*}
We will prove the bound
\begin{align}
    \mc{G}_{*}&\prec\frac{1}{\eta^{m/2+1}}\label{eq:G_*bound}
\end{align}
by induction on $m$. Assume that the above bound is true with $m$ replaced by $l=0,...,m-1$ and that we have the a priori bound $\mc{G}_{*}\prec\Phi$. We replace the first factor of $\mc{G}_{i}$ using \eqref{eq:self1} and \eqref{eq:self3}. Since the matrix entries in non-underlined terms are of the form $(GF)^{l}G$ for $l\leq m-1$, we can use the induction hypothesis and $\mc{L}^{iso}_{z}(l,\delta,\tau)$ for $l=1,...,m$ to obtain
\begin{align*}
    \mbb{E}\sum_{j}\Tr{(GF)^{l}GE_{\mu}}\bigl((GF)^{m-l}G\bigr)_{ij}|\mc{G}_{i}|^{2p-1}&\prec\frac{1}{\eta^{m/2+1}}\mbb{E}|\mc{G}_{i}|^{2p-1}.
\end{align*}
This leaves us with the underlined terms:
\begin{align*}
    \mbb{E}\sum_{j}\bigl(\underline{W(GF)^{m}G}\bigr)_{i,j}\mc{G}_{i}^{p-1}\bar{\mc{G}}_{i}^{p}&=\mbb{E}\sum_{k}\underline{X_{i,k}\mc{G}_{\hat{k}}}\mc{G}_{i}^{p-1}\bar{\mc{G}}_{i}^{p}\\
    &=\frac{1}{N}\mbb{E}\sum_{k}\mc{G}_{\hat{k}}\bar{\partial}_{i,k}\bigl(\mc{G}_{i}^{p-1}\bar{\mc{G}}_{i}^{p}\bigr)\\
    &+\frac{1}{N}\sum_{r+s=2}^{L}\frac{\kappa_{r+1,s}}{q^{r+s-1}}\mbb{E}\sum_{k}\partial^{r}_{i,k}\bar{\partial}^{s}_{i,k}\bigl(\mc{G}_{\hat{k}}\mc{G}_{i}^{p-1}\bar{\mc{G}}_{i}^{p}\bigr)\\
    &+O(N^{-D}).
\end{align*}
The derivatives of $\mc{G}_{i}$ are given by
\begin{align*}
    \partial_{i,k}\mc{G}_{i}&=-\sum_{n=0}^{m}\bigl((GF)^{n}G\bigr)_{i,i}\sum_{j}\bigl((GF)^{m-n}G\bigr)_{\hat{k},j},\\
    \bar{\partial}_{i,k}\mc{G}_{i}&=-\sum_{n=0}^{m}\bigl((GF)^{n}G\bigr)_{i,\hat{k}}\sum_{j}\bigl((GF)^{m-n}G\bigr)_{i,j}
\end{align*}
Using the induction hypothesis, $\mc{L}^{iso}_{z}(l,l,\delta,\tau)$ for $l=1,...,m$ and the a priori bound $\mc{G}_{*}\prec\Phi$ we have
\begin{align*}
    \{\partial_{i,k},\bar{\partial}_{i,k}\}^{r}\mc{G}_{\mathfrak{i}}&\prec\frac{1}{\eta^{m/2+1}}+\Phi.
\end{align*}
Thus the sum over $r+s\geq2$ is easily dealt with:
\begin{align*}
    \frac{\kappa_{r+1,s}}{Nq^{r+s-1}}\mbb{E}\sum_{k}\partial^{r}_{i,k}\bar{\partial}^{s}_{i,k}\bigl(\mc{G}_{\hat{k}}\mc{G}_{i}^{p-1}\bar{\mc{G}}_{i}^{p}\bigr)&\prec\frac{1}{q^{r+s-1}}\cdot\sum_{a=1}^{(r+s)\wedge2p}\left(\frac{1}{\eta^{m/2+1}}+\Phi\right)^{a}\mbb{E}|\mc{G}_{i}|^{2p-a}\\
    &\prec\sum_{a=1}^{(r+s)\wedge2p}\left(\frac{1}{\sqrt{q}\eta^{m/2+1}}+\frac{\Phi}{\sqrt{q}}\right)^{a}\mbb{E}|\mc{G}_{i}|^{2p-a}.
\end{align*}
For the first derivative, we observe that
\begin{align*}
    \bar{\partial}_{i,k}\mc{G}_{i}+G_{i,\hat{k}}\mc{G}_{i}&=-\sum_{n=1}^{m}\bigl((GF)^{n}G\bigr)_{i,\hat{k}}\sum_{j}\bigl((GF)^{m-n}G\bigr)_{i,j}\\
    &\prec\frac{1}{\eta^{m/2+1}},
\end{align*}
by the induction hypothesis. Therefore we have
\begin{align*}
    \frac{1}{N}\mbb{E}\sum_{k}\mc{G}_{\hat{k}}\bar{\partial}_{i,k}\bigl(\mc{G}_{i}\bigr)\mc{G}^{p-2}\bar{\mc{G}}_{i}^{p}+\frac{1}{N}\mbb{E}\sum_{k}G_{i,\hat{k}}\mc{G}_{\hat{k}}\mc{G}_{i}^{p-1}\bar{\mc{G}}_{i}^{p}&\prec\frac{\Phi}{\eta^{m/2+1}}\mbb{E}|\mc{G}_{i}|^{2p-2}.
\end{align*}
For the second term on the left hand side we use Cauchy--Schwarz and the a priori bound $|\mc{G}_{\hat{k}}|\prec\Phi$:
\begin{align*}
    \frac{1}{N}\mbb{E}\sum_{k}|G_{i,\hat{k}}|\mc{G}_{\hat{k}}||\mc{G}_{i}|^{2p-1}&\prec\frac{\Phi}{\sqrt{N\eta}}\mbb{E}|\mc{G}_{i}|^{2p-1}.
\end{align*}
Thus we find
\begin{align*}
    \frac{1}{N}\mbb{E}\sum_{k}\mc{G}_{\hat{k}}\bar{\partial}_{i,k}\bigl(\mc{G}_{i}^{p-1}\bar{\mc{G}}_{i}^{p}\bigr)&\prec\left(\frac{1}{\eta^{m/2+1}}+\frac{\Phi}{\sqrt{N\eta}}\right)\mbb{E}|\mc{G}_{i}|^{2p-1}+\frac{\Phi}{\eta^{m/2+1}}\mbb{E}|\mc{G}_{i}|^{2p-2}.
\end{align*}
Combined with the estimate for the sum over $r+s\geq2$ we have
\begin{align*}
    \mbb{E}|\mc{G}_{i}|^{2p}&\prec\left(\frac{1}{\eta^{m/2+1}}+\frac{\Phi}{\sqrt{N\eta}}\right)\mbb{E}|\mc{G}_{i}|^{2p-1}+\frac{\Phi}{\eta^{m/2+1}}\mbb{E}|\mc{G}_{i}|^{2p-2}+\sum_{a=1}^{2p}\left(\frac{1}{\sqrt{q}\eta^{m/2+1}}+\frac{\Phi}{\sqrt{q}}\right)^{a}\mbb{E}|\mc{G}_{i}|^{2p-a},
\end{align*}
which implies that
\begin{align*}
    \mc{G}_{*}&\prec\frac{1}{\eta^{m/2+1}}+\sqrt{\frac{\Phi}{\eta^{m/2+1}}}+\left(\frac{1}{\sqrt{N\eta}}+\frac{1}{\sqrt{q}}\right)\Phi,
\end{align*}
whenever $\mc{G}_{*}\prec\Phi$. Iterating this a finite number of times (since we can take $\Phi=N^{C}$ for some constant $C>0$) we obtain \eqref{eq:G_*bound}. 

For the base case we have a stronger bound
\begin{align}
    \sum_{j}G_{ij}&\prec\left(\frac{1}{\sqrt{N\eta}}+\frac{1}{q}\right)\frac{1}{\eta},
\end{align}
which is proven in a similar (and simpler way). The contribution of the second cumulant is bounded by Cauchy--Schwarz and the fact that $|G_{\mathfrak{i},\mathfrak{j}}|\prec1$, e.g.
\begin{align*}
    \frac{1}{N}\sum_{j,k}G_{\hat{k},j}\bar{\partial}_{ik}\left(\sum_{l}G_{il}\right)^{2p-1}&=-\frac{1}{N}\sum_{j,k}G_{\hat{k},j}G_{i,\hat{k}}G_{i,l}\left(\sum_{m}G_{im}\right)^{2p-2}\\
    &=-\frac{1}{2N}\sum_{j}(G(E_{+}-E_{-})G)_{i,j}\left(\sum_{l}G_{i,l}\right)^{2p-1}\\
    &\lesssim\frac{1}{N^{1/2}\eta^{3/2}}\Bigl|\sum_{j}G_{i,j}\Bigr|^{2p-1}.
\end{align*}
The contribution of higher order cumulants is bounded using the extra factors of $q^{-1}$.
\end{proof}

We are now in a position to prove Lemma \ref{lem:isoGFT}. In summary, applying It\^{o}'s lemma and a cumulant expansion, we obtain an expression for $\partial_{t}\mbb{E}|S|^{2p}$ as a sum of derivatives of $|S|^{2p}$ with respect to matrix entries. We reorganise this into terms consisting of a product of a polynomial in the sense of Definition \ref{def:polynomials} and a lower power of $S$. Whenever this polynomial is in $\mc{P}^{(1)}\cup\mc{P}^{(2)}$ we can use Lemmas \ref{lem:polynomialBound2} and \ref{lem:polynomialBound3}, which in most cases is sufficient to obtain the desired bound. The remaining terms are reduced to polynomials in $\mc{P}^{(1)}\cup\mc{P}^{(2)}$ by an iterative cumulant expansion similar to that in the proof of Lemma \ref{lem:polynomialBound3}.
\begin{proof}[Proof of Lemma \ref{lem:isoGFT}]
We start from It\^{o}'s lemma and the cumulant expansion: for any $D>0$, there is an $L\equiv L(D,p,m)$ such that
\begin{align}
    \frac{\diff}{\diff t}\mbb{E}|S|^{2p}&=\sum_{r+s=2}^{L}\frac{\kappa_{r+1,s}}{Nq^{r+s-1}}\sum_{i_{1},i_{2}=1}^{N}\mbb{E}(\partial_{i_{1},i_{2}}^{r+1}\bar{\partial}_{i_{1},i_{2}}^{s}+\partial_{i_{1},i_{2}}^{r}\bar{\partial}_{i_{1},i_{2}}^{s+1})|S|^{2p}+O(N^{-D}),
\end{align}
where we have assumed for ease of notation that $\kappa_{r+1,s}=\kappa_{r,s+1}$. Note that the sum starts at $r+s=2$ since the second cumulant is preserved by the matrix Ornstein--Uhlenbeck process. The distinction between $S$ and $\bar{S}$ is not important and so we will ignore it in the following, i.e. we will write $S^{a}\bar{S}^{b}$ as $S^{a+b}$. As in the proof of Lemma \ref{lem:polynomialBound3}, we will assume that $B_{i}=F$ and drop subscripts from resolvents $G_{j}$; the general case follows in exactly the same way.

Using the basic identities
\begin{align*}
    \partial_{ij}((GF)^{n}G)_{\mathfrak{i},\mathfrak{j}}&=-\sum_{l=0}^{n}((GF)^{n-l}G)_{\mathfrak{i},i}((GF)^{l}G)_{\hat{\jmath},\mathfrak{j}},\\
    \bar{\partial}_{ij}((GF)^{n}G)_{\mathfrak{i},\mathfrak{j}}&=-\sum_{l=0}^{n}((GF)^{n-l}G)_{\mathfrak{i},\hat{\jmath}}((GF)^{l}G)_{i,\mathfrak{j}},
\end{align*}
and the notation of Definition \ref{def:polynomials}, for a given $r+s\geq2$ and $d\in[r+s+1]$ we can write a term in the cumulant expansion in the form
\begin{align}
    \frac{1}{Nq^{r+s-1}}\sum_{i_{1},i_{2}}P_{x}(\mbf{n},\mathfrak{i},\mathfrak{j})Q_{y}(\mathfrak{k},\mathfrak{l})S^{2p-d},\label{eq:cumulantTerm}
\end{align}
with $y\leq d$, $x+y=r+s+2$ and $|\mbf{n}|=m(d-y)$.

Observe that there are at least $d$ off-diagonal entries of $(GF)^{n}G$ with $n\leq m/2$ and so all polynomials corresponding to $d\geq2$ belong to $\mc{P}^{(1)}\cup\mc{P}^{(2)}$. Thus we can use Lemmas \ref{lem:polynomialBound2} and \ref{lem:polynomialBound3} with $b=2$ and $u=r+s-1$ to obtain the bound
\begin{align*}
    \frac{1}{q\eta}\left(\frac{1_{\{m>M/2\}}+(\mc{E}^{iso}_{\eta,q})^{2p(r+s+|\mc{O}_{x}|-4)/(d-y)}}{\eta^{mp}}+\Omega^{iso}_{t}(m,p)\right).
\end{align*}
Since $|\mc{O}_{x}|\geq2d-y$ and $d\geq2$, we have $r+s+|\mc{O}_{x}|-4\geq d-y$ and the desired bound follows.

Now let $d=1$; in this case all derivatives act on a single copy of $S$. If the resulting polynomial contains a factor of $(GF)^{m}G$ (i.e. $y=1$) and belongs to $\mc{P}^{(2)}$ (note that it is not possible to belong to $\mc{P}^{(1)}$), then we use Lemma \ref{lem:polynomialBound3} with $d=y=1$. If $y=0$, then we need to gain a factor of $\mc{E}^{iso}_{\eta,q}$ so Lemma \ref{lem:polynomialBound2} yields the desired bound only if there are at least three off-diagonal entries of $(GF)^{n}G$ with $n\leq m-1$ or $r+s\geq3$ (i.e. at least four derivatives). We are left with estimating two kinds of term:
\begin{enumerate}[label=\roman*)]
\item those with $y=1$ and exactly one off-diagonal entry of $G$; 
\item those with $y=0$ and exactly four entries with summation indices, two of which are off-diagonal.
\end{enumerate}

\paragraph{Case i)}
A general term satisfying i) has the form
\begin{align}
    \frac{1}{Nq^{r_{1}+s_{1}-1}}\sum_{i_{1},i_{2}}G_{\mathfrak{i},\mathfrak{j}}G_{\mathfrak{l}_{1},\mathfrak{l}_{1}}\cdots G_{\mathfrak{l}_{x},\mathfrak{l}_{x}}((GF)^{m}G)_{\mathfrak{k},\mathfrak{l}}S^{2p-1},\label{eq:casei}
\end{align}
where $x=r_{1}+s_{1}$, $\mathfrak{l}_{j}\in\{i_{1},\hat{\imath}_{2}\}$ and $|\{\mathfrak{i},\mathfrak{j}\}\cap\{i_{1},\hat{\imath}_{2}\}|=|\{\mathfrak{k},\mathfrak{l}\}\cap\{i_{1},\hat{\imath}_{2}\}|=1$. Our approach to these terms will be to introduce a new summation index in $G_{\mathfrak{i},\mathfrak{j}}$ through a recursive cumulant expansion as in Lemma \ref{lem:polynomialBound3}. 

For concreteness we assume that $\mathfrak{i}=i_{1}$ and $\mathfrak{j}=j\in[N]$. By \eqref{eq:self2} we have
\begin{align*}
    G_{i_{1},j}&=\hat{m}\delta_{i_{1},j}-\hat{m}(\underline{GW})_{i_{1},j}+\bar{z}\hat{u}(\underline{GW})_{i_{1},\hat{\jmath}}.
\end{align*}
The first term causes the sum over $i_{1}$ to collapse, leading to the bound
\begin{align*}
    \frac{1}{Nq^{r_{1}+s_{1}-1}}\sum_{i_{2}}|((GF)^{m}G)_{\mathfrak{k},\mathfrak{l}}|S|^{2p-1}\prec\frac{\Omega^{iso}_{t}(m,p)}{q^{r_{1}+s_{1}-1}}.
\end{align*}

The procedure to bound the contribution of underlined terms is similar to that in Lemma \ref{lem:polynomialBound3}. Consider the contribution of $\hat{m}(\underline{GW})_{i_{1},j}$: by a cumulant expansion, this is equal to
\begin{align*}
    &\frac{1}{N^{2}q^{r_{1}+s_{1}-1}}\mbb{E}\sum_{i_{1},i_{2},i_{3}}G_{i_{1},\hat{\imath}_{3}}\partial_{j,i_{3}}\hat{m}G_{\mathfrak{l}_{1},\mathfrak{l}_{1}}\cdots G_{\mathfrak{l}_{x},\mathfrak{l}_{x}}((GF)^{m}G)_{\mathfrak{k},\mathfrak{l}}S^{2p-1}\\
    &+\frac{1}{N^{2}}\sum_{r_{2}+s_{2}=2}^{L}\frac{\kappa_{r_{2},s_{2}+1}}{q^{|\mbf{r}|+|\mbf{s}|-2}}\mbb{E}\sum_{i_{1},i_{2},i_{3}}\partial_{j,i_{3}}^{r}\bar{\partial}_{j,i_{3}}^{s}\hat{m} G_{i_{1},\hat{\imath}_{3}}G_{\mathfrak{l}_{1},\mathfrak{l}_{1}}\cdots G_{\mathfrak{l}_{x},\mathfrak{l}_{x}}((GF)^{m}G)_{\mathfrak{k},\mathfrak{l}}S^{2p-1}\\
    &+O(N^{-D}).
\end{align*}
The derivative in the first line creates an entry with left index $\hat{\imath}_{3}$ and so the sum over $i_{3}$ can be bounded by Cauchy--Schwarz and the Ward identity. For the higher order cumulants, we observe that any derivative acting on entries other than $G_{i_{1},\hat{\imath}_{3}}$ and $((GF)^{m}G)_{\mathfrak{k},\mathfrak{l}}$ necessarily creates at least one off-diagonal entry of $(GF)^{n}G$ with $n\leq m/2$ and a polynomial in $\mc{P}^{(1)}\cup\mc{P}^{(2)}$, which can be bounded by Lemmas \ref{lem:polynomialBound2} and \ref{lem:polynomialBound3}. We treat the remaining terms, which are of the same form as \eqref{eq:casei} with additional factors of $q^{-1}$, by repeating this procedure.

\paragraph{Case ii)}
A general term satisfying ii) has the form
\begin{align*}
    \frac{1}{Nq}\sum_{i_{1},i_{2}}((GF)^{n_{1}}G)_{\mathfrak{x},\mathfrak{j}}((GF)^{n_{2}}G)_{i_{1},i_{1}}((GF)^{n_{3}}G)_{\hat{\imath}_{2},\hat{\imath}_{2}}((GF)^{n_{4}}G)_{\mathfrak{k},\mathfrak{y}}S^{2p-1},
\end{align*}
where $n_{j}\in[m-1]$, $\sum_{j}n_{j}=m$ and $\mathfrak{j},\mathfrak{k}\in\{i_{1},\hat{\imath}_{2}\}$ such that $\mathfrak{j}\neq\mathfrak{k}$. Thus at least one of $n_{1}$ and $n_{4}$ is at most $\lfloor m/2-2\rfloor$. For concreteness, we assume that $n_{1}\leq\lfloor m/2-2\rfloor$ and $\mathfrak{j}=i_{1}$. Compared with case i), we replace the diagonal entry $((GF)^{n_{2}}G)_{i_{1},i_{1}}$ in order to apply the fluctuation averaging bound from Lemma \ref{lem:fluctuationAveraging} to $((GF)^{n_{1}}G)_{\mathfrak{x},i_{1}}$, which we can do since $n_{1}\leq \lfloor m/2-2\rfloor$.

From \eqref{eq:self4} we have 
\begin{align}
    ((GF)^{n_{2}}G)_{i_{1},i_{1}}&=-\bar{z}\hat{u}((GF)^{n_{2}-1}G)_{i_{1},i_{1}}\nonumber\\
    &+\frac{\hat{m}}{2}\sum_{l=0}^{n_{2}-1}\Tr{(GF)^{n_{2}-l}G(E_{+}-E_{-})}((GF)^{l}G)_{i_{1},i_{1}}\nonumber\\
    &-\frac{\bar{z}\hat{u}}{2}\sum_{l=0}^{n_{2}-1}\Tr{(GF)^{n_{2}-l}G(E_{+}+E_{-})}((GF)^{l}G)_{i_{1},\hat{\imath}_{1}}\nonumber\\
    &-\bigl(\underline{(GF)^{n_{2}}GW}\hat{M}\bigr)_{i_{1},i_{1}}.\label{eq:replacement}
\end{align}
The resolvent entries in the non-underlined terms have at most $n_{2}-1$ factors of $F$, so we can iterate this identity until all matrix entries are either entries of $\hat{M}$ (multiplied by products of traces of resolvent chains) or entries of underlined terms. A general non-underlined term is of the form
\begin{align}
    (-\bar{z}\hat{u})^{\alpha}\hat{m}^{\beta+1}\Tr{(GF)^{l_{1}}GE_{j_{1}}}\cdots\Tr{(GF)^{l_{\beta}}GE_{j_{\beta}}},
\end{align}
where $\alpha+\beta=n_{2}$, $l_{j}\geq1$ and $\sum_{j}l_{j}\leq n_{2}$. A general underlined term is of the form
\begin{align}
    (-\bar{z}\hat{u})^{\alpha}\hat{m}^{\beta}\Tr{(GF)^{l_{1}}GE_{j_{1}}}\cdots\Tr{(GF)^{l_{\beta}}GE_{j_{\beta}}}(\underline{(GF)^{l_{\beta+1}}GW})_{i_{1},i_{1}},
\end{align}
where $l_{j}$ is as before but now $\alpha+\beta\leq n_{2}-1$. We can estimate the products of traces using $\mc{L}^{iso}_{z}(l,\delta,\tau)$ for $l=1,...,m-1$ and the deterministic bounds in Lemma \ref{lem:deterministicBounds}:
\begin{align}
    |(-\bar{z}\hat{u})^{\alpha}\hat{m}^{\beta+1}\Tr{(GF)^{l_{1}}GE_{j_{1}}}\cdots\Tr{(GF)^{l_{\beta}}GE_{j_{\beta}}}|&\prec\frac{1}{\eta^{\frac{1}{2}\sum_{u=1}^{\beta}l_{u}}}.
\end{align}
The non-underlined terms are independent of the summation indices and so can be brought outside the sum, leaving a mutliple of the sum
\begin{align*}
    \frac{1}{Nq}\sum_{i_{1},i_{2}}((GF)^{n_{1}}G)_{\mathfrak{x},\mathfrak{j}}((GF)^{n_{3}}G)_{\hat{\imath}_{2},\hat{\imath}_{2}}((GF)^{n_{4}}G)_{\mathfrak{k},\mathfrak{y}}&\prec\frac{1}{q\eta}\cdot\frac{1}{\sqrt{N\eta}}\cdot\frac{1}{\eta^{(n_{1}+n_{4})/2}},
\end{align*}
where we have used Cauchy--Schwarz and the Ward identity for the sum over $i_{2}$ and Lemma \ref{lem:fluctuationAveraging} for the sum over $i_{1}$ since $\mathfrak{j}\neq\mathfrak{k}$. The underlined terms are bounded in exactly the same way as in case i).
\end{proof}

\subsection{Proof of Proposition \ref{thm:multiLL}}
Recall that $q=N^{\epsilon}$. Let $M\in\mbb{N}$ be even. The assumptions in Lemmas \ref{lem:isoGFT} and \ref{lem:avGFT} dictate that we should prove $\mc{L}^{iso}_{z}(l,\delta,\tau)$ before $\mc{L}^{av}_{z}(l,\delta,\tau)$. Let $\tau_{1}=1-\epsilon$ and $\delta_{1}>0$ to be specified later and set $\mc{D}_{1}:=\mc{D}(\delta_{1},\tau_{1})$. Observe that any $w\in\mc{D}_{1}$ satisfies $q|\Im w|\geq1$. From the reversed dynamics in \eqref{eq:reverse1}, we have
\begin{align}
    w_{-t}&=w+\frac{tm_{z}(w)}{2}+O(t^{2}).\label{eq:reverse2}
\end{align}
We also have $|\Im m_{z}(w)|>c>0$ when $|z|<1-c'$ and $|\Re w|<\delta<2$. Thus, for any $w\in\mc{D}_{1}$ and $z\in\mbb{D}$, we can find a $\delta'>0$, $w_{0}\in\mc{D}(\delta',1)$, $z_{0}\in\mbb{C}$ and a $t<T(w_{0})$ such that $w_{t}=w$ and $z_{t}=z$. By the global law in Lemma \ref{lem:global}, the assumptions in Proposition \ref{prop:zig} are satisfied and thus the desired local law holds at time $t$, when the matrix has a Gaussian component of variance $t$. Removing the Gaussian component using Lemma \ref{lem:isoGFT} and Gr\"{o}nwall's inequality, we extend the local law to $\mc{D}_{1}$ and thus prove $\mc{L}^{iso}_{z}(1,\delta_{1},\tau_{1})$. We repeat these steps to first prove $\mc{L}^{iso}_{z}(m,\delta_{1},\tau_{1})$ and then $\mc{L}^{av}_{z}(m,\delta_{1},\tau_{1})$ (for which we use Lemma \ref{lem:avGFT}) for $m=1,...,M$.

Now we set $\tau_{2}=\tau_{1}-\epsilon$ and $\mc{D}_{2}:=\mc{D}(\delta_{2},\tau_{2})$ with $\delta_{2}\in(0,\delta_{1})$ to be chosen later. For $|z|<1-\delta$ and $w\in\mc{D}_{2}$, we have $|\Im m_{z}(w)|>c>0$. From this and \eqref{eq:reverse1} it follows that we can reach $|\Im w|\geq N^{-1+\tau_{2}}$ from a point $w_{0}\in\mc{D}_{1}$ after time
\begin{align}
    t&\lesssim N^{-1+\tau_{1}}=N^{-\epsilon},\label{eq:t}
\end{align}
with $t<T(w_{0})$. Thus, for $\delta_{2}=\delta_{1}-O(N^{-\epsilon})$ and any $w\in\mc{D}_{2}$, we can find a $w_{0}\in\mc{D}_{1}$ and $z_{0}\in\mbb{D}$ such that $w_{t}=w$ and $z_{t}=z$. Applying Proposition \ref{prop:zig} with the local laws $\mc{L}^{av/iso}_{z}(m,\delta_{1},\tau_{1})$ for $m=1,...,M$ as input, we extend these local laws to $\mc{D}_{2}$ for the Gaussian-divisible $X_{t}$. Since $t\lesssim q|\Im w|$ for $w\in\mc{D}_{2}$ by \eqref{eq:t}, we can use Lemmas \ref{lem:isoGFT} and \ref{lem:avGFT} and Gr\"{o}nwall's inequality to remove the Gaussian component. Repeating these steps we obtain $\mc{L}^{iso/av}_{z}(m,\delta_{2},\tau_{2})$ for $m=1,...,M$.

Continuing in this way, at the $j$th step we obtain $\mc{L}^{iso/av}_{z}(m,\delta_{j},\tau_{j})$, where $\tau_{j}=1-j\epsilon$ and $\delta_{j}=\delta_{1}-O(jN^{-\epsilon})$. Choosing $\delta_{1}$ appropriately, after $O(1/\epsilon)$ steps we can extend the local laws to $\mc{D}(\delta,\tau)$ and hence prove $\mc{L}^{iso}_{z}(m,\delta,\tau)$ for $m=1,...,M$.

\section{Proof of Theorem \ref{thm1}}\label{sec4}
The proof is a standard application of Girko's formula and the Lindeberg method, so we only give a sketch. Girko's formula states that for $f\in C^{2}$ we have
\begin{align}
    \sum_{j}f(\lambda_{j})&=\frac{1}{2\pi}\int_{\mbb{C}}\Delta f(z)\log|\det(X-z)|\diff^{2}z.
\end{align}
We introduce the resolvent $G_{z}$ by
\begin{align}
    \log|\det(X-z)|&=\int_{0}^{N^{L}}\Im\tr G_{z}(i\eta)\diff\eta+O_{\prec}(N^{-D}),
\end{align}
for some $L>0$. For any $\xi>0$, by arguing as in \cite[Section 5.2]{he_edge_2023}, we can restrict the $\eta$-integral to the region $[N^{-1-\xi},T]$ using the least singular value bound in \cite[Theorem 2.5]{tao_random_2008} (applied to the matrix $Y=N^{1/2}X$) for the region $\eta<N^{-B}$ for some large $B$ and the bound in \cite[Theorem 2.2]{che_universality_2019} for the region $N^{-B}<\eta<N^{-1-\xi}$. Note that the latter theorem is stated for real matrices and $z=0$ but the proof generalises to complex matrices and any fixed $z\in\mbb{C}$ (see the Remark on pg 9 in \cite{che_universality_2019}). We deduce that comparing local statistics of eigenvalues of two random matrices $X$ and $Y$ amounts to comparing expectation values of $\Im\tr G_{z}(i\eta)$ for $\eta>N^{-1-\xi}$. This is done by replacing entries of $X$ by $Y$ one by one. Let the entries of $X^{(0)}$ be those of $X$ except at entry $(i,j)$ where it is zero, and denote the resolvent of its Hermitisation by $G^{(0)}_{z}$. Then we have
\begin{align}
    G&=\sum_{r=0}^{p-1}X_{ij}^{r}(G^{(0)}\mbf{e}_{i}\mbf{e}_{\hat{\jmath}}^{*})^{r}G^{(0)}+X_{ij}^{p}(G^{(0)}\mbf{e}_{i}\mbf{e}_{\hat{\jmath}}^{*})^{p}G.
\end{align}
Now fix an order on pairs of indices and let $X^{(\gamma)}$ denote the matrix whose first $\gamma$ entries are those of $X$ and the remaining $N^{2}-\gamma$ entries are those of $\sqrt{1-t}X+\sqrt{t}Y$, where $Y$ is an independent Ginibre matrix. We expand $G^{(\gamma)}$ and $G^{(\gamma-1)}$ in terms of $G^{(0)}$ and take the difference of the two expressions. Since the first two moments of the entries match and higher order moments differ by at most $tN^{-1}q^{-r+2}$, after expanding to a sufficiently high power $p$ (depending on $\epsilon=\frac{\log q}{\log N}$) we find
\begin{align}
    \left|\mbb{E}\int_{N^{-1-\xi}}^{N^{L}}\tr G^{(\gamma)}_{z}(i\eta)\diff\eta-\mbb{E}\int_{N^{-1-\xi}}^{N^{L}}\tr G^{(\gamma-1)}_{z}(i\eta)\diff\eta\right|&\lesssim\frac{N^{3\xi}\log(N)t}{Nq},
\end{align}
by the local law $|G_{\mathfrak{i},\mathfrak{j}}|\prec1\vee(N\eta)^{-1}$. Summing over all $N^{2}$ entries, the total difference is bounded by
\begin{align}
    \frac{N^{4\xi}Nt}{q}.
\end{align}
Choosing $\xi<\epsilon/8$ and $t=N^{-1+\epsilon/2}$, the above bound becomes $N^{-\epsilon/2}$. Thus the local statistics of $X$ and $\sqrt{1-t}X+\sqrt{t}Y$ are equal up to $O(N^{-\epsilon/2})$. By Theorem \ref{thm3}, $X\in\mc{M}_{N}(z,n_{\epsilon/4,k},\epsilon/4)$ and $\sigma_{z}=1+o(1)$ with very high probability, where $\sigma_{z}$ is defined in \eqref{eq:sigma}. By Theorem \ref{thm2} the local statistics of $\sqrt{1-t}X+\sqrt{t}Y$ are given by those of the Ginibre ensemble.

\paragraph{Acknowledgements}
This work was supported by the Royal Society, grant number \\RF/ERE210051.

\appendix
\section{Proof of Lemma \ref{lem:FBound}}\label{app:FBound}
We start with the bound on $K_{i}$. By \cite[Lemma 7.1]{maltsev_bulk_2024} we have the formula
\begin{align*}
    K_{i}(z_{i})&=e^{\frac{N}{t}(\eta^{(i-1)}_{z_{i}})^{2}}\det^{-1}\bigl((\eta^{(i-1)}_{z_{i}})^{2}+|X^{(i-1)}_{z_{i}}|^{2}\bigr)\\
    &\times\int_{-\infty}^{\infty}e^{\frac{iNp}{t}}\det^{-1}\bigl(1+ipH^{(i-1)}_{z_{i}}(i\eta^{(i-1)}_{z_{i}})\bigr)\diff p.
\end{align*}
By interlacing and the fact that $\eta^{(i)}_{z_{i}}\simeq t$, we have
\begin{align*}
    \Tr{(H^{(i-1)}_{z_{i}}(\eta^{(i-1)}_{z_{i}}))^{2}}&=\Tr{H^{2}_{z_{i}}(\eta^{(i-1)}_{z_{i}})}+O\left(\frac{1}{Nt^{4}}\right)\\
    &\gtrsim\frac{1}{t^{3}},
\end{align*}
and hence
\begin{align*}
    \int_{-\infty}^{\infty}e^{\frac{iNp}{t}}\det^{-1}\bigl(1+ipH^{(i-1)}_{z_{i}}(i\eta^{(i-1)}_{z_{i}})\bigr)\diff p&\lesssim\sqrt{\frac{t^{3}}{N}}.
\end{align*}

Now recall the definition of $F_{N-k}$:
\begin{align*}
    F_{N-k}(\mbf{z},X^{(k)})&=\mbb{E}_{N-k}\prod_{i=1}^{k}\Bigl|\det\Bigl(X^{(k)}_{z_{i}}+\sqrt{\frac{Nt}{N-k}}Y_{N-k}\Bigr)\Bigr|^{2}.
\end{align*}
By \cite[Lemma 6.1]{maltsev_bulk_2024} we have the formula
\begin{align*}
    F_{N-k}(\mbf{z},X^{(k)})&=\left(\frac{N}{\pi t}\right)^{k^{2}}\int_{\mbb{C}^{k\times k}}e^{-\frac{N}{t}\tr|Q|^{2}}D(Q)\diff Q,
\end{align*}
where
\begin{align*}
    D(Q)&:=\det\begin{pmatrix}-iQ\otimes1_{N-k}&1_{k}\otimes X^{(k)}-\mbf{z}\otimes1_{N-k}\\1_{k}\otimes X^{(k)*}-\mbf{z}^{*}\otimes1_{N-k}&-iQ^{*}\otimes1_{N-k}\end{pmatrix}.
\end{align*}
By Fischer's inequality we have
\begin{align*}
    |D(Q)|&\leq\prod_{i=1}^{k}\det^{1/2}\bigl((QQ^{*})_{ii}+|X^{(k)}_{z_{i}}|^{2}\bigr)\det^{1/2}\bigl((Q^{*}Q)_{ii}+|X^{(k)}_{z_{i}}|^{2}\bigr).
\end{align*}
Inserting this into the expression for $F_{N-k}$ and using Cauchy--Schwarz we obtain
\begin{align*}
    F_{N-k}(\mbf{z},X^{(k)})&\leq\left(\frac{N}{\pi t}\right)^{k^{2}}\int_{\mbb{C}^{k\times k}}e^{-\frac{N}{t}\tr|Q|^{2}}\prod_{i=1}^{k}\det\bigl((Q^{*}Q)_{ii}+|X^{(k)}_{z_{i}}|^{2}\bigr)\diff Q.
\end{align*}
Using spherical coordinates for the columns of $Q$ we have
\begin{align*}
    F_{N-k}(\mbf{z},X^{(k)})&\leq\left(\frac{N}{\pi t}\right)^{k^{2}}\left(\frac{\text{Vol}(S^{k-1})}{2}\right)^{k}\prod_{i=1}^{k}\int_{0}^{\infty}x^{k-1}e^{-N\phi^{(k)}_{i}(x)}\diff x,
\end{align*}
where
\begin{align*}
    \phi^{(j)}_{i}(x)&:=\frac{x}{t}-\Tr{\log(x+|X^{(j)}_{z_{i}}|^{2})}.
\end{align*}
By interlacing, we can deduce that
\begin{align*}
    \phi^{(j)}_{i}(x)-\phi^{(j)}_{i}\bigl((\eta^{(j-1)}_{z_{i}})^{2}\bigr)&\gtrsim\frac{\bigl(x-(\eta^{(j-1)}_{z_{i}})^{2}\bigr)^{2}}{t^{3}},
\end{align*}
uniformly in $x>0$, and hence by Laplace's method we obtain
\begin{align*}
    F_{N-k}(\mbf{z},X^{(k)})&\lesssim(Nt)^{k(k-1/2)}\prod_{i=1}^{k}e^{-\frac{N}{t}(\eta^{(i-1)}_{z_{i}})^{2}}\det\bigl((\eta^{(i-1)}_{z_{i}})^{2}+|X^{(i-1)}_{z_{i}}|^{2}\bigr),
\end{align*}
as desired.

\section{Proof of Proposition \ref{thm:singleLL}}\label{app:singleLL}
The proof is a simpler version of the arguments used to prove Theorem \ref{thm:multiLL} and so we only give a sketch. Our input will be the weak local law from \cite[Theorem 3.1]{he_edge_2023}:
\begin{align}
    \max_{\mathfrak{i},\mathfrak{j}\in[2N]}|\bigl(G_{z}(w)-M_{z}(w)\bigr)_{\mathfrak{i},\mathfrak{j}}|&\prec\frac{1}{(N\eta)^{1/6}}+\frac{1}{q^{1/3}},\label{eq:weakLL}
\end{align}
uniformly in 
\begin{align}
    \mbf{D}_{\delta}&:=\{(z,w)\in\mbb{C}^{2}:|z|\leq\delta^{-1},\,|\Re w|\leq\delta^{-2},\,N^{-1+\delta}\leq\eta\leq\delta^{-1}\}.
\end{align}
With the notation of Section \ref{sec:zig} we have
\begin{align*}
    \partial S^{av}_{t}(w_{1},E_{+})&=\frac{1}{2}S^{av}_{t}(w_{1},E_{+})+\Tr{G^{2}_{t}}S^{av}_{t}(w_{1},E_{+})\\
    &+\frac{1}{\sqrt{N}}\sum_{i,j=1}^{N}\partial_{ij}S^{av}_{t}(w_{1},E_{+})\diff B_{ij}\\
    &+\frac{1}{\sqrt{N}}\sum_{i,j=1}^{N}\bar{\partial}_{ij}S^{av}_{t}(w_{1},E_{+})\diff\bar{B}_{ij}.
\end{align*}
By \eqref{eq:weakLL} we have 
\begin{align*}
    |\Tr{G^{2}_{t}}|&\lesssim|m'_{z_{t}}(w_{t})|+\left(\frac{1}{(N\eta_{t})^{1/6}}+\frac{1}{q^{1/3}}\right)\frac{N^{\xi}}{\eta_{t}}\\
    &\lesssim\frac{1}{|1-|z_{t}||+\eta_{t}^{2/3}}+\left(\frac{1}{(N\eta_{t})^{1/6}}+\frac{1}{q^{1/3}}\right)\frac{N^{\xi}}{\eta_{t}},
\end{align*}
with very high probability. The quadratic variation of the martingale term is bounded by
\begin{align*}
    \frac{1}{N}\sum_{i,j}|\partial_{ij}\Tr{G_{t}}|^{2}&\leq\frac{1}{N^{2}}\Tr{|G_{t}|^{4}}\lesssim\frac{1}{N^{2}\eta^{3}_{t}},
\end{align*}
where the last inequality follows from \eqref{eq:weakLL}. By the BDG and Gr\"{o}nwall inequalities, we obtain
\begin{align*}
    |S^{av}_{t}(w_{1},E_{+})|&\lesssim|S^{av}_{0}(w_{1},E_{+})|+\frac{1}{N\eta_{t}},
\end{align*}
with very high probability. We can argue similarly for $S^{iso}_{t}(w_{1})$.

To remove the Gaussian component, we argue as in Section \ref{sec:zag}. For $S=S^{av}_{z}(w_{1},E_{+})$, we need to bound expressions of the form
\begin{align*}
    \frac{1}{Nq^{r+s-1}}\sum_{i,j=1}^{N}\partial_{ij}^{r+1}\bar{\partial}_{ij}^{s}|S|^{2p}.
\end{align*}
Each copy of $S$ on which a derivative acts gives a factor $(N\eta)^{-1}$: $N^{-1}$ from the normalised trace and $\eta^{-1}$ from the entry of $G^{2}$. The dominant contribution is therefore due to terms for which all derivatives act on the same copy. Among these, the largest contribution is from an even number of derivatives which result in diagonal entries, e.g.
\begin{align*}
    \frac{1}{N^{2}q^{r+s-1}}\sum_{i,j=1}^{N}|G_{i,i}|^{(r+s-1)/2}|G_{\hat{\jmath},\hat{\jmath}}|^{(r+s+1)/2}|(G^{2})_{i,i}||S^{2p-1}|&\lesssim\frac{1}{q\eta}\left(\frac{1}{q^{2p}}+\mbb{E}|S|^{2p}\right).
\end{align*}
For $S=S^{iso}_{z}(w_{1})$, we need the additional input
\begin{align*}
    \sum_{j}G_{ij}&\prec\frac{1}{\eta},
\end{align*}
from Lemma \ref{lem:fluctuationAveraging}. Note that the proof of this bound relies only on upper bounds for $\hat{m}_{z}(w)$ and $\hat{u}_{z}(w)$ (recall \eqref{eq:Mhat}) which follow from the weak local law in \eqref{eq:weakLL}. With this input we can proceed along the same lines as the proof of Lemma \ref{lem:isoGFT}.

\section{Proof of Lemma \ref{lem:deterministicBounds}}\label{app:deterministicBounds}

We recall the notation of Section 3.2: $M_{z}(w_{1},B_{1},...,w_{m})$ denotes the deterministic approximation to $G_{z}(w_{1})B_{1}\cdots G_{z}(w_{m})$ and $M_{t}(w_{1},B_{1},...,w_{m})$ the deterministic approximation to $G_{z_{t}}(w_{1,t})B_{1}\cdots G_{z_{t}}(w_{m,t})$, where $(w_{j,t},z_{t})$ is the solution of \eqref{eq:dw} with initial condition $(w_{j},z)$. We will prove the bounds dynamically using Gr\"{o}nwall's inequality. Before doing so we consider some cases that can be handled directly.

The case $m=1$ is immediate from the cubic equation for $m_{z}(w)$ in \eqref{eq:cubic}, which implies that $|m_{z}(w)|\lesssim1$ and $|u_{z}(w)|\lesssim1$. For $m=2$ and $a=1$ we have
\begin{align}
    |\Tr{M_{z}(w_{1},F,w_{2})}|&=|z|\left|\frac{u_{1}-u_{2}}{w_{1}-w_{2}}\right|\lesssim1,\label{eq:(2,1)}
\end{align}
by \eqref{eq:m}. In general, if $m>a$ we can remove each factor of $B_{j}=E_{\pm}$ by contour integration and the resolvent identity at the cost of a factor $\eta^{-1}$. If $a=1$ we reduce to $m=a+1$ and use \eqref{eq:(2,1)}, otherwise we reduce to $m=a$. Moreover, the case $m=a\leq4$ of \eqref{eq:deterministicAverageBound} follows from \cite[Lemma 4.3]{cipolloni_optimal_2024}. Thus we are left with proving \eqref{eq:deterministicAverageBound} for $m=a>4$ which we do by induction.

Assume that
\begin{align}
    |\Tr{M_{z}(w_{1},B_{1},...,w_{l})B_{l}}|&\lesssim\frac{1}{\eta^{\lfloor l/2\rfloor-1}},\label{eq:induction1}
\end{align}
uniformly in $\eta>0$ and $B_{j}\in\{F,F^{*}\}$ for $l=1,...,m-1$ and $m>4$, where $\eta=\min_{j}|\Im w_{j}|$. Let $B_{m}\in\{F,F^{*}\}$ and consider $\Tr{M_{t}(w_{1},B_{1},...,w_{m})B_{m}}$. By the argument in the proof of \cite[Lemma 4.8]{cipolloni_eigenstate_2023}, we can derive the differential equation
\begin{align}
    \partial_{t}\Tr{M_{t}(w_{1},B_{1},...,w_{m})B_{m}}&=\frac{m}{2}\Tr{M_{t}(w_{1},B_{1},...,w_{m})B_{m}}+\sum_{p<r}^{m}A_{p,r}(t),\label{eq:ODE}
\end{align}
where
\begin{align}
    A_{p,r}(t)&=\sum_{\nu=\pm}\nu\Tr{M_{t}(w_{p},B_{p},...,w_{r})E_{\nu}}\Tr{M_{t}(w_{1},B_{1},...,w_{p},E_{\nu},w_{r},B_{r},...,w_{m})B_{m}}.
\end{align}
In each trace on the right hand side a matrix $B_{j}\in\{F,F^{*}\}$ has been replaced by $E_{\nu}$. Thus we can use the induction hypothesis and the argument below \eqref{eq:(2,1)} to obtain
\begin{align*}
    |A_{p,r}(t)|&\lesssim\frac{1}{\eta^{\lfloor m/2\rfloor}_{t}}.
\end{align*}
Since $\lfloor m/2\rfloor\geq2$ for $m>4$, by \eqref{eq:integralBound} we have
\begin{align*}
    \int_{0}^{t}|A_{p,r}(s)|\diff s&\lesssim\frac{1}{\eta^{\lfloor m/2\rfloor-1}_{t}}.
\end{align*}
By Gr\"{o}nwall's inequality we obtain
\begin{align*}
    |\Tr{M_{t}(w_{1},B_{1},...,w_{m})B_{m}}|&\lesssim\frac{1}{\eta^{\lfloor m/2\rfloor-1}_{t}}+|\Tr{M_{z}(w_{1},B_{1},...,w_{m})B_{m}}|.
\end{align*}
Choosing the initial condition such that $|\Im w_{j}|\geq1$, we obtain \eqref{eq:deterministicAverageBound}.

The norm bound in \eqref{eq:deterministicNormBound} follows from the fact that the deterministic approximation is a multiple of $1_{N}$ in each block (since all $B_{i}$ are of this form) and so
\begin{align*}
    \|M_{z}(w_{1},B_{1},...,w_{m+1})\|&\leq2\max_{B_{m+1}\in\mbb{H}}|\Tr{M_{z}(w_{1},B_{1},...,w_{m})B_{m+1}}|.
\end{align*}

\section{Proof of Lemma \ref{lem:entry-wiseStoppingTime}}\label{app:entry-wiseStoppingTime}
The proof proceeds almost exactly as the proof of Lemma \ref{lem:averagedStoppingTime} after replacing $B_{m}$ with $N\mbf{e}_{\mathfrak{j}}\mbf{e}_{\mathfrak{i}}^{*}$. As in that proof, we can restrict to the case in which $\pm \Im w_{p}\Im w_{p+1}>0$ when $B_{p}=E_{\pm}$ for $p=1,...,m$. We have the SDE
\begin{align}
    \diff S^{iso}_{t}(w_{1},B_{1},...,w_{m+1})&=\frac{m+1}{2}S^{iso}_{t}(w_{1},B_{1},...,w_{m+1})\diff t+\sum_{p\leq r}^{m+1}A_{p,r}(t)\diff t\nonumber\\
    &+\frac{1}{\sqrt{N}}\sum_{i,j=1}^{N}\partial_{ij}\bigl(S^{av}_{t}(w_{1},B_{1},...,w_{m+1})\bigr)\diff B_{ij}\\
    &+\frac{1}{\sqrt{N}}\sum_{i,j=1}^{N}\bar{\partial}_{ij}\bigl(S^{av}_{t}(w_{1},B_{1},...,w_{m+1})\bigr)\diff\bar{B}_{ij},\nonumber
\end{align}
where
\begin{align}
    A_{p,p}(t)&=\Tr{G_{p,t}-M_{p,t}}\bigl(G_{1,t}B_{1}\cdots G^{2}_{p,t}B_{p}\cdots G_{m+1,t}\bigr)_{\mathfrak{i},\mathfrak{j}},
\end{align}
and
\begin{align}
    A_{p,r}(t)&=\sum_{\nu=\pm}\nu\Big[\bigl(M_{t}(w_{1},B_{1},...,w_{p},E_{\nu},w_{r},...,w_{m+1})\bigr)_{\mathfrak{i},\mathfrak{j}}S^{av}_{t}(w_{p},B_{p},...,w_{r},E_{\nu})\nonumber\\
    &+\Tr{M_{t}(w_{p},B_{p},...,w_{r})E_{\nu}}S^{iso}_{t}(w_{1},B_{1},...,w_{p},E_{\nu},w_{r},B_{r},...,w_{m+1})\nonumber\\
    &+S^{av}_{t}(w_{p},B_{p},...,w_{r},E_{\nu})S^{iso}_{t}(w_{1},B_{1},...,w_{p},E_{\nu},w_{r},B_{r},...,w_{m+1})\Big]
\end{align}

Consider $A_{p,p}$. We use the single resolvent local law for the first factor. If $m\leq M-1$, then we can directly estimate the second factor using the definiton of $\tau(\mbf{w}_{0})$ to obtain
\begin{align*}
    |A_{p,p}(t)|&\lesssim \frac{N^{\xi/2}}{N\eta_{t}}\cdot\frac{1}{\eta^{m-a/2+1}_{t}},
\end{align*}
with very high probability. If $m=M$, then by Cauchy--Schwarz we have
\begin{align*}
    |\bigl(G_{1,t}B_{1}\cdots G^{2}_{p,t}B_{p}\cdots G_{m+1,t}\bigr)_{\mathfrak{i},\mathfrak{j}}|&\leq\frac{1}{\eta_{t}}\bigl(G_{1,t}B_{1}G_{2,t}\cdots B_{p-1}\Im(G_{p,t})B_{p-1}^{*}\cdots G_{2,t}^{*}B_{1}^{*}G_{1,t}^{*}\bigr)_{\mathfrak{i},\mathfrak{i}}^{1/2}\\
    &\times\bigl(G_{m+1,t}^{*}B_{m}^{*}G_{m,t}^{*}\cdots B_{p}^{*}\Im(G_{p,t})B_{p}\cdots G_{m,t}B_{m}G_{m+1,t}\bigr)_{\mathfrak{j},\mathfrak{j}}^{1/2}.
\end{align*}
The two terms on the right hand side have $2(p-1)$ and $2(m-p+1)$ test matrices respectively. Assume that $p\leq m/2$ (if $p>m/2$ we interchange the roles of the two terms). Then the first term can be bounded using the definition of $\tau(\mbf{w}_{0})$. For the second term, we use a spectral decomposition (similar to that in the proof of Lemma \ref{lem:averagedStoppingTime}) to obtain
\begin{align*}
    \bigl(G_{m+1,t}^{*}B_{m}^{*}G_{m,t}^{*}\cdots B_{p}^{*}\Im(G_{p,t})B_{p}\cdots G_{m,t}B_{m}G_{m+1,t}\bigr)_{\mathfrak{j},\mathfrak{j}}^{1/2}&\prec \frac{N^{1/2+(m+2+a)\xi/2}}{\eta^{m-p-(a-a_{1,p})/2+1/2}_{t}},
\end{align*}
where $a_{p,r}$ is the number of $B_{j}\in\{F,F^{*}\}$ for $j=p,...,r-1$, which implies that
\begin{align*}
    |A_{p,p}(t)|&\prec\frac{N^{(m+2+a)\xi/2}}{\sqrt{N\eta_{t}}\eta^{m-a/2+1}_{t}}.
\end{align*}
Integrating over $t$ we find
\begin{align*}
    \int_{0}^{t}|A_{p,p}(s)|\diff s&\lesssim\frac{N^{(m+2+a)\xi/2}\mc{E}^{iso}_{\eta_{t},q}}{\eta^{m-a/2}_{t}}\lesssim N^{(m+1+a)\xi}\Psi^{iso}(\eta_{t},m,a).
\end{align*}

Now consider $A_{p,r}$ with $p<r$. We can treat the terms with $r-p\in[2,3,...,m-1]$ as in Lemma \ref{lem:averagedStoppingTime}, since these have lower order $S^{av/iso}_{t}$, and likewise the terms with $r=p+1$ and $a_{p,p+1}=1$ (i.e. $B_{p}\in\{F,F^{*}\}$). When $r=p+1<m+1$ and $a_{p,r}=0$ (i.e. $B_{p}=E_{\pm}$) we have
\begin{align*}
    A_{p,p+1}&=\pm\Tr{M_{t}(w_{p},B_{p},w_{p+1})B_{p}}S^{iso}_{t}(w_{1},B_{1},...,w_{m+1})+O\left(N^{(m+a)\xi}\Psi^{iso}(\eta,m,a)\right),
\end{align*}
where the $\pm$ sign is chosen according to $B_{p}=E_{\pm}$. When $p=1$ and $r=m+1$ we have
\begin{align*}
    A_{1,m+1}&=\sum_{\nu=\pm}\nu\bigl(M_{t}(w_{1},E_{\nu},w_{m+1})\bigr)_{\mathfrak{i},\mathfrak{j}}S^{av}_{t}(w_{1},B_{1},...,w_{m+1},E_{\nu})+O\left(N^{(m+a)\xi}\Psi^{iso}(\eta,m,a)\right).
\end{align*}
If $m<M$ we can bound this using the definition of $\tau(\mbf{w}_{0})$:
\begin{align*}
    |S^{av}_{t}(w_{1},B_{1},...,w_{m+1},E_{\nu})|&\leq \frac{N^{(m+1+a)\xi}\Psi^{av}(\eta_{t},m,a)}{\eta_{t}}\\
    &\leq N^{(m+1+a)\xi}\Psi^{iso}(\eta_{t},m,a).
\end{align*}
Otherwise, we use Cauchy--Schwarz and the integral representation of $|G|$ as in
\begin{align*}
    |\Tr{G(GB)^{m}}|&=|\Tr{(GB)^{m/2}G^{1/2}\cdot G^{1/2}(BG)^{m/2}}|\\
    &\leq\frac{1}{\eta}\Tr{\Im(G)B(GB)^{m/2-1}|G|B^{*}(G^{*}B^{*})^{m/2-1}}\\
    &\lesssim\frac{1}{\eta^{m-a/2}}.
\end{align*}
In either case we obtain
\begin{align*}
    |S^{av}_{t}(w_{1},B_{1},...,w_{m+1},E_{\nu})|&\lesssim\frac{N^{(m+1+a)\xi}\Psi^{iso}(\eta_{t},m,a)}{\eta^{m-a/2}_{t}}.
\end{align*}
Since $|\Tr{M_{t}(w_{1},E_{\nu},w_{m+1})}|\lesssim\eta_{t}^{-1}$, after integrating over $t$ using \eqref{eq:integralBound} we obtain
\begin{align*}
    \int_{0}^{t}|A_{p,r}(s)|\diff s&\lesssim N^{(m+1+a)\xi}\Psi(\eta_{t},m,a).
\end{align*}

Finally, we have the stochastic term whose quadratic variation is bounded by
\begin{align*}
    &\frac{1}{N\eta_{p,t}^{2}}\bigl(G_{1,t}B_{1}\cdots G_{p-1,t}B_{p-1}\Im(G_{p,t})B_{p-1}^{*}G_{p-1,t}^{*}\cdots B_{1}^{*}G_{1,t}^{*}\bigr)_{\mathfrak{i},\mathfrak{i}}\\
    &\times\bigl(G_{m+1,t}^{*}B_{m}^{*}\cdots G_{p+1,t}^{*}B_{p}^{*}\Im(G_{p,t})B_{p}G_{p+1,t}\cdots B_{m}G_{m+1,t}\bigr)_{\mathfrak{j},\mathfrak{j}}.
\end{align*}
The two terms on the right hand side contain $2(p-1)$ and $2(m-p+1)$ test matrices respectively. If $m\leq M/2$ then these can both be bounded by the definition of $\tau(\mbf{w}_{0})$. If $m>M/2$, we use a spectral decomposition for the term with the greater number of test matrices, as in the bound for $A_{p,p}$. Ultimately we find that the quadratic variation is bounded by
\begin{align*}
    \frac{1+1_{\{m>M/2\}}\cdot N^{1+(m+2+a)\xi}\eta_{t}}{N\eta^{2m-a+2}_{t}}.
\end{align*}
By the BDG inequality we find
\begin{align*}
    \left|\frac{1}{\sqrt{N}}\sum_{i,j=1}^{N}\int_{0}^{t}\partial_{ij}S^{iso}_{s}(w_{1},B_{1},...,w_{m+1})\diff s\right|&\lesssim\frac{N^{\xi}(\mc{E}^{iso}_{\eta_{t},q}+1_{\{m>M/2\}}\cdot N^{(m+2+a)\xi/2})}{\eta^{m-a/2}_{t}}\\
    &\lesssim N^{(m+1+a)\xi}\Psi^{iso}(\eta_{t},m,a),
\end{align*}
with very high probability.

Combining all these bounds we have
\begin{align*}
    \diff S^{iso}_{t}(w_{1},B_{1},...,w_{m+1})&=\phi(t)S^{iso}_{t}(w_{1},B_{1},...,w_{m+1})\diff t+R(t)\diff t,
\end{align*}
with
\begin{align*}
    \phi(t)&=\left(\frac{m+1}{2}+\sum_{p:B_{p}=E_{\pm}}\pm\Tr{M_{t}(w_{p},B_{p},w_{p+1})B_{p}}\right).
\end{align*}
Since
\begin{align*}
    \int_{0}^{t}|R(s)|\diff s&\lesssim N^{(m+1+a)\xi}\Psi^{iso}(\eta_{t},m,a),
\end{align*}
with very high probability, and
\begin{align*}
    \int_{0}^{t}|\phi(s)|\diff s&\lesssim1
\end{align*}
since $\pm \Im w_{p}\Im w_{p+1}>0$ when $B_{p}=E_{\pm}$, we conclude the proof by Gr\"{o}nwall's inequality.

\section{Proof of Lemma \ref{lem:avGFT}}\label{app:avGFT}
Recall the notation of Lemma \ref{lem:avGFT}:
\begin{align*}
    S&:=S^{av}_{t}(w_{1},B_{1},...,w_{m},B_{m})=\Tr{G_{1}B_{1}\cdots G_{m}B_{m}}-\Tr{M_{z}(w_{1},B_{1},...,w_{m})B_{m}},
\end{align*}
where $G_{j}\equiv G_{z}(w_{j})$ is the resolvent of the Hermitisation of $X(t)$ and $X(t)$ is the solution of the matrix Ornstein--Uhlenbeck process \eqref{eq:matrixOU}.

By It\^{o}'s lemma and the cumulant expansion, we have
\begin{align*}
    \frac{\diff}{\diff t}\mbb{E}|S|^{2p}&=\sum_{r+s=2}^{L}\frac{\kappa_{r+1,s}}{Nq^{r+s-1}}\sum_{i_{1},i_{2}=1}^{N}\mbb{E}(\partial^{r+1}_{i_{1},i_{2}}\bar{\partial}^{s}_{i_{1},i_{2}}+\partial^{r}_{i_{1},i_{2}}\bar{\partial}^{s+1}_{i_{1},i_{2}})|S|^{2p}+O(N^{-D}).
\end{align*}
As in the proof of Lemma \ref{lem:isoGFT}, we ignore the distinction between $S$ and $\bar{S}$ and assume that $\kappa_{r+1,s}=\kappa_{r,s+1}$.

The action of $\partial_{i_{1},i_{2}}$ on $S$ is given by
\begin{align*}
    \partial_{i_{1},i_{2}}\Tr{(GF)^{m}}&=-\frac{m}{N}\bigl((GF)^{m}G\bigr)_{\hat{\imath}_{2},i_{1}}.
\end{align*}
Thus, each copy of $S$ on which a derivative acts generates a factor of $N^{-1}$. Moreover, any term generated by an odd number of derivatives necessarily contains at least one off-diagonal entry, which contributes a factor of $1_{\{m>M/2\}}+\mc{E}^{iso}_{\eta,q}$. Let $d$ be the number of copies of $S$ on which a derivative acts. In the notation of Definition \ref{def:polynomials}, a general such term has the form
\begin{align*}
    \frac{1}{N^{d+1}q^{r+s-1}}\sum_{i_{1},i_{2}}P_{x}(\mbf{n},\mathfrak{i},\mathfrak{j})Q_{y}(\mathfrak{k},\mathfrak{l})S^{2p-d},
\end{align*}
where $y\leq d$ and $|\mbf{n}|=m(d-y)$. We bound all terms in $P_{x}$ and $Q_{y}$ using the assumption $\mc{L}^{iso}_{z}(l,\delta,\tau)$ for $l=1,...,m$:
\begin{align*}
    |P_{x}(\mbf{n},\mathfrak{i},\mathfrak{j})Q_{y}(\mathfrak{k},\mathfrak{l})|&\prec\frac{1}{\eta^{dm/2}}.
\end{align*}
From this we obtain
\begin{align*}
    \frac{1}{N^{d+1}q^{r+s-1}}\sum_{i_{1},i_{2}}P_{x}(\mbf{n},\mathfrak{i},\mathfrak{j})Q_{y}(\mathfrak{k},\mathfrak{l})S^{2p-d}&\prec\frac{N^{2}}{N^{d+1}q^{r+s-1}\eta^{dm/2}}\cdot|S|^{2p-d}\\
    &\prec\frac{1}{q\eta}\cdot\frac{1}{(N\eta)^{d-1}q^{r+s-2}\eta^{d(m/2-1)}}\cdot|S|^{2p-d}\\
    &\prec\frac{1}{q\eta}\left(\frac{\mc{E}^{av}_{\eta,q}}{\eta^{m/2-1}}\right)^{d}|S|^{2p-d}\\
    &\prec\frac{1}{q\eta}\left(\left(\frac{\mc{E}^{av}_{\eta,q}}{\eta^{m/2-1}}\right)^{2p}+|S|^{2p}\right),
\end{align*}
if $d>1$ or $r+s>2$.

Now consider the case $r+s=2$ and $d=1$, i.e. the terms generated by three derivatives acting on a single copy of $S$. If there is more than one off-diagonal entries of $(GF)^{n}G$, then at least one of these has $n\leq m/2$ and we can apply Cauchy--Schwarz to gain a factor of $(N\eta)^{-1/2}$. Thus we need only consider those terms which contain exactly one off-diagonal entry. For concreteness let us consider the term
\begin{align*}
    \frac{1}{N^{2}q}\sum_{i_{1},i_{2}}\bigl((GF)^{m}G\bigr)_{\hat{\imath}_{2},i_{1}}G_{i_{1},i_{1}}G_{\hat{\imath}_{2},\hat{\imath}_{2}}S^{2p-1},
\end{align*}
generated by all three derivatives acting on a single copy of $S$. We replace $G_{i_{1},i_{1}}$ with a quantity independent of $i_{1}$ using \eqref{eq:self1} and then summing over $i_{1}$ using the fluctuation averaging bound
\begin{align*}
    \sum_{i_{1}}\bigl((GF)^{m}G\bigr)_{\hat{\imath}_{2},i_{1}}&\prec\frac{1}{\eta^{m/2+1}}
\end{align*}
from Lemma \ref{lem:fluctuationAveraging}. This procedure is the same as that in the proof of Lemma \ref{lem:isoGFT} and so we omit the details. The end result is the bound
\begin{align*}
    \frac{1}{Nq\eta^{m/2+1}}|S|^{2p-1}&\prec\frac{1}{q\eta}\left(\left(\frac{\mc{E}^{av}_{\eta,q}}{\eta^{m/2-1}}\right)^{2p}+|S|^{2p}\right),
\end{align*}
which concludes the proof.

\end{document}